\documentclass[11pt, twoside]{amsart}

\title[Constructing non-semisimple modular categories]{Constructing non-semisimple modular categories \\with 
local modules}

\author{Robert Laugwitz}
\address{School of Mathematical Sciences,
University of Nottingham, University Park, Nottingham, NG7 2RD, UK}
\email{robert.laugwitz@nottingham.ac.uk}

\author{Chelsea Walton}
\address{Department of Mathematics, Rice University,
P.O. Box 1892, Houston, TX 77005-1892, USA}
\email{notlaw@rice.edu}

\usepackage{mathabx}
\usepackage{import}
\usepackage{amsmath}
\usepackage{amsfonts}
\usepackage{amsthm}
\usepackage{amssymb,bbm}
\usepackage{dsfont}
\usepackage[alphabetic, initials, nobysame]{amsrefs}
\usepackage[english]{babel}
\usepackage{url}
\usepackage{fancyhdr}
\usepackage{graphicx}
\usepackage{verbatim}
\usepackage[normalem]{ulem}
\usepackage[shortlabels]{enumitem}
\newcommand{\stkout}[1]{\ifmmode\text{\sout{\ensuremath{#1}}}\else\sout{#1}\fi}
\usepackage[
colorlinks=true,
linkcolor=black, 
anchorcolor=black,
citecolor=black,
urlcolor=black, 
]{hyperref}
\usepackage[all]{xy}
\usepackage{geometry}
\usepackage{microtype} 
\usepackage[dvipsnames]{xcolor}

\allowdisplaybreaks

\input xy
\xyoption{all}

\definecolor{forest}{rgb}{0.0, 0.5, 0.0}

\newcommand{\bijar}[1][]{%
 \ar[#1]
 \ar@<0.7ex>@{}[#1]|-*=0[@]{\sim}} 

\usepackage{enumitem}

\usepackage{calrsfs}
\DeclareMathAlphabet{\cal}{OMS}{zplm}{m}{n}

\usepackage[justification=centering]{caption}

\newcommand{\leftexpsub}[3]{{\vphantom{#3}}^{#1}_{#2}{#3}}

\newcommand{\lYD}[1]{\leftexpsub{#1}{#1}{\mathsf{YD}}}

\newcommand{\ov}[1]{\overline{#1}}

\newcommand{\bimod}[1]{#1\text{-}\mathsf{Bimod}}
\newcommand{\lmod}[1]{#1\text{-}\mathsf{Mod}}

\newcommand{\rmod}[1]{\mathsf{Mod}\text{-}#1}

\newcommand{\lcomod}[1]{#1\text{-}\mathsf{Comod}}
\newcommand{\rcomod}[1]{\mathsf{Comod}\text{-}#1}

\newcommand{\cha}{\operatorname{char}}
\newcommand{\coev}{\mathsf{coev}}
\newcommand{\coevr}{\widetilde{\mathsf{coev}}}

\newcommand{\Drin}{\operatorname{Drin}}
\newcommand{\ev}{\mathsf{ev}}
\newcommand{\evr}{\widetilde{\mathsf{ev}}}
\newcommand{\End}{\operatorname{End}}

\newcommand{\Hom}{\operatorname{Hom}}

\newcommand{\ide}{\mathsf{Id}}

\newcommand{\isomorph}{\stackrel{\sim}{\to}}
\newcommand{\tensimeq}{\overset{\otimes}{\simeq}}

\newcommand{\Ob}{\operatorname{Ob}}
\newcommand{\one}{\mathds{1}}

\newcommand{\triv}{\mathrm{triv}}

\newcommand{\Alg}{\mathsf{Alg}}

\newcommand{\Bialg}{\mathsf{Bialg}}

\newcommand{\Coalg}{\mathsf{Coalg}}

\newcommand{\FPdim}{\mathsf{FPdim}}
\renewcommand{\dim}{\mathsf{dim}}

\newcommand{\HopfAlg}{\mathsf{HopfAlg}}

\newcommand{\Rep}{\mathsf{Rep}}
\newcommand{\locmod}{\mathsf{Rep}^{\mathsf{loc}}}

\newcommand{\Vect}{\mathsf{Vect}_\Bbbk}

\newcommand{\cA}{\cal{A}}
\newcommand{\cC}{\cal{C}}
\newcommand{\cD}{\cal{D}}
\newcommand{\cB}{\cal{B}}

\newcommand{\cZ}{\cal{Z}}

\newtheoremstyle{defstyle}
  {0.5cm}                   
  {0.5cm}                   
  {\normalfont}           
  {}     
  {\normalfont\bfseries}  
  {:}                     
  {0.3cm}              
  {\thmname{#1}\thmnumber{ #2}\thmnote{ (#3)}}

\numberwithin{equation}{section}

\newtheorem*{rep@theorem}{\rep@title}
\newcommand{\newreptheorem}[2]{%
\newenvironment{rep#1}[1]{%
 \def\rep@title{#2 \ref{##1}}%
 \begin{rep@theorem}}%
 {\end{rep@theorem}}}
\makeatother

\newtheorem{theorem}{Theorem}[section]

\newtheorem{proposition}[theorem]{Proposition}
\newreptheorem{proposition}{Proposition}
\newtheorem{corollary}[theorem]{Corollary}
\newreptheorem{corollary}{Corollary}
\newtheorem{lemma}[theorem]{Lemma}
\newtheorem{conjecture}[theorem]{Conjecture}

\newtheorem{theorem*}{Theorem}
\newreptheorem{theorem}{Theorem}

\theoremstyle{definition}
\newtheorem{definition}[theorem]{Definition}
\newtheorem{notation}[theorem]{Notation}

\newtheorem{example}[theorem]{Example}
\newtheorem{remark}[theorem]{Remark}

\newtheorem{question}[theorem]{Question}

\makeatletter              
\let\c@equation\c@theorem  
\makeatother
\numberwithin{equation}{section}

\makeatletter
\@namedef{subjclassname@2020}{%
  \textup{2020} Mathematics Subject Classification}
\makeatother

\subjclass[2020]{18M20, 18M15}
\keywords{Frobenius algebra, local module, modular tensor category, relative monoidal center}

\begin{document}


\begin{abstract}
We define the class of rigid Frobenius algebras in a (non-semisimple) modular category and prove that their categories of local modules are, again, modular. This generalizes previous work of A.~Kirillov, Jr. and V.~Ostrik [\emph{Adv. Math.} {\bf 171} (2002), no.~2] in the semisimple setup. Examples of non-semisimple modular categories via local modules, as well as connections to the authors' prior work on relative monoidal centers, are provided. In particular, we classify rigid Frobenius algebras in Drinfeld centers of module categories over  group algebras, thus generalizing the classification by A.~Davydov [\emph{J. Algebra} {\bf 323} (2010), no.~5] to arbitrary characteristic. 
\end{abstract}

\maketitle

\vspace{-10pt}

\begin{changemargin}{1.5cm}{1.5cm} 
{\footnotesize \tableofcontents}
\end{changemargin}


\section{Introduction}\label{sec:intro}

The goal of this work is to extend constructions of modular categories in the semisimple setting to the non-semisimple (= not necessarily semisimple) setting. We work over an algebraically closed field $\Bbbk$. Here, we use Kerler-Lyubashenko's notion of a modular category, which does not require semisimplicity: a braided finite tensor category is said to be {\it modular} if it is non-degenerate and ribbon \cite{KL}*{Definition~\ref{def:modular}}. The original notion of a (semisimple) modular category has appeared in a myriad of fields including 
low-dimensional topology \cites{Turaev-1992,TV}, conformal field theory \cites{MoS,Gannon,Hua}, and subfactor theory \cites{KLM-2001,EP}. Moreover, modular categories in the non-semisimple setting have been structures of intense investigation due to their expanding list of applications, such as in non-semisimple topological quantum field theories \cites{KL, DGGPR}, logarithmic conformal field theories \cites{HLZ,Len}, modular functors and mapping class group actions \cites{FSS,LMSS,SW}. Very recently, non-semisimple modular categories have been used in a $3$-dimensional QFT construction with derived categories of quantum groups as line operators \cite{CDGG}, and in defining invariants of $4$-dimensional $2$-handle bodies \cite{BeD}. 
Constructions of non-semisimple modular categories have, for example, been given in \cites{Lusztig,BCGP,CGR,LO,GLO, Shi1, Negron,LW2}. 
We review modular tensor categories and related categorical structures  in Section~\ref{sec:monoidal}. 

\smallskip

In the semisimple setting, one known way of obtaining new modular categories from old ones is to construct a category of local modules \cites{Par,Sch} over certain commutative algebras \cite{KO}. Categories of local modules have applications in rational conformal field theory by describing the extension theory of VOAs \cite{HKL}*{Theorem~3.4} \cite{CKM}. Such categories of local modules have been used extensively in the mathematical physics literature, see e.g. \cites{FRS2,FFRS,DRCR,FL}. Here, for a finite tensor category $(\cC,\otimes)$ with braiding $c$ (e.g., a modular category) and for a commutative algebra $A$ in $\cC$, a right $A$-module in $(V,\; a_V:V \otimes A \to V)$ is said to be {\it local} if $a_V = a_V \; c_{A,V} \; c_{V,A}$  [Definition~\ref{def:locmod}]. The main achievement of the present paper is that we construct non-semisimple modular categories of local modules; see Theorem~\ref{thm:locmodular-intro}. 

\smallskip

Towards our main result, we first reconcile the various algebras $A$ in modular categories $\cC$ that have appeared in the literature for the purpose of building semisimple modular categories of local modules. We employ here in the non-semisimple setting a version of an algebra $A$ that is a certain type of  Frobenius algebra in $\cC$, as defined below. 

\begin{definition}[Definition~\ref{def:rigid-Frob}]
An algebra in a braided finite tensor category  is called a {\it rigid Frobenius algebra} if it is connected, commutative, and special Frobenius. 
\end{definition}

Each of the terms above are discussed in Section~\ref{sec:Frobalg}, and general algebraic structures in braided finite tensor categories are reviewed in Section~\ref{sec:alg-tensor}. When $\cC$ is a semisimple modular tensor category, certain {\it rigid $\cC$-algebras} [Definition~\ref{def:rigidCalgebra}] and {\it connected étale algebras} [Section~\ref{sec:co/alg}] yield modular categories of local modules; see \cite{KO} and \cite{Dav3}, respectively.  We compare  these structures below within ribbon finite tensor categories.

\begin{proposition}[Proposition~\ref{prop:conn-etale}]
\label{prop:conn-etale-intro} Let $\cC$ be a ribbon finite tensor category (e.g., a modular tensor category), and take $A$ an algebra in $\cC$. Then the following statements are equivalent.
\begin{enumerate}[(a),font=\upshape]
    \item $A$ is a rigid $\cC$-algebra with trivial twist \textnormal{(as in \cite{KO}*{Theorem~4.5})}. \smallskip
    \item $A$ is a connected \'{e}tale algebra, with nonzero quantum dimension and trivial twist \textnormal{(as in \cite{Dav3}*{Theorem 2.6.3})}. \smallskip
    \item $A$ is a rigid Frobenius algebra \textnormal{(as in Theorem~\ref{thm:locmodular-intro} below)}.
\end{enumerate}
\end{proposition}

In subfactor theory, the theory of the local modules over an étale algebra was first developed under the name of a {\it Q-system} \cites{BEK1999, BEK2000, BEK2001}. These algebras are also related to {\it condensable algebras} and {\it bosons} in the study of topological phases of matter \cite{Kong}*{Section~2}.

\smallskip 

Next, after giving preliminary results about categories of local modules in Section~\ref{sec:local-mod}, we establish the main result of this article. 

\begin{theorem}[Theorem~\ref{thm:locmodular}]  \label{thm:locmodular-intro} If $\cC$ is a modular tensor category and $A$ is a rigid Frobenius algebra in $\cC$, then the category of local modules over $A$ in $\cC$ is also modular.
\end{theorem}

Its semisimple version was established by Kirillov-Ostrik \cite{KO}*{Theorem~4.5}. The main difference between the proof of that result and the proof of  Theorem~\ref{thm:locmodular-intro} is the verification of non-degeneracy for the category of local modules: we showed that the M\"uger center is trivial, whereas \cite{KO} showed that the S-matrix is invertible (which cannot be generalized to the non-semisimple case). Moreover, ribbonality is achieved in a more direct way here: via pivotality and Proposition~\ref{prop:ribbon-left-right}; see Remark~\ref{rem:thm-vs-ss}.

\smallskip

Section~\ref{sec:local-mod} ends with the computation of the Frobenius--Perron (FP-)dimension of the categories of local modules over a commutative  algebra (e.g., rigid Frobenius algebra) in a braided finite tensor  category (e.g., a modular tensor category), see Corollary~\ref{cor:locmod-FPdim}; this generalizes work of Davydov et al. \cite{DMNO}*{Lemma~3.11 and Corollary~3.32} to the non-semisimple case. In particular, we obtain a bound on the possible FP-dimensions of rigid Frobenius algebras; see Remark~\ref{rem:FPbound}.

\smallskip

Next, in Section~\ref{sec:rel-center}, we build on previous work by the authors  on the construction of non-semisimple modular categories with relative monoidal centers \cites{LW,LW2}. In  Section~\ref{sec:rel-back}, we review  facts about the relative monoidal center, $\cZ_\cB(\cC)$, of a finite tensor category $\cC$ with respect to a braided finite tensor category $\cB$. Such categories are useful as they naturally include representation categories of quantum groups \cites{L18, LW}, and when $\cB$ is the category of $\Bbbk$-vector spaces, $\Vect$, we recover the usual monoidal center $\cZ(\cC)$. We recall conditions from \cite{LW2} that imply when $\cZ_\cB(\cC)$ is a modular tensor category [Theorem~\ref{thm:ZBCmodular}]. As a consequence of Theorem~\ref{thm:locmodular-intro}, we also have a method to construct a modular tensor category of local modules in $\cZ_\cB(\cC)$; see Corollary~\ref{cor:rel-modular1}. On a related note, we generalize Schauenburg's result, \cite{Sch}*{Corollary~4.5} (see Theorem~\ref{thm:local-center}), on monoidal centers to the relative setting as follows.

\begin{theorem}[Theorem~\ref{thm:ZRepA}]
Take a  commutative algebra $A$ in a relative monoidal center $\cZ_\cB(\cC)$. Then, the category of local modules over $A$ in $\cZ_\cB(\cC)$ is braided equivalent to the relative monoidal center of the category of right $A$-modules in $\cC$.
\end{theorem}

Using this result, along with Corollary~\ref{cor:rel-modular1}, we obtain the following result.

\begin{corollary}[Corollary~\ref{cor:local-mod}] Assume the hypotheses of Theorem~\ref{thm:ZBCmodular} for which the  relative monoidal center $\cZ_\cB(\cC)$ is modular. Then if $A$ is a rigid Frobenius algebra $\cZ_\cB(\cC)$, we obtain that the relative monoidal center of the category of right $A$-modules in $\cC$ is also modular.
\end{corollary}

Next, in Section~\ref{sec:YDB}, we compute the category of local modules over a Hopf algebra $H$ in a category of Yetter-Drinfeld modules $\lYD{H}(\cB)$, for $\cB$ a braided finite tensor category. The outcome is the following statement.

\begin{proposition}[Proposition~\ref{prop:local-triv}]
Given a braided finite tensor category $\cB$ with a Hopf algebra $H$ in $\cB$, then the category of local modules over $H$ in $\lYD{H}(\cB)$ is braided equivalent to $\cB$.
\end{proposition}

\smallskip

In Section~\ref{sec:examples}, we close the paper by studying many examples of the results above, the first set of examples involving commutative Hopf algebras in positive characteristic.

\begin{proposition}[Proposition~\ref{prop:ex-charp}] \label{prop:ex-charp-intro}
Assume that $\textnormal{char}(\Bbbk) = p>0$.
Let $K$ be a finite-dimensional, commutative, non-semisimple Hopf algebra over $\Bbbk$ such that its Drinfeld double is ribbon. Let $L= \Bbbk N$, for $N$ a finite abelian group with $p \nmid  |N|$. Let $H$ be the Hopf algebra $K \otimes L$ over $\Bbbk$. Then, $L$ is a rigid Frobenius algebra in the non-semisimple modular tensor category $\cZ(\lmod{H})$, and the category of local modules over $L$ in $\cZ(\lmod{H})$ is equivalent to  $\cZ(\lmod{K})$ as modular  categories.
\end{proposition}

In particular, one could take $K = \Bbbk G$, for $G$ a finite abelian group with $ |G|$ divisible by $p$, or take $K = u^{[p]}(\mathfrak{g})$, the restricted enveloping algebra of an abelian restricted Lie algebra $\mathfrak{g}$ [Example~\ref{ex:charp}]. We inquire in Question~\ref{ques:charp} ways in which Proposition~\ref{prop:ex-charp-intro} can be generalized.

\smallskip 
For the next set of examples, we revise Davydov's classification of rigid Frobenius algebras in $\cZ(\lmod{\Bbbk G})$ \cite{Dav3}*{Theorem~3.5.1}, by way of Proposition~\ref{prop:conn-etale-intro}, and generalize the result to the non-semisimple setup.

\begin{theorem}[Theorem~\ref{thm:davclassification}] \label{thm:davclassification-intro}
For $\Bbbk$ a field of arbitrary characteristic, the rigid Frobenius algebras in $\cZ(\lmod{\Bbbk G})$ are explicitly classified in terms of group-theoretic data.
\end{theorem}

In particular, in Theorem~\ref{thm:davclassification-intro}  we provide a correction to the product formulas of the algebras in \cite{Dav3}*{Theorem~3.5.1}, as their version of relations do not form an ideal [Remark~\ref{rem:correction}]. Moreover, our revised multiplication and unit formulas are derived from a lax monoidal functor discussed in Remark~\ref{rem:laxmon}.
Results about the corresponding modular tensor categories of local modules are provided in  Corollary~\ref{cor:localmodules}, including their Frobenius--Perron dimension, and a ribbon equivalence to $\cZ(\lmod{\Bbbk H})$, for $H$ a subgroup of $G$ in special cases. We inquire if there is such a ribbon equivalence in the general case in Question~\ref{ques:Davrib}.

\smallskip

Our last set of examples in this work are conjectural, and are related to the notion of a {\it completely anisotropic category} and to the notion of {\it Witt equivalence} of non-degenerate braided finite tensor categories. The former pertains to a non-degenerate braided finite tensor category containing only the unit object as a rigid Frobenius algebra, and the latter is an important tool for classifying such categories. See Section~\ref{sec:ques} and \cite{DMNO} for more details. These notions have an application to  anyon condensation and boundary conditions for topological phases of matter (in the semisimple setting); see, e.g., \cite{Kong}*{Section~7} \cite{FSV}. Now consider the representation category of the small quantum group, $u_q(\mathfrak{sl}_2)$, and the conjecture below.

\begin{conjecture}[Conjecture~\ref{conj:uqsl2}]
Assume that $\Bbbk$ has characteristic 0. The non-semisimple modular tensor category $\cC:=\lmod{u_q(\mathfrak{sl}_2)}$, for $q$ an odd root of unity,   is completely anisotropic in the sense that the only rigid Frobenius algebra in $\cC$ is the unit object.
\end{conjecture}

On a related note, we inquire in Question~\ref{ques:Witt}, which (modular) non-degenerate  braided finite tensor categories are Witt equivalent to $\lmod{u_q(\mathfrak{sl}_2)}$. This could have physical applications in the non-semisimple setting via the references mentioned above. Moreover, after discussions with Victor Ostrik and Christoph Schweigert, we end with general questions about the completely anisotropic property and Witt equivalence for non-semisimple modular tensor categories in characteristic 0; see Question~\ref{ques:nonsem-rigidFrob}.


\section{Preliminaries on Monoidal Categories}\label{sec:monoidal}
 In this section, we review terminology pertaining to monoidal categories. We refer the reader to \cite{BK}, \cite{EGNO}, and \cite{TV} for general information. We recall monoidal categories  [Section~\ref{sec:monoidal1}], various types of rigid categories [Section~\ref{sec:rigid}], finite tensor categories [Section~\ref{sec:finitetens}], various braided monoidal categories [Section~\ref{sec:braided}], ribbon monoidal categories [Section~\ref{sec:ribbon}], and modular tensor categories in the non-semisimple setting [Section~\ref{sec:modular}].
 
 \smallskip
 
We assume that all categories here are {\it locally small} (i.e., the collection of morphisms between any two objects is a set),  and are abelian. A full subcategory of a category is called \emph{topologizing} if it is closed under finite direct sums and subquotients \cite{Ros}*{Section~3.5.3}, \cite{Shi1}*{Definition~4.3}. 
 Given a functor $F\colon \cC\to \cD$ between two categories $\cC$ and $\cD$, the \emph{full image} of $F$ is the full subcategory of $\cD$ on all objects isomorphic to an object of the form $F(C)$ for $C$ in~$\cC$. 


\subsection{Monoidal categories and monoidal functors}
\label{sec:monoidal1}
We refer the reader to \cite{EGNO}*{Sections~2.1--2.6} and \cite{TV}*{Sections~1.1--1.4} for further details.

\smallskip

A {\it monoidal category} consists of a category $\cC$ equipped with a bifunctor $\otimes\colon  \cC \times \cC \to \cC$, a natural isomorphism $\alpha_{X,Y,Z}\colon  (X \otimes Y) \otimes Z \overset{\sim}{\to} X \otimes (Y \otimes Z)$ for each $X,Y,Z \in \cC$, an object $\one \in \cC$, and natural isomorphisms $l_X\colon  \one \otimes X \overset{\sim}{\to} X$ and $r_X\colon  X  \otimes \one \overset{\sim}{\to} X$ for each $X \in \cC$, such that the pentagon and triangle axioms hold. By MacLane's coherence theorem, we will assume that all monoidal categories are {\it strict} in the sense that $(X \otimes Y) \otimes Z = X \otimes (Y \otimes Z)$ and $\one \otimes X  = X = X \otimes \one$, for all $X, Y, Z \in \cC$; that is, $\alpha_{X,Y,Z},\; l_X,\; r_X$ are identity maps. 

\smallskip

A {\it (strong) monoidal functor} $(F, F_{-,-}, F_0)$ between monoidal categories $(\cC, \otimes_\cC, \one_\cC)$ to $(\cD, \otimes_\cD, \one_\cD)$ is a functor $F\colon  \cC \to \cD$ equipped with a natural isomorphism $F_{X,Y}\colon  F(X) \otimes_\cD F(Y) \isomorph F(X \otimes_\cC Y)$ for all $X,Y \in \cC$, and an isomorphism $F_0\colon  \one_\cD \isomorph F(\one_\cC)$ in $\cD$, that satisfy associativity and unitality constraints. An {\it equivalence of  monoidal categories} is provided by a  monoidal functor between the two monoidal categories that yields an equivalence of the underlying categories.


\subsection{Rigid and pivotal monoidal categories, and quantum dimension} \label{sec:rigid}
We refer  to \cite{EGNO}*{Sections~2.10 and~4.7} and \cite{TV}*{Sections~1.5--1.7} for further details of the items discussed below.

\smallskip

A monoidal category $(\cC, \otimes, \one)$ is {\it rigid} if it comes equipped with left and right dual objects,  i.e., for each $X \in \cC$ there exist, respectively, an object $X^* \in \cC$ with co/evaluation maps $\ev_X\colon  X^* \otimes X \to \one$ and $\coev_X\colon  \one \to  X \otimes X^*$, as well as an object ${}^*X \in \cC$ with co/evaluation maps $\evr_X\colon  X \otimes {}^*X \to \one$, $\coevr_X\colon \one \to  {}^*X \otimes X$,  satisfying the usual coherence conditions of left and right duals. 

\smallskip

 A rigid monoidal category is {\it pivotal} if it is equipped with isomorphisms $j_X\colon  X \overset{\sim}{\to} X^{**}$ natural in $X$ and satisfying $j_{X \otimes Y} = j_X \otimes j_Y$ for all $X,Y \in \cC$. In this case,  a right duality can be defined on the left duals via ${}^*X = X^*$ and
\begin{align}\label{rightdual-pivotal}
\coevr_X := (\ide_{X^*}\otimes j_X^{-1})\coev_{X^*} \quad \text{and} \quad
\evr_X := \ev_{X^*}(j_X\otimes \ide_{X^*}).
\end{align}
Given a pivotal category, we will use this right dual structure unless specified otherwise. Moreover, a pivotal category has the following characterization, which will be of use later.

\begin{proposition}\cite{TV}*{Remark~1.7.4} \label{prop:piv-equiv}
A rigid monoidal category $\cC$ is pivotal if and only if the left and right duality functors of $\cC$ coincide in the following sense: There exists a left and right duals on the same objects, ${}^* X = X^*$ for all $X \in \cC$, and
\begin{enumerate}[(a),font=\upshape]
    \item [(i)] For all morphisms $f: X \to Y$ in $\cC$, we get that 
    $$(\ev_Y \otimes \ide_{X^*})(\ide_{Y^*} \otimes f \otimes \ide_{X^*})(\ide_{Y^*} \otimes \coev_X) = (\ide_{X^*} \otimes \evr_Y)(\ide_{X^*} \otimes f \otimes \ide_{Y^*})(\coevr_X \otimes \ide_{Y^*}).$$
    \item [(ii)] $\coev_\one = \coevr_\one$. \smallskip
    \item [(iii)] For all $X,Y \in \cC$, we get that
    \[
    \begin{array}{l}
    \smallskip
    (\ev_X \otimes \ide_{(Y \otimes X)^*})(\ide_{X^*} \otimes \ev_Y \otimes \ide_{X \otimes (Y \otimes X)^*})(\ide_{X^* \otimes Y^*} \otimes \coev_{Y \otimes X})\\
    = (\ide_{(Y \otimes X)^*} \otimes \evr_Y)(\ide_{(Y \otimes X)^* \otimes Y} \otimes \evr_X \otimes \ide_{Y^*})(\coev_{Y \otimes X} \otimes \ide_{X^* \otimes Y^*}).  \qed     
    \end{array}
    \]
\end{enumerate}
\end{proposition}

\smallskip

The {\it quantum dimension} of an object $X$ of a pivotal (rigid) monoidal category $(\cC, \otimes, \one, j)$  is defined to be 
\begin{align}\label{qdim-j}
\dim_j(X) = {\sf ev}_{X^*}  (j_X \otimes {\sf Id}_{X^*})  {\sf coev}_X = \evr_X \; \coev_X \in \End_\cC(\one).
\end{align}


\subsection{Finite tensor categories and Frobenius--Perron dimension} \label{sec:finitetens}
Recall that $\Bbbk$ is an algebraically closed field. We now discuss certain $\Bbbk$-linear monoidal categories following the terminologies of \cite{EGNO}*{Sections~1.8,~4.2,~4.5,~6.1,~7.1--7.3,~7.9}.

\smallskip

A $\Bbbk$-linear abelian category $\cC$ is {\it locally finite} if, for any two objects $V,W$ in $\cC$, $\Hom_{\cC}(V,W)$ is a finite-dimensional $\Bbbk$-vector space and every object has a finite filtration by simple objects. Moreover, we say that $\cC$ is {\it finite} if it is locally finite, if every simple object has a projective cover, and there are finitely many isomorphism classes of simple objects. Equivalently, $\cC$ is finite if it is equivalent to the category of finite-dimensional modules over a finite-dimensional $\Bbbk$-algebra.
A  {\it tensor category} is a locally finite, rigid, monoidal category $(\cC, \otimes, \one)$ such that $\otimes$ is $\Bbbk$-linear in each slot and $\one$ is a simple object of $\cC$ (i.e., $\Hom_\cC(\one,\one) \cong \Bbbk$).
 A {\it tensor functor} is a $\Bbbk$-linear, exact, faithful,  monoidal functor $F$ between  tensor categories $\cC$ and $\cD$, with $F(\one_\cC) = \one_\cD$. 
 
\smallskip

An example of a finite tensor category is $\Vect$, the category of finite-dimensional $\Bbbk$-vector spaces. More generally, the category $\lmod{H}$ of finite-dimensional $\Bbbk$-modules over a (finite-dimensional) Hopf algebra $H$ is a (finite) tensor category.

\smallskip

We will use the following tensor product of finite tensor categories. The {\it Deligne tensor product} of two finite abelian categories is the abelian category $\cC \boxtimes \cD$  equipped with a bifunctor $\boxtimes\colon  \cC \times \cD \to \cC \boxtimes \cD$, $(X,Y) \mapsto X \boxtimes Y$, right exact in both variables so that for any abelian category $\cA$ and any bifunctor $F\colon  \cC \times \cD \to \cA$ right exact in both slots, there exists a unique right exact functor $\overline{F}\colon  \cC \boxtimes \cD \to \cA$ with $\overline{F} \circ \boxtimes = F$ \cite{Del}*{Section 5}. Note that the simple objects in $\cC\boxtimes \cD$, up to isomorphism, of the form $S\boxtimes T$, where $S$ and $T$ are simple objects in $\cC$, respectively, $\cD$ since $\Bbbk$ is algebraically closed and hence a perfect field \cite{Del}*{5.9}. Since $\cC$ and $\cD$ are equivalent to module categories over a finite-dimensional $\Bbbk$-algebra, all indecomposable projective modules in $\cC\boxtimes \cD$ are isomorphic to projective objects of the form $P\boxtimes Q$, for $P,Q$ projective objects in $\cC$, respectively, $\cD$.
We have that $\cC\boxtimes \cD$ is monoidal  when both $\cC$ and $\cD$ are monoidal such that the tensor products of $\cC$ and $\cD$ are $\Bbbk$-linear and right exact; this is via 
\begin{equation} \label{eq:Deligne-monoidal}
    (X \boxtimes Y) \otimes^{\cC \boxtimes \cD} (X' \boxtimes Y'):=(X \otimes^\cC X') \boxtimes (Y \otimes^\cD Y'),
\end{equation}
for all $X,X' \in \cC$ and $Y,Y' \in \cD$,
and with the unit object $\one_\cC \boxtimes \one_\cD$.
As  $\cC$, $\cD$ are finite tensor categories, then so is $\cC\boxtimes \cD$. Given two tensor functors $F\colon \cC\to\cD$ and $F'\colon \cC'\to \cD'$ between finite tensor categories, there exists an induced tensor functor $F\boxtimes F'\colon \cC\boxtimes\cC'\to \cD\boxtimes \cD'$.

\smallskip

For example, in the case when $\cC = \lmod{K}$ and $\cD = \lmod{L}$, for $K$ and $L$ finite-dimensional Hopf algebras over $\Bbbk$, we get that
\begin{equation} \label{eq:boxtimes}
    \lmod{(K \otimes_{\Bbbk} L)} \simeq \lmod{K} \;  \boxtimes\; \lmod{L}.
\end{equation}

Next, we turn our attention to the Frobenius--Perron dimension of categories and objects within them. Let $\{X_i\}_{i \in I}$ be a collection of isomorphism class representatives of simple objects in a finite tensor category $\cC$. For an object $X \in \cC$, define the matrix $N_X$ by
$$(N_X)_{i,j} := ([X \otimes X_i : X_j])_{i,j}, \quad \text{for } i,j \in I,$$
where $[X \otimes X_i : X_j]$ is the multiplicity of $X_j$ in a Jordan-H\"{o}lder series of $X \otimes X_i$ in $\cC$. The largest nonnegative real eigenvalue of $N_X$ (which exists by the Frobenius--Perron Theorem) is called the {\it Frobenius--Perron (FP-)dimension of $X$}, denoted by $\FPdim_{\cC}(X)$. The {\it Frobenius--Perron (FP-)dimension of $\cC$} is defined as
$$\FPdim(\cC):=\textstyle \sum_{i \in I} \FPdim(X_i) \cdot \FPdim(P_i),$$
where $P_i$ is the projective cover of $X_i$. For example, $\FPdim(\lmod{H})= \dim_\Bbbk H$ for any finite-dimensional Hopf algebra $H$.

\smallskip

We will later use the fact that for finite tensor categories  $\cC$, $\cD$,
\begin{equation}\label{eq:FPdim-Deligne}
\FPdim(\cC\boxtimes \cD)=\FPdim(\cC) \cdot \FPdim(\cD).    
\end{equation}
This follows since all simple (projective) objects in $\cC\boxtimes \cD$ are isomorphic to $S\boxtimes T$, for $S$, $T$ simple (respectively, projective) objects in $\cC$, respectively, $\cD$.


\subsection{Braided tensor categories, the monoidal center \texorpdfstring{$\cZ(\cC)$}{Z(C)}, and the M\"uger center~\texorpdfstring{$\cC'$}{C'}} 
\label{sec:braided}
Here, we discuss braided tensor categories and related constructions, and refer the reader to \cite{BK}*{Chapter~1}, \cite{EGNO}*{Sections~8.1--8.3, 8.5,~8.20}, and \cite{TV}*{Sections~3.1 and~5.1} for more information.

\smallskip

A {\it braided tensor category} $(\cC, \otimes, \one, c)$ is a tensor category equipped with a natural isomorphism $c_{X,Y}\colon  X \otimes Y \overset{\sim}{\to} Y \otimes X$ for each $X,Y \in \cC$ such that the hexagon axioms hold. We also have a {\it mirror braiding} on $\cC$ given by 
$$c_{Y,X}^{-1}\colon X \otimes Y \overset{\sim}{\to} Y \otimes X$$
for $X,Y \in \cC$. 
We refer to the braided tensor category
$$\overline{\cC}:=(\cC, \otimes, \one, c^{-1})$$
as the {\it mirror} of $(\cC, \otimes, \one, c)$. 
By a {\it braided tensor subcategory} of a braided tensor category $\cC$ we mean a subcategory of $\cC$ containing the unit object of $\cC$, closed under the tensor product of $\cC$, and containing the braiding isomorphisms. 
A {\it braided tensor functor} between braided tensor categories $\cC$ and $\cD$ is a tensor  functor $(F, F_{-,-}, F_0)\colon  \cC \to \cD$ so that $F_{Y,X} \; c^{\cD}_{F(X),F(Y)} = F(c^{\hspace{.01in}\cC}_{X,Y}) \; F_{X,Y}$ for all $X,Y \in \cC$. An {\it equivalence of  braided tensor categories} is a braided tensor functor between the two tensor categories that yields an equivalence of the underlying categories.

\smallskip 

An important example of a braided tensor category is the {\it monoidal center} (or {\it Drinfeld center}) $\cZ(\cC)$ of a tensor category $(\cC, \otimes, \one)$: its objects are pairs $(V, c_{V,-})$ where $V$ is an object of $\cC$ and $c_{V,X}\colon  V \otimes X \overset{\sim}{\to} X \otimes V$ is a natural isomorphism (called a  {\it half-braiding}) satisfying $$c_{V,X \otimes Y} = (\ide_{X} \otimes c_{V,Y})(c_{V,X} \otimes \ide_{Y}),$$
and its morphisms $(V, c_{V,-}) \to (W, c_{W,-})$ are given by
 morphisms $f \in \Hom_{\cC}(V,W)$ such that $(\ide_X \otimes f)c_{V,X} = c_{W,X}(f \otimes \ide_X)$ for all $X \in \cC$. 
The monoidal product is $(V,c_{V,-}) \otimes (W,c_{W,-}):=(V\otimes W, c_{V \otimes W,-})$, for $c_{V\otimes W, X}:=(c_{V,X} \otimes \ide_W)(\ide_V \otimes c_{W,X})$ for all $X \in \cC$.
An important feature of $\cZ(\cC)$ is the braiding defined by $$c_{(V,c_{V,-}),(W,c_{W,-})}:=c_{V,W}.$$
If $\cC$ is a (finite) tensor category, then $\cZ(\cC)$ is a braided (finite) tensor category. Moreover,
\begin{equation} \label{eq:FP-ZC}
\FPdim(\cZ(\cC)) = \FPdim(\cC)^2.
\end{equation}

For example,  when $\cC = \lmod{H}$ for $H$ a finite-dimensional Hopf algebra over $\Bbbk$, we get that
\begin{equation} \label{eq:Drin}
   \cZ(\lmod{H})\; \simeq \;\lmod{\Drin(H)} \;\simeq \;\lYD{H},
\end{equation}
where $\Drin(H)$ is the Drinfeld double of $H$, and $\lYD{H}$ is the category of left Yetter-Drinfeld modules over $H$.

Note that if $(\cC,c)$ is a braided monoidal category, there is a functor of braided monoidal categories
\begin{equation} \label{eq:Z(C)object+}
\cC\to \cZ(\cC), \quad V\longmapsto V^+:=(V,c_{V,-}).
\end{equation}
We will denote the image of a morphism $f\colon V\to W$ in $\cC$ under this functor by 
\begin{equation} \label{eq:Z(C)morphism+}
    f^+\colon V^+ \longrightarrow W^+.
\end{equation}

Given two braided finite tensor  categories $(\cC,\otimes^\cC,\one_\cC,c^\cC)$ and $(\cD,\otimes^\cD,\one_\cD,c^\cD)$, the Deligne tensor product $\cC\boxtimes \cD$  is a braided finite tensor category. The braiding is obtained from 
\begin{equation} \label{eq:Deligne-braiding}
    c_{X\boxtimes Y,X'\boxtimes Y'}=c^\cC_{X,X'}\boxtimes c^\cD_{Y,Y'}\colon (X\otimes^\cC X')\boxtimes (Y\otimes^\cD Y')\to (X'\otimes^\cC X)\boxtimes (Y'\otimes^\cD Y)
\end{equation}
for all $X,X'\in \cC$, $Y,Y'\in\cD$.

\begin{lemma}\label{lem:Z-Deligne}
Given two finite tensor categories $\cC$ and $\cD$, the canonical functor 
$$\cZ(\cC)\boxtimes \cZ(\cD)\to \cZ(\cC\boxtimes\cD), \quad (V,c_{V,-})\boxtimes (W,d_{W,-})\mapsto (V\boxtimes W,c_{V,-}\boxtimes d_{W,-})$$
is an equivalence of braided tensor categories.
\end{lemma}

\begin{proof}
One checks that the given functor is indeed a functor of braided tensor categories. Further, it is fully faithful. Now, by \cite{EGNO}*{Proposition~6.3.3.} it is an equivalence of categories since 
$$\FPdim(\cZ(\cC)\boxtimes \cZ(\cD))=\FPdim(\cC)^2\cdot \FPdim(\cD)^2 =\FPdim(\cZ(\cC\boxtimes \cD)),$$
using \eqref{eq:FPdim-Deligne} and \eqref{eq:FP-ZC}.
\end{proof}

We also need  to consider later the {\it M\"uger center}  of a braided tensor category $(\cC, \otimes, \one, c)$, which is the full subcategory on the objects
\begin{equation} \label{eq:Mugercenter}
\Ob (\cC') := \{ X \in \cC ~|~ c_{Y,X}\; c_{X,Y} = {\sf Id}_{X \otimes Y} \text{ for all } Y \in \cC \}.
\end{equation}


\subsection{Ribbon tensor categories}
\label{sec:ribbon}
In this section we assume that $\cC := (\cC, \otimes, \one, c)$ is a braided tensor category, and we refer the reader to \cite{BK}*{Chapter~2}, \cite{EGNO}*{Sections~8.9--8.11},  \cite{TV}*{Section~3.3}, and \cite{Rad}*{Chapter~12} for details of the discussion below.

\smallskip

A braided tensor category $(\cC, \otimes, \one, c)$  is {\it ribbon} (or \emph{tortile}) if  it is equipped with a natural isomorphism $\theta_X\colon X \overset{\sim}{\to} X$  (a {\it twist}) satisfying $\theta_{X \otimes Y} = (\theta_X \otimes \theta_Y)   c_{Y,X}  c_{X,Y}$, $\theta_\one = \ide_\one$, 
and $(\theta_X)^* = \theta_{X^*}$ for all $X,Y \in \cC$. A \emph{functor (or, equivalence) of ribbon categories} is a functor (respectively, equivalence) $F\colon \cC\to \cD$ of braided  tensor  categories such that $F(\theta_X^{\cC})=\theta^{\cD}_{F(X)}$, for any $X\in \cC$, cf. \cite{Shum}*{Section~1}.

\smallskip

In a ribbon category $(\cC,\otimes,\one,c,\theta)$, consider the {\it Drinfeld isomorphism}:
\begin{align}\label{Driniso}
\phi_X=(\ide_{X^{**}}\otimes \ev_X)(c_{X^*,X^{**}}\otimes \ide_{X})(\coev_{X^*}\otimes\ide_X)\colon X \overset{\sim}{\to} X^{**}.
\end{align}
Then,
\begin{align}\label{eq:ribbonpivotal} 
j_X:=\phi_X\theta_X\colon X\isomorph X^{**}
\end{align}
defines a pivotal structure on $\cC$. Now by \eqref{rightdual-pivotal}, a right duality can be defined from a left duality in a ribbon category via ${}^*X = X^*$, and
\begin{align}\label{rightdual}
\coevr_X := 
c_{X^*,X}^{-1}(\theta_X^{-1} \otimes \ide_{X^*})\coev_X\quad \text{and} \quad
\evr_X := 
\ev_X(\ide_{X^*} \otimes \theta_X) c_{X,X^*}
.
\end{align}


\smallskip

The following lemma reconciles the definition of ribbonality in \cite{EGNO} and in \cite{TV}.

\begin{proposition} \label{prop:ribbon-left-right}
A ribbon category $\cC$ is, equivalently, a braided pivotal category such that the left and right twists coincide: That is,   $\theta^l_X=\theta^r_X$ for any object $X$ of $\cC$, where
\begin{align*}
    \theta_X^l&:=(\ev_X\otimes\ide_X)(\ide_{X^*}\otimes c_{X,X})(\coevr_{X}\otimes\ide_X), &\theta_X^r&:=(\ide_X\otimes\evr_X)(c_{X,X}\otimes \ide_{X^*})(\ide_X\otimes \coev_{X}).
\end{align*}
\end{proposition}

\begin{proof}

 Assume that $\cC$ is a ribbon category with twist $\theta$. Then, $\cC$ is pivotal with $j_X$ defined in~\eqref{eq:ribbonpivotal}. Then we have the following calculations:
 \[
 \begin{array}{rl}
 \medskip
 \theta_X^l &= (\ev_X \otimes \theta_X^{-1})(\ide_{X^*} \otimes c_{X,X}) (c_{X^*,X}^{-1} \; \coev_X \otimes \ide_X)\\ \medskip
 &= (\theta_X^{-1} \otimes \ev_X) (c_{X,X^*}^{-1}\; c_{X^*,X}^{-1}\; \coev_X \otimes \ide_X)\\ \medskip
 \medskip
 &= (\theta_X^{-1} \otimes \ev_X) (c_{X,X^*}^{-1}\; c_{X^*,X}^{-1} \; \theta_{X \otimes X^*} \; \coev_X \otimes \ide_X)\\  \medskip
  &= (\theta_X^{-1} \otimes \ev_X) (\theta_{X \otimes X^*}\; c_{X,X^*}^{-1}\; c_{X^*,X}^{-1}  \; \coev_X \otimes \ide_X)\\  \medskip
 &= (\theta_X^{-1} \otimes \ev_X) ((\theta_X \otimes \theta_{X^*})\; \coev_X \; \otimes \ide_X)\\  \medskip
 &= \theta_X \\  \medskip
 &= (\ide_X \otimes \ev_X c_{X,X^*})(\ide_X \otimes \theta_X  \otimes \ide_{X^*})(c_{X,X^*} \otimes \ide_X)(\ide_X \otimes \coev_X)\\  \medskip
 &= \theta_X^r.
 \end{array}
 \]
 The first equation holds by \eqref{rightdual} and naturality of the braiding, and the second equation also by naturality of the braiding. The next step holds by the naturality of $\theta$ applied to $\coev_X$ with $\theta_\one = \ide_\one$: $\theta_{X \otimes X^*}\; \coev_X = \coev_X$. The following equation holds again by the naturality of $\theta$ applied to $c_{X,X^*}^{-1}\; c_{X^*,X}^{-1}$. The fifth equation holds as $\theta_{X \otimes X^*} = (\theta_X \otimes \theta_{X^*})  \; c_{X^*,X} \; c_{X,X^*}$. The next step follows from rigidity and the assumption that $\theta_{X^*} = (\theta_X)^*$. The last steps hold by naturality and by definition.
  
 For the converse direction, one checks that $\theta_X:=\theta_X^l=\theta_X^r$ satisfies the axioms of a twist; see \cite{TV}*{Exercise~3.3.3}.
\end{proof}



\subsection{Non-degeneracy and modular tensor categories}
\label{sec:modular}
In the section, we discuss a notion of a modular tensor category for the non-semisimple setting. This is based on work of Kerler--Lyubashenko \cite{KL}, \cite{DGNO},  and recent work of Shimizu \cite{Shi1}. To proceed, we adopt the definition of non-degeneracy below, which extends the notion of non-degeneracy in the semisimple setting; see \cite{EGNO}*{Definition~8.19.2 and Theorem~8.20.7}.

\begin{definition}[{\cite{Shi1}*{Theorem~1}}] 
\label{def:nondeg} 
We call a braided finite tensor category $(\cC, \otimes, \one, c)$ {\it non-degenerate}  if its M\"uger center $\cC'$ (see \eqref{eq:Mugercenter}) is equal to $\Vect$. 
\end{definition}

The lemma below will also be of use.

\begin{lemma}\label{lem:DeligneMueger}
Let $\cC$ and $\cD$ be braided finite tensor categories. If $\cC\boxtimes \cD$ is non-degenerate, then so are both $\cC$ and $\cD$.
\end{lemma}

\begin{proof}
 If $\cC$ fails to be non-degenerate, then there exists an object $X \in \cC'$ that is not a finite direct sum of copies of $\one_\cC$. Since $X \boxtimes \one_{\cD}$ belongs to $(\cC \boxtimes \cD)'$, we then obtain that $\cC \boxtimes \cD$  fails to be non-degenerate. 
\end{proof}

Next we discuss a characterization of non-degeneracy. 
Let $(\cC, \otimes, \one, c_{X,Y}\colon X \otimes Y \overset{\sim}{\to} Y \otimes X)$ be a braided tensor category and $\ov{\cC}$ its mirror from Section~\ref{sec:braided}.  
The assignments $\cC \to \cZ(\cC)$, $X \mapsto (X, c_{X,-})$,  and $\overline{\cC} \to \cZ(\cC)$, $X \mapsto (X, c^{-1}_{-,X})$, extend to a braided tensor functor $\cC \boxtimes \overline{\cC} \to \cZ(\cC)$. If this functor yields an equivalence between the braided tensor categories $\cC \boxtimes \overline{\cC}$ and $\cZ(\cC)$, then we say that $(\cC, \otimes, \one, c)$  is {\it factorizable}.
A braided finite tensor category is non-degenerate if and only if it is factorizable \cite{Shi1}*{Theorem~4.2}.

\smallskip

Moreover, the following type of tensor  categories are of primary interest in this work.

\begin{definition}[{\cite{KL}*{Definition~5.2.7}, \cite{Shi1}*{Section 1}}] \label{def:modular}
A braided finite tensor category is called {\it modular} if it is non-degenerate and ribbon.
\end{definition}

Now consider the braided finite tensor category $\lmod{H}$ for $H$ a finite-dimensional, quasi-triangular Hopf algebra over $\Bbbk$. We get that $\lmod{H}$ is modular precisely when $H$ is ribbon and factorizable \cite{EGNO}*{Proposition~8.11.2 and Example~8.6.4}.
Moreover, it is straight-forward to show that if $\cC$ and $\cD$ are modular, then so is $\cC \boxtimes \cD$ via the monoidal structure~\eqref{eq:Deligne-monoidal}, the braiding~\eqref{eq:Deligne-braiding}, and with ribbon structure $\theta^{\cC \boxtimes \cD} := \theta^\cC \boxtimes \theta^\cD$.


\section{Preliminaries on algebraic structures in braided finite tensor categories}\label{sec:alg-tensor}
In this section, let $\cC:=(\cC, \otimes, \one, c)$ be a braided finite tensor category over $\Bbbk$. Assume that all structures below are $\Bbbk$-linear as well. We recall facts about algebras, coalgebras, and their (co)modules in $\cC$ in Section~\ref{sec:co/alg}, and recall facts about bialgebras, Hopf algebras, and their (co)modules in $\cC$ in Section~\ref{sec:Hopfalg}. In Section~\ref{sec:Frobalg}, we then discuss various types of Frobenius algebras in $\cC$ whose representation categories are used to obtain modular tensor categories later in this work and in other parts of the literature.


\subsection{Algebras, coalgebras, and their (co)modules} \label{sec:co/alg}
 We discuss in this part algebras and coalgebras in~$\cC$ and their (co)modules. More information is available in \cite{EGNO}*{Section~7.8} and \cite{TV}*{Section~6.1}. See also \cite{DMNO} and \cite{FFRS} for terminology.

\smallskip

An {\it algebra} in $\cC$ is an object $A \in \cC$ equipped with two morphisms $m\colon  A \otimes A \to A$ (multiplication) and $u\colon  \one \to A$ (unit) satisfying $m(m\otimes \ide_A) = m(\ide_A \otimes m)$ and $m(u \otimes \ide_A) = \ide_A = m(\ide_A \otimes  u)$. We denote by $\Alg(\cC)$ the category of algebras in $\cC$, where morphisms in $\Alg(\cC)$ are morphisms $f\colon  A \to A'$ in $\cC$ so that $f \; m_A = m_{A'} (f \otimes f)$ and $f \; u_A = u_{A'}$. 

\smallskip

For monoidal categories $\cC$ and $\cD$, with $A \in \Alg(\cC)$ and $B \in \Alg(\cD)$, we get that $A \boxtimes B \in \Alg(\cC \boxtimes \cD)$. Here, we naturally identify $(A \boxtimes B) \otimes (A \boxtimes B)$ with $(A \otimes A) \boxtimes (B \otimes B)$ (see \cite{EGNO}*{Proposition~4.6.1}) and get that $m_{A\boxtimes B} = m_A \boxtimes m_B$ and $u_{A\boxtimes B} = u_A \boxtimes u_B$.

\smallskip

 We say that an algebra $(A,m,u)$ in $\cC$ is  {\it commutative} if $m = m c_{A,A}$. Note that if $A \in \Alg(\cC)$ and $B \in \Alg(\cD)$ are commutative, then $A \boxtimes B \in \Alg(\cC \boxtimes \cD)$ is commutative as well.

\smallskip

The following result is well-known.

\begin{proposition} \label{prop:mon-alg}
 Suppose $\cC$ and $\cD$ are (resp., braided) monoidal categories, and $F: \cC \to \cD$ is a (resp., braided) monoidal functor. If $A$ is a (resp., commutative) algebra in $\cC$, then $F(A)$ is a (resp., commutative) algebra in $\cD$ with $m_{F(A)} = F(m_A) F_{A,A}$ and $u_{F(A)} = F(u_A)F_0$. \qed
\end{proposition}

Given an algebra $(A,m,u)$ in $\cC$, a \emph{left $A$-module in $\cC$} is a pair $(V,a_V)$ for $V$ an object in $\cC$ and $$a_V:=a_V^l\colon A\otimes V\to V,$$ a morphism in $\cC$ satisfying 
$a_V(m\otimes \ide_V)=a_V(\ide_A \otimes  \; a_V)$ and $a_V(u\otimes \ide_V)=\ide_V$.
A morphism of $A$-modules $(V,a_V) \to (W, a_W)$ is a morphism $V \to W$ in $\cC$ that intertwines with $a_V$ and $a_W$. 
This way, we define the category $\lmod{A}(\cC)$ of left $A$-modules in $\cC$. Analogously, we define $\rmod{A}(\cC)$, the category of \emph{right $A$-modules $(V, a_V^r)$ in $\cC$}, with right $A$-action morphism:
$$a_V^r: V \otimes A \to V.$$
An {\it $A$-bimodule in $\cC$} is triple $(A, a_V^l, a_V^r)$ in $\cC$ where $(A,a_V^l) \in  \lmod{A}(\cC)$, $(A,a_V^r) \in  \rmod{A}(\cC)$, and $a_V^r(a_V^l \otimes \ide_A) = a_V^l(\ide_A \otimes a_V^r)$. Morphisms of $A$-bimodules in $\cC$ are morphisms simultaneously in $\lmod{A}(\cC)$ and in $\rmod{A}(\cC)$, and these together with  $A$-bimodules in $\cC$, form a category which we denote by $\bimod{A}(\cC)$.

\smallskip

An algebra $(A,m,u)$ in $\cC$ is called \emph{separable} if $m$ has a splitting as an $A$-bimodule map, i.e., if there exists a morphism $t: A \to A\otimes A$ in $\bimod{A}(\cC)$ so that $m t = \ide_A$. Here, $A \in \bimod{A}(\cC)$ via left and right multiplication, and likewise  $A \otimes A  \in \bimod{A}(\cC)$  via left and right multiplication respectively in the first and second slot. So, $t \in \bimod{A}(\cC)$ means that $t \in \cC$ and 
\begin{equation} \label{eq:sep}
(m \otimes \ide_A) (\ide_A \otimes  t) = t  m = (\ide_A \otimes m) (t  \otimes \ide_A).
\end{equation}
An algebra in $\cC$ is said to be \emph{\'etale} if it is commutative and separable.

\smallskip

We say that an algebra $(A,m,u)$ in $\cC$ is {\it connected} (or {\it haploid}) if $\dim_\Bbbk \Hom_\cC(\one, A) = 1$.
We remark that there is an isomorphism of $\Bbbk$-algebras
\begin{equation} \label{eq:modAA}
\phi\colon \Hom_\cC(\one, A) \isomorph \End_{\rmod{A}(\cC)}(A), \quad f \mapsto  \phi_f:=m(f\otimes \ide_A),
\end{equation}
where the algebra structure is given by the convolution product on the left hand side and by composition on the right hand side.

\smallskip

An algebra in $\cC$ is called \emph{indecomposable} if it is not isomorphic to a  direct sum of nonzero algebras in $\cC$; else it is {\it decomposable}. Note that if $A$ is decomposable, with $A \cong A_1 \oplus A_2$, then there exist morphisms $\one \overset{u_1}{\to} A_1 \hookrightarrow A$ and $\one \overset{u_2}{\to} A_2 \hookrightarrow A$ in $\cC$ that are not scalar multiples of each other. So, the connected condition implies indecomposability.  

\smallskip

A {\it coalgebra} in $\cC$ is an object $C \in \cC$ equipped with two morphisms $\Delta\colon  C \to C \otimes C$ (comultiplication) and $\varepsilon\colon  C \to \one$ (counit) satisfying $(\Delta \otimes \ide_C) \Delta= (\ide_C \otimes\; \Delta)\Delta$ and $(\varepsilon \otimes \ide_C) \Delta= \ide_C = (\ide_C \otimes\; \varepsilon)\Delta$. Dual to above, we can define the category $\Coalg(\cC)$ of coalgebras and their morphisms in $\cC$, and given $C \in \Coalg(\cC)$ we can define categories, $\lcomod{C}(\cC)$ and $\rcomod{C}(\cC)$, of {\it left} and {\it right $C$-comodules in $\cC$}, respectively. For $V \in \lcomod{C}(\cC)$, the left $C$-coaction map is denoted by
$$\delta_V\colon  V \to C \otimes V.$$


\subsection{Bialgebras and Hopf algebras, and their (co)modules} \label{sec:Hopfalg}
In this part, we define bialgebras and Hopf algebras in a braided finite tensor category $\cC$ over $\Bbbk$. We refer the reader to  \cite{EGNO}*{Sections~7.14,~7.15,~8.3}, \cite{TV}*{Sections~6.1 and~6.2}, and \cite{Bes}*{Section~3} for more details.

\smallskip

A {\it bialgebra} in $\cC$ is a tuple $H:=(H,m,u,\Delta,\varepsilon)$ where $(H,m,u) \in {\sf Alg}(\cC)$ and $(H,\Delta,\varepsilon) \in {\sf Coalg}(\cC)$ so that $\Delta m=(m\otimes m)(\ide\otimes c \otimes \ide)(\Delta\otimes \Delta)$, $\Delta u = u \otimes u$, $\varepsilon m = \varepsilon \otimes \varepsilon$, and $\varepsilon u = \ide_{\one}$.  We denote by $\Bialg(\cC)$ the category of bialgebras in $\cC$, where morphisms in $\Bialg(\cC)$ are morphisms in $\cC$ that belong to $\Alg(\cC)$ and $\Coalg(\cC)$ simultaneously. 

\smallskip

A {\it Hopf algebra} is a tuple $H:=(H,m,u,\Delta,\varepsilon, S)$, where $(H,m,u,\Delta,\varepsilon) \in \Bialg(\cC)$ and $S\colon H \to H$ is a morphism in $\cC$ (called an {\it antipode}) so that 
$m(S \otimes \ide_H)\Delta = m(\ide_H \otimes S)\Delta = u \varepsilon$. We denote by $\HopfAlg(\cC)$ the category of Hopf algebras in $\cC$, where morphisms  are morphisms in $\Bialg(\cC)$ that preserve the antipode. Note that by \cite{Tak99}*{Theorem~4.1, Proposition--Definition~2.9} the antipode of a Hopf algebra in $\cC$ must be invertible.

\smallskip

Now we discuss (co)modules over Hopf algebras $H$ in $\cC$.
If $V,W$ are left $H$-modules in $\cC$, then so is the tensor product $V \otimes W$, via the action \eqref{tensorproductaction} below:
\begin{equation}\label{tensorproductaction}a_{V\otimes W}:= (a_V\otimes a_W)(\ide_H\otimes c_{H,V}\otimes \ide_W)(\Delta_{H} \otimes \ide_{V\otimes W}).
\end{equation}
This makes the category $\lmod{H}(\cC)$  a monoidal category, with unit object $(\one, a_{\one} = \varepsilon_{H} \otimes \ide_{\one})$.
Take $(V,a_V) \in \lmod{H}(\cC)$. Then its left dual 
 $(V^*,a_{V^*}) \in \lmod{H}(\cC)$ is defined using $S_H$, and its right dual $({}^*V,a_{{}^*V}) \in \lmod{H}(\cC)$ is defined using $S^{-1}_H$. 
 It follows that $\lmod{H}(\cC)$ is a (finite) tensor category provided $\cC$ is a (finite) braided tensor category.
 
\smallskip

For  a supply of braided tensor categories, take a Hopf algebra $H$ in $\cC$, and consider the category of {\it $H$-Yetter--Drinfeld modules in $\cC$}, denoted by $\lYD{H}(\cC)$, which consists of objects $(V,a_V,\delta_V)$, where $(V,a_V) \in \lmod{H}(\cC)$ with left $H$-coaction  in $\cC$ denoted by $\delta_V\colon  V \to H \otimes V$, subject to the compatibility condition:
\[
\begin{split} 
\label{eq:HYDB}
&(m_H \otimes a_V)(\ide_H \otimes c_{H,H} \otimes \ide_V)(\Delta_H \otimes \delta_V)\\
&\quad = (m_H \otimes \ide_V)(\ide_H \otimes c_{V,H})(\delta_V \otimes \ide_H) (a_V \otimes \ide_H) (\ide_H \otimes c_{H,V})(\Delta_H \otimes \ide_V).
\end{split}
\]
A morphism $f\colon (V,a_V,\delta_V)\to (W,a_W,\delta_W)$ in $\lYD{H}(\cC)$ is given by a morphism $f\colon V\to W$ in $\cC$ that belongs to $\lmod{H}(\cC)$ and $\lcomod{H}(\cC)$. Given two objects $(V,a_V,\delta_V)$ and $(W,a_W,\delta_W)$ in $\lYD{H}(\cC)$, their tensor product  is given by $(V\otimes W, a_{V\otimes W}, \delta_{V\otimes W})$, where
 $a_{V\otimes W}$  as in \eqref{tensorproductaction} and 
\begin{align*}
\delta_{V\otimes W}=(m_H\otimes \ide_{V\otimes W})(\ide_H\otimes c_{H,V}\otimes \ide_W)(\delta_V\otimes \delta_W)
.
\end{align*}
The category $\lYD{H}(\cC)$ is braided with braiding  given by
$c^{\sf YD}_{V,W} = (a_W \otimes \ide_V)(\ide_H \otimes c^{\cC}_{V,W})(\delta_V \otimes \ide_W).$


\subsection{Frobenius algebras} \label{sec:Frobalg}
For this section, take $(\cC, \otimes, \one, c, \theta, j)$ to be a ribbon tensor category over $\Bbbk$,
with induced pivotal structure \eqref{eq:ribbonpivotal}. In this part, we compare different sets of assumptions on algebras in $\cC$ that yield modular categories of local modules (see Section~\ref{subsec:local} later).

A {\it Frobenius algebra} in $\cC$ is an object $A \in \cC$ equipped with morphisms $m: A \otimes A \to A$, $u: \one \to A$, $\Delta: A \to A \otimes A$, and $\varepsilon: A \to \one$ such that $(A,m,u) \in \Alg(\cC)$, $(A,\Delta, \varepsilon) \in \Coalg(\cC)$, and 
\begin{equation} \label{eq:Frob-comp}
(m \otimes \ide) (\ide \otimes  \Delta) = \Delta m = (\ide \otimes m) (\Delta \otimes \ide).
\end{equation}
Note that a Frobenius algebra $(A,m,u,\Delta,\varepsilon)$ in $\cC$ is self-dual with evaluation and coevaluation morphisms given by 
$\ev_{A} = \evr_A =\varepsilon m$ and  $\coev_A = \coevr_A =\Delta u.$ Moreover, by \eqref{qdim-j} we have
\begin{align} \label{dimj-Frob}
\dim_j(A) =  \evr_A \; \coev_A = \varepsilon \; m \; \Delta \; u.
\end{align}

We call a Frobenius algebra $(A,m,u,\Delta,\varepsilon)$ in $\cC$ {\it special} if $m \Delta = \beta_A \ide_A$ and $\varepsilon u = \beta_\one \ide_\one$ for $\beta_A, \beta_\one \in \Bbbk^\times$. In this case (whether $\beta_A, \beta_\one$ are nonzero or not), 
\begin{align} \label{dimj-beta}
\dim_j(A)  =  \beta_A \beta_{\one} \ide_\one.
\end{align}

\smallskip

The following definition will be central in Section \ref{sec:local-mod} when we build modular categories of local modules in the non-semisimple setting; see Section~\ref{subsec:local} and Theorem~\ref{thm:locmodular}.

\begin{definition} \label{def:rigid-Frob}
An algebra $A$ in $\cC$ is called {\it rigid Frobenius} if it is connected, commutative, and special Frobenius.
\end{definition}

Next, we compare the definition of a rigid Frobenius algebra above to the algebras used in \cite{KO} to obtain modular categories of  {\it local modules} (cf., Section~\ref{subsec:local} below). 
These algebras $(A,m,u)$ in $\cC$  satisfy the following conditions: 
\begin{itemize} 
\item[(r.i)] $A$ is connected and commutative, and  $\dim_j(A)$ is a nonzero scalar in $\Bbbk$, 
\item[(r.ii)] $A$ is equipped with a morphism $\varepsilon: A \to \one$ so that $\varepsilon u = \ide_\one$, 
\item[(r.iii)] $p:=\varepsilon m: A \otimes A \to \one$ is a {\it non-degenerate pairing}, i.e.,  $p(m \otimes \ide_A) = p(\ide_A \otimes m)$ and there exists a morphism 
$q: \one \to A \otimes A$ in $\cC$ such that $(p \otimes \ide_A) (\ide_A \otimes  q) = \ide_A = (\ide_A \otimes p) (q \otimes \ide_A),$
\item[(r.iv)] $\theta_A = \ide_A$, that is, $A$ has a {\it trivial twist}. 
\end{itemize}

\begin{definition} \label{def:rigidCalgebra} \cite{KO}*{Definition~1.11}
An algebra $A$ in $\cC$ for which conditions (r.i)-(r.iii) hold is called a {\it rigid $\cC$-algebra}.
\end{definition}

In particular, conditions (r.i)-(r.iv) are used to show that a category of (local) modules over $A$ is rigid \cite{KO}*{Theorem~1.15} in the semisimple setting. 
Now we compare the algebra with conditions (r.i)-(r.iv) above with rigid Frobenius algebras.

\begin{lemma} \label{lem:rigid-alg} 
If $A$ is a rigid  Frobenius algebra, then it is a rigid $\cC$-algebra with trivial twist. 
\end{lemma}

\begin{proof}
We must show that  a rigid Frobenius algebra  is an algebra for which conditions (r.i)-(r.iv) hold. Namely, condition (r.i) holds  by Definition~\ref{def:rigid-Frob} and \eqref{dimj-beta}. Next, we can choose the counit of a special Frobenius algebra in $\cC$ so that condition (r.ii) holds, i.e., with $\beta_A = \dim_j(A)$ and $\beta_\one = 1$ in \eqref{dimj-beta}.
Moreover, condition (r.iii) holds by the various characterizations of Frobenius algebras; see, e.g.,  \cite{FS}*{Proposition~8(i)}. Here, $q = \Delta u$. Finally, connected Frobenius algebras  must satisfy condition (r.iv) by
\cite{FRS}*{Corollary~3.10} and
\cite{FFRS}*{Proposition~2.25(i)}. (In particular, \cite{FRS}*{Corollary~3.10} uses that  $A$ is connected if and only if $\dim_{\Bbbk} \Hom_{\cC}(A,\one) = 1$. This follows since both $A$ and $\one$ are self-dual in $\cC$.) Therefore, $A$ is a rigid $\cC$-algebra with $\theta_A = \ide_A$.
\end{proof}

We will see that  the converse of the result above holds in Proposition~\ref{prop:conn-etale} below.

\smallskip

As a restatement of \cite{KO}*{Theorem~1.15} in the semisimple setting, the result \cite{Dav3}*{Theorem~2.6.3} also builds modular categories of local modules over certain types of algebras in modular categories. As mentioned above, the former result uses rigid $\cC$-algebras with trivial twists. But we need the next two results to understand the assumptions in the latter result.
In particular, see Remark~\ref{rem:Davydov} below.

\begin{proposition} \label{prop:conn-etale}
The following statements are equivalent:
\begin{enumerate}[(a),font=\upshape]
    \item $A$ is a rigid $\cC$-algebra with trivial twist;
    \item $A$ is a connected \'{e}tale algebra, with $\dim_j A \neq 0$, and with trivial twist;
    \item $A$ is a rigid Frobenius algebra.
\end{enumerate}
\end{proposition}

\begin{proof}
(a) $\Rightarrow$ (b). Let $(A,m,u)$ be a rigid $\cC$-algebra (with trivial twist). Then, $A$ is commutative, connected, with $\dim_j A \neq 0$ by definition. Moreover, by conditions (r.ii), (r.iii) and \cite{FS}*{Proposition~8(ii)}, $A$ is special Frobenius. We can rescale the coproduct $\Delta$ in the special Frobenius condition so that it is the desired $A$-bimodule splitting of $m$ to yield separability; see \cite{FFRS}*{Remark~2.23(i)}.  Thus, $A$ is connected, \'{e}tale, with $\dim_j A \neq 0$ (with trivial twist).

\smallskip

(b) $\Rightarrow$ (c). 
Suppose that $(A,m,u)$ is a connected \'{e}tale algebra with $\dim_j A \neq 0$ (and with trivial twist). To show that $A$ is rigid Frobenius, it remains to show that $A$ is special Frobenius.

By the separability assumption on $A$, there exists a morphism $t \in \Hom_{\bimod{A}(\cC)}(A, A \otimes A)$ so that $m t = \ide_A$ and \eqref{eq:sep} holds. Take $\Delta = t$; this yields the compatibility conditions between $\Delta$ and $m$ for the special Frobenius requirement. Moreover, 
\[
\begin{array}{rll}
(\Delta \otimes \ide_A)\Delta m 
&= 
(\Delta \otimes \ide_A)(m \otimes\ide_A)(\ide_A \otimes \Delta) 
&= 
(\ide_A \otimes  m \otimes \ide_A)(\Delta \otimes \Delta)\\
&=(\ide_A \otimes \Delta)(\ide_A \otimes m)(\Delta \otimes \ide_A)
&=(\ide_A \otimes \Delta)\Delta m, 
\end{array}
\]
by applying \eqref{eq:sep} in each step above. Since $m$ is epic, we obtain that $\Delta$ is coassociative.

Next, take $\varepsilon = \evr_A(m \otimes \ide_{A^*})(\ide_A \otimes \coev_A)$. Then $\varepsilon u = \dim_j(A)$, by the unitality axiom. Therefore, $\varepsilon u = \beta_\one \ide_\one$, for $\beta_\one \neq 0$. Moreover, by using \eqref{eq:sep}, we have that $(\varepsilon \otimes \ide_A)\Delta$ is a right $A$-module map. So, $(\varepsilon \otimes \ide_A)\Delta = \lambda\; \ide_A$, for some scalar $\lambda$, by \eqref{eq:modAA} and the connected assumption. Now,
\[
\begin{split}
\lambda\;  \dim_j(A) 
&= \varepsilon(\lambda\; \ide_A)u \\
& = \varepsilon(\varepsilon \otimes \ide_A)\Delta u\\ 
&= \evr_A(m \otimes \ide_{A^*})(\ide_A \otimes \coev_A)[(\evr_A(m \otimes \ide_{A^*})(\ide_A \otimes \coev_A)) \otimes \ide_A]\Delta u\\
&= (\evr_A \otimes \evr_A)(m \otimes \ide_{A^* \otimes A \otimes A^*})(\ide_A \otimes \coev_A \otimes \ide_{A \otimes A^*})(\Delta m \otimes \ide_{A^*}) (u \otimes \coev_A)\\
&= (\evr_A \otimes \evr_A)(m \otimes \ide_{A^* \otimes A \otimes A^*})(\ide_A \otimes \coev_A \otimes \ide_{A \otimes A^*})(\Delta \otimes \ide_{A^*})\coev_A\\
&= (\evr_A \otimes \evr_A)(m \;c_{A,A} \otimes \ide_{A^* \otimes A \otimes A^*})(\ide_A \otimes \coev_A \otimes \ide_{A \otimes A^*})(\Delta \otimes \ide_{A^*})\coev_A\\
&= \evr_A(m \otimes \evr_A \otimes \ide_{A^*})(\ide_A \otimes \Delta \otimes \ide_{A^* \otimes A^*})(\ide_A \otimes \coev_A \otimes \ide_{A^*})\coev_A\\
&= \evr_A(\ide_A \otimes \evr_A \otimes \ide_{A^*})(\Delta m \otimes \ide_{A^* \otimes A^*}) (\ide_A \otimes \coev_A \otimes \ide_{A^*})\coev_A\\
&= \evr_A (m \Delta \otimes \ide_{A^*}) \coev_A\\
&= \dim_j(A).
\end{split}
\]
Here, the fourth equality holds by \eqref{eq:sep}, the fifth equality holds by unitality, the sixth equality holds by commutativity, the seventh equality holds by a rearranging of maps (best viewed by graphical calculus), the eighth equality holds by \eqref{eq:sep},  the ninth equality holds by the trace being cyclic, and the last equality holds by $m \Delta = \ide_A$. Thus, $\lambda = 1$, and $(\varepsilon \otimes \ide_A)\Delta = \ide_A$.

On the other hand, $(\ide_A \otimes \varepsilon)\Delta$ is a left $A$-module map, and is thus a scalar multiple of $\ide_A$ by a left module version of \eqref{eq:modAA} and the connected assumption. This scalar is equal to 1 by an argument similar to the above because
$\varepsilon(\ide_A \otimes \varepsilon)\Delta u =
\varepsilon(\varepsilon \otimes \ide_A)\Delta u$.
Thus, $\Delta$ is counital via $\varepsilon$ defined above.

\smallskip

(c) $\Rightarrow$ (a). This implication follows from Lemma~\ref{lem:rigid-alg}.
\end{proof}

\begin{corollary} \label{cor:sep}
Let $A=(A,m,u)$ be a connected, commutative algebra  in $\cC$ with $\dim_j(A) \neq 0$. Then, $A$ is separable if and only if $\varepsilon_{\sharp} \;  m \colon A\otimes A\to \one$ is a non-degenerate pairing of $A$, for any non-zero morphism $\varepsilon_{\sharp}:A\to \one$.
\end{corollary}

\begin{proof}
If $p:=\varepsilon_{\sharp} \;  m$ is a non-degenerate pairing on $A$, say with copairing $q\colon \one \to A \otimes A$, then $(A,m,u)$ admits the structure of a Frobenius algebra by taking $\Delta:=(\ide_A \otimes m)(q \otimes \ide_A)$ and $\varepsilon:=p(u \otimes \ide_A) = \varepsilon_\sharp$; see, e.g., \cite{FS}*{Proposition~8(ii)}. In particular, $\Delta = t$ satisfies \eqref{eq:sep}. Moreover, $m \Delta \in \Hom_{\rmod{A}(\cC)}(A,A)$. Since $A$ is connected, we get that $m \Delta = \beta_A \ide_A$ by \eqref{eq:modAA}. But $\beta_A \neq 0$ due to $\dim_j(A) \neq 0$ (see \eqref{dimj-beta}), so we can rescale $\Delta$ to get that $\Delta$ is an $A$-bimodule splitting morphism of $m$, as required.

Now suppose that $A$ is separable. Then, $A$ is a connected \'{e}tale algebra with $\dim_j(A) \neq 0$. The proof of (b) $\Rightarrow$ (c) in Proposition~\ref{prop:conn-etale} then implies that $A$ admits the structure of a Frobenius algebra $(A,m,u,\Delta,\varepsilon)$. In particular, $\varepsilon m$ is a non-degenerate pairing on $A$; see \cite{FS}*{Proposition~8(i)}. 
\end{proof}

\begin{remark} \label{rem:Davydov}
The goal of  \cite{Dav3}*{Theorem~2.6.3} is to provide sufficient conditions for commutative algebras $(A,m,u)$ with a form $\varepsilon_\sharp: A \to \one$ to yield modular categories of local modules. Such conditions $A$ in \cite{Dav3} are that (i) $A$ must be indecomposable, and that (ii) $\varepsilon_\sharp \; m$ is a non-degenerate pairing.
\begin{enumerate}
\item The choice $\varepsilon_{\sharp}:=\ev_A \; c_{A,A^*}(m \otimes \ide_A)(\ide_A \otimes \coev_A)$ is used in \cite{Dav3} and  in  \cite{FFRS}*{(2.39)}.

\smallskip

\item The condition that $\dim_j A \neq 0$ is implicit in \cite{Dav3} (as $\cC$ is semisimple there), so this was added explicitly to Proposition~\ref{prop:conn-etale} and Corollary~\ref{cor:sep} for the non-semisimple case.

\smallskip

    \item For (i), note that the references for the result in \cite{Dav3}, namely \cite{KO} and \cite{DMNO}, employ the (stronger) connected condition instead of indecomposability. This justifies the connected assumption being used in Proposition~\ref{prop:conn-etale} and Corollary~\ref{cor:sep}. Note that if $\Bbbk$ is algebraically closed, then any indecomposable \'etale algebra of  nonzero quantum dimension in $\cC$ is connected by Corollary~\ref{cor:indec-con}. In particular, such algebras are commutative and special Frobenius by the proof of Proposition~\ref{prop:conn-etale}(b) $\Rightarrow$ (c).

    \smallskip
    
    \item Condition (ii) is referred to as separability in \cite{Dav3}, but this implies the usual notion separability  as in shown in Corollary~\ref{cor:sep} in the case when $A$ is connected.

\smallskip

\item Since \cite{Dav3}*{Theorem~2.6.3} is claimed to be a restatement of \cite{KO}*{Theorem~1.15}, a rigid $\cC$-algebra is supposed to be precisely an indecomposable \'{e}tale algebra. Given the last two remarks, this is precisely the equivalence of parts (a) and (b) in Proposition~\ref{prop:conn-etale}.

\smallskip

\item In view of Proposition~\ref{prop:conn-etale}, Davydov classified all rigid Frobenius algebras in $\cZ(\lmod{\Bbbk G})$, for a finite group $G$, over a field $\Bbbk$ of characteristic zero \cite{Dav3}*{Theorem 3.5.1}. Indeed, all algebras in this classification have a trivial twist \cite{Dav3}*{Remark 3.5.2}. We will generalize this to arbitrary characteristic in Section~\ref{sec:Z(kG-mod)}.
\end{enumerate}
\end{remark}


\section{Modularity of categories of local modules}\label{sec:local-mod}

In this section, we study a monoidal category of \emph{local modules} over a commutative algebra $A$ in a braided finite tensor category $\cC$. The main result is that if $\cC$ is modular (and not necessarily semisimple) [Definition~\ref{def:modular}], and $A$ a rigid Frobenius algebra [Definition~\ref{def:rigid-Frob}], then the category of such local modules is also modular [Theorem~\ref{thm:locmodular}]. We  provide preliminary material on categories of (local) modules over commutative algebras in Sections~\ref{subsec:RepCA} and~\ref{subsec:local}. Then, in Section~\ref{subsec:mainresult},  we establish our main result when categories of local modules over rigid Frobenius algebras are modular; this is a non-semisimple generalization of \cite{KO}*{Theorem~4.5}. 
We end by computing  the Frobenius--Perron (FP-)dimension of these categories of modules in Section~\ref{sec:FP-results}.

\begin{notation} \label{not:Sec4}
Unless stated otherwise, take $A$ to be a commutative algebra in a braided finite tensor category $(\cC,\otimes, \one, c)$. 
\end{notation}

\subsection{Categories of modules over (braided) commutative algebras}
\label{subsec:RepCA}

Consider the following construction.

\begin{definition}\label{def:RepCA}  \cite{KO}*{Definition~1.2 and Theorem~1.5} Take $\Rep_{\cC}(A)$ to be the category whose objects are pairs $(V,a_V^r) \in \rmod{A}(\cC)$, and morphisms are morphisms in $\rmod{A}(\cC)$. We define  $a_V^l$   as
\begin{align}\label{def:a^l}
a^l_V := a^r_V c_{A,V}\colon A\otimes V\isomorph V\otimes A \longrightarrow V.
\end{align}
With this, $(V,a_V^l) \in \lmod{A}(\cC)$. As $A$ is commutative, the actions $a^r_V,a^l_V$ commute, $(V,a^l_V,a^r_V)$ becomes an $A$-bimodule in $\cC$, and $\Rep_{\cC}(A)$ is a full subcategory of $\bimod{A}(\cC)$.

The category $\Rep_{\cC}(A)$ is monoidal as follows. Given two objects $V,W$ in $\Rep_{\cC}(A)$, their tensor product $V\otimes_A W$ is defined to be the coequalizer
\begin{align}\label{eq:rel-tensor}
\xymatrix{
V\otimes A \otimes W\ar@/^/[rr]^{a_V^r\otimes \ide}\ar@/_/[rr]_{\ide\otimes a_W^l}&& V\otimes W \ar[r]& V\otimes_A W,
}
\end{align}
which is an object in $\Rep_\cC(A)$ using the right $A$-module structure induced by $a_{V\otimes_A W}^r = \ide_V\otimes a^r_W$, and $a_{V\otimes_A W}^l = a_{V\otimes_A W}^r \; c_{A,V\otimes_A W}$.  The unit object is the $A$-bimodule $A$ in $\cC$. This way,  $\Rep_{\cC}(A)$ is a monoidal subcategory of $\bimod{A}(\cC)$.
\end{definition}

We will consider a version of $\Rep_\cC(A)$ where $\cC$ is not necessarily a braided monoidal category in Definition \ref{def:RepCA-general} below.

The following lemma provides a description of objects in the category $\Rep_\cC(A)$ as quotients of free modules.

\begin{lemma}\label{indrep} 
There is a monoidal functor $U\colon \cC\to \Rep_\cC(A)$ defined by $U(X)=X\otimes A$ with right $A$-module structure induced by the multiplication map $m$, which is left adjoint to the forgetful functor $F\colon \Rep_\cC(A)\to \cC$. 

Moreover, if $\cC$ is a tensor category, the functor $U$ is faithful, and is surjective in the sense that every object in the target is a quotient of an object in the image of $U$.
\end{lemma}

\begin{proof}
Defining $$U\colon \cC\to \Rep_\cC(A), \quad U(X)=(X\otimes A,~ a^r_{U(X)}:=\ide_X\otimes m_A)$$
clearly gives a functor as stated. We check that $U$ is monoidal with the tensor product on $\Rep_\cC(A)$ defined in \eqref{eq:rel-tensor} via
\begin{align}
   \mu^U_{X,Y}:=(\ide_{X\otimes Y}\otimes m_A)(\ide_X\otimes c_{A,Y}\otimes \ide_A)\colon U(X)\otimes U(Y)\longrightarrow U(X\otimes Y). 
\end{align}
One checks that $\mu^U_{X,Y}$ descends to an isomorphism $U_{X,Y}:$ $U(X)\otimes_A U(Y)\to U(X\otimes Y)$, with inverse
$\big(\mu^U_{X,Y}\big)^{-1}= \ide_X\otimes u_A\otimes \ide_{Y\otimes A}$ inducing $U_{X,Y}^{-1}$. These maps are natural in $X,Y$. Moreover, take $U_0 = \ide_A : \one_{\Rep_\cC(A)} \to U(\one_{\cC})$, making $U$ a monoidal functor.

To prove adjointness, we note that, for $X$ in $\cC$, mutually inverse isomorphisms 
$$\Hom_{\Rep_\cC(A)}(U(X), V)\cong\Hom_{\cC}(X, F (V))$$
are given by mapping a morphism $f\colon X\to F(V)$ in $\cC$ to $a^r_V(f\otimes \ide_A)$ in $\Rep_\cC(A)$, and a morphism $g\colon U(X)=X\otimes A\to V$ in $\Rep_\cC(A)$ to $g(\ide_X\otimes u_A)$.

Now assume that $\cC$ is a tensor category. Then, $U(-)=(-)\otimes A$ is faithful by \cite{EGNO}*{Proposition~4.2.8, Exercise~4.3.11(1)}.
To show that $U$ is surjective, let $(V,a^r_V)$ be an object in $\Rep_\cC(A)$. The morphism 
$a^r_V\colon F(V)\otimes A\to V$
is a surjective morphism in $\Rep_\cC(A)$. Thus, $V$ is a quotient of $U(F(V))=V\otimes A$ and hence $U$ is surjective.
\end{proof}

The next lemma shows that exact functors of braided monoidal categories induce functors between module categories over commutative algebras.

\begin{lemma}\label{lem:RepF}
If $F\colon \cC\to \cD$ is an exact functor of braided monoidal categories, then  it induces a functor of monoidal categories 
$$F'\colon \Rep_\cC(A)\to \Rep_{\cD}(F(A)),$$
for any commutative algebra $A$ in $\cC$.
\end{lemma}

\begin{proof}
First note that for $(A,m_A,u_A)$ a commutative algebra in $\cC$, $(F(A),F(m_A)F_{A,A}, F(u_A)F_0)$ defines a commutative algebra in $\cD$. Similarly, given a right $A$-module
$(V,a^r_V)$ in $\Rep_{\cC}(A)$, $(F(V), F(a^r_V)F_{V,A})$ is a right $F(A)$-module. This way, $F$ induces a functor of categories $$F'\colon \Rep_\cC(A)\to \Rep_\cD(F(A)).$$
Now consider a pair of right $A$-modules $(V,a^r_V)$, $(W,a^r_W)$ and compute
\begin{align*}
    F(\pi_{V,W})F_{V,W}(a^r_{F(V)}\otimes \ide_{F(W)})&=F(\pi_{V,W})F_{V,W}(F(a_V^r)F_{V,A}\otimes \ide_{F(W)})\\
    &=F(\pi_{V,W}(a_V^r\otimes \ide_W))F_{V\otimes A,W}(F_{V,A}\otimes \ide_{F(W)})\\
     &=F(\pi_{V,W}(\ide_V\otimes a^r_Wc_{A,W}))F_{V,A\otimes W}(\ide_{F(V)}\otimes F_{A,W})\\
      &=F(\pi_{V,W})F_{V,W}(\ide_{F(V)}\otimes F(a^r_Wc_{A,W}))(\ide_{F(V)}\otimes F_{A,W})\\
       &=F(\pi_{V,W})F_{V,W}(\ide_{F(V)}\otimes a^r_{F(W)}c_{F(A),F(W)}),
\end{align*}
where $\pi_{V,W}\colon V\otimes W\to V\otimes_A W$ is the canonical epimorphism of the coequalizer. The calculation uses the definition of $a^r_{F(V)}$ in the first equality, followed by naturality of $F_{V,W}$ in the first component in the second equality, and  \eqref{eq:rel-tensor} as well as the coherence of $F_{-,-}$ in the third equality. The fourth and fifth equality use naturality of $F_{-,-}$ in the second component followed by compatibility of $F$ with the braiding and the definition of $a^r_{F(W)}$. Thus, we obtain an induced morphism 
$$F'_{V,W}\colon F'(V)\otimes_AF'(W)\to F'(V\otimes_A W).$$
As $F$ is exact, it preserves coequalizers. Hence, $F'_{V,W}$ is an isomorphism, and with $F'_0 = \ide_{F(A)}$, we get that  $F'$ is a monoidal functor.
\end{proof}

\begin{lemma}\label{lem:RepAB}
Let $A,B$ be commutative algebras in braided finite tensor categories $\cC,\cD$, respectively. Then there is an equivalence of monoidal categories
$$\Rep_\cC(A)\boxtimes \Rep_\cD(B)\tensimeq \Rep_{\cC\boxtimes\cD}(A\boxtimes B).$$
\end{lemma}

\begin{proof}
Consider the functor $$T\colon \Rep_\cC(A)\boxtimes \Rep_\cD(B)\to \Rep_{\cC\boxtimes\cD}(A\boxtimes B)$$
determined by sending $(X,a_X^r)\boxtimes (Y,a_Y^r)$ to $(X\boxtimes Y,a_X^r\boxtimes a_Y^r)$. This functor defines an equivalence of $\Bbbk$-linear categories, cf.
 \cite{DSS2}*{Proposition~3.9}. We equip $T$ with the structure of a functor of monoidal categories. The structural isomorphisms can be defined by $$T_{(X,a_X^r)\boxtimes (Y,a_Y^r),(X',a_{X'}^r)\boxtimes (Y',a_{Y'}^r)}=\ide_{(X\otimes_A X')\boxtimes (Y\otimes_B Y')}  \quad \text{and} \quad  T_0 = \ide_{A \boxtimes B},$$ and extends uniquely to general objects in the Deligne tensor product.
\end{proof}

Now consider the following result about the algebra $A^+$ in $\cZ(\cC)$, and its category of modules $\Rep_{\cZ(\cC)}(A^+)$, and how this relates to the category $\Rep_{\cC}(A)$ when $\cC$ is non-degenerate.

\begin{lemma} \label{lem:A+}
Recall the notation of  \eqref{eq:Z(C)object+}, \eqref{eq:Z(C)morphism+}. Take $(A, c_{A,-})$  a commutative algebra in $\cC$. Then the following statements hold. 
\begin{enumerate}[(1),font=\upshape]
    \item $A^+:=(A^+, (m_A)^+, (u_A)^+)$ is a commutative algebra of $\cZ(\cC)$.
    \smallskip
    
    \item If $\cC$ is non-degenerate, then we  have equivalences of monoidal categories, $$\Rep_{\cZ(\cC)}(A^+) \; \tensimeq \; \Rep_{\cC \boxtimes \overline{\cC}}(A \boxtimes \one)\; \tensimeq \; \Rep_{\cC}(A) \boxtimes \overline{\cC},$$ where $\overline{\cC}$ is the mirror of $\cC$. 
\end{enumerate}
\end{lemma}

\begin{proof}
Part (1) follows from Proposition~\ref{prop:mon-alg} and the fact that $(-)^+$ is a braided monoidal functor.

\smallskip

For (2), let us assume that $\cC$ is non-degenerate. So, the functor 
$$G\colon \cC\boxtimes \ov{\cC}\longrightarrow \cZ(\cC),$$
determined by  $X\boxtimes \one\mapsto (X,c_{X,-})$ and $\one\boxtimes X\mapsto (X,c^{-1}_{X,-})$, is an equivalence of braided monoidal categories, see \cite{Shi1}*{(3.10) and Theorem~4.2}. Note that the structural isomorphism
$$ 
\xymatrix@R=10pt{
G(X\boxtimes Y)\otimes G(X'\boxtimes Y')\ar@{=}[d]\ar[rr]^{G_{X\boxtimes Y,X'\boxtimes Y'}}&& G((X\otimes X')\boxtimes (Y\otimes Y'))\ar@{=}[d]\\
X\otimes Y\otimes X'\otimes Y'\ar[rr]_{\ide_X\otimes c_{X',Y}^{-1}\otimes \ide_{Y'}}&& X\otimes X'\otimes Y\otimes Y'
}
$$
is used here.
Since $G(A\boxtimes \one)=A^+$, Lemma \ref{lem:RepF} implies that the functor $G$ induces a monoidal equivalence,
$$H \colon \Rep_{\cC\boxtimes \ov{\cC}}(A\boxtimes \one)\;\tensimeq \;\Rep_{\cZ(\cC)}(A^+).$$

Moreover, using Lemma \ref{lem:RepAB}, we have the natural equivalences of monoidal categories below:
$$\Rep_{\cC\boxtimes \ov{\cC}}(A\boxtimes \one) \;\tensimeq \;\Rep_{\cC}(A)\boxtimes \Rep_{
\ov{\cC}}(\one) \; \tensimeq \; \Rep_{\cC}(A)\boxtimes \ov{\cC}.$$
This completes the result.
\end{proof}


\subsection{Local modules}
\label{subsec:local}
We now recall background material on local modules following \cite{DMNO}*{Section~3.5}; recall Notation~\ref{not:Sec4}.

\begin{definition}\cite{Par}*{Definition 2.1}\label{def:locmod}
A right $A$-module $(V,a_V^r)$ in $\cC$ is called \emph{local}  if 
\begin{align*}
a^r_V&=a^r_V \; c_{A,V} \; c_{V,A}.
\end{align*}
The category of such local modules is denoted by $\locmod_{\cC}(A)$.
\end{definition}

\begin{proposition}\cite{Par}*{Theorem 2.5} \label{prop:pareigis}
The category $\locmod_\cC(A)$ is a monoidal subcategory of  $\Rep_\cC(A)$,  
and $\locmod_\cC(A)$ is braided. \qed
\end{proposition}

In particular, the braiding on $\locmod_\cC(A)$ is obtained from the braiding $c$ in $\cC$ and descends to the relative tensor products of two local modules. 
We also have the following useful result on the monoidal center of $\Rep_\cC(A)$ in terms of a category of local modules. For this, recall the notation from Lemma \ref{lem:A+}.

\begin{theorem} \cite{Sch}*{Corollary~4.5} \label{thm:local-center} 
The category $\cZ(\Rep_\cC(A))$ is equivalent to $\locmod_{\cZ(\cC)}(A^+)$  as braided monoidal categories. 
\qed
\end{theorem}

We will require the following auxiliary lemmata.

\begin{lemma}\label{lem:Floc}
For an equivalence of braided monoidal categories $F\colon \cC\to \cD$ and any  commutative algebra $A$ in $\cC$, the monoidal functor $F'$ from Lemma \ref{lem:RepF} restricts to an equivalence of braided monoidal categories 
$$F'\colon \locmod_\cC(A)\to \locmod_\cD(F(A)).$$
\end{lemma}

\begin{proof}
Clearly, if $(V,a^r_V)$ is a local module over $A$, then $(F(V),a^r_{F(V)})$ is a local module over $F(A)$. Conversely,
assume $F'(V)=(F(V),a^r_{F(V)})$ is a local module. Then we compute, using that $F$ is a functor of braided monoidal categories, that
\begin{align*}
F(a^r_V)F_{V,A}&=a^r_{F(V)}=a^r_{F(V)}c_{F(A),F(V)}c_{F(V),F(A)}\\
&=F(a^r_V)F_{V,A}c_{F(A),F(V)}c_{F(V),F(A)}\\
&=F(a^r_V)F(c_{A,V}c_{V,A})F_{V,A}\\
&=F(a^r_Vc_{A,V}c_{V,A})F_{V,A}.
\end{align*}
As $F_{V,A}$ is invertible, faithfulness of $F$ now implies that $(V,a^r_V)$ is a local module. Since $F$ is an equivalence, $F'$ is an equivalence and we have shown that it induces an equivalence on the full subcategories of local modules. As $F$ is compatible with braidings, this equivalence is one of braided monoidal categories.
\end{proof}

\begin{lemma}\label{lem:locmodAB}
Let $A,B$ be commutative algebras in braided finite tensor categories $\cC,\cD$, respectively. Then the monoidal equivalence $T$ from Lemma \ref{lem:RepAB} induces an equivalence of braided monoidal categories 
$$\locmod_{\cC}(A)\boxtimes \locmod_{\cD}(B) \; \overset{br.\otimes}{\simeq}\; \locmod_{\cC\boxtimes \cD}(A \boxtimes B).$$
\end{lemma}

\begin{proof}
If $(V,a_V^r) \boxtimes (W,a_W^r) \in \locmod_{\cC}(A)\boxtimes \locmod_{\cD}(B)$, then $T((V,a_V^r) \boxtimes (W,a_W^r))$ is in $\locmod_{\cC\boxtimes \cD}(A \boxtimes B)$. Conversely, an object of $\locmod_{\cC\boxtimes \cD}(A \boxtimes B)$ is in the essential image of $T$ by Lemma \ref{lem:RepAB}. So for $T((V,a_V^r) \boxtimes (W,a_W^r)) \in \locmod_{\cC\boxtimes \cD}(A \boxtimes B)$, we get that $(V,a_V^r) \boxtimes (W,a_W^r) \in \locmod_{\cC}(A)\boxtimes \locmod_{\cD}(B)$ with a similar argument as the proof of Lemma~\ref{lem:Floc}. Moreover, $T$ is compatible with the braidings of $\locmod_{\cC}(A)\boxtimes \locmod_{\cD}(B)$ and $\locmod_{\cC\boxtimes \cD}(A \boxtimes B)$.
\end{proof}

Now consider the following result regarding the category of local modules $\locmod_{\cZ(\cC)}(A^+)$, and how this relates to the category $\locmod_{\cC}(A)$ when $\cC$ is non-degenerate. This appears in \cite{DMNO}*{Lemma~3.29} in the semisimple case.

\begin{lemma} \label{lem:A+local} 
Retain the notation from Lemma~\ref{lem:A+}. If $\cC$ is non-degenerate, then we  have an equivalence of braided monoidal categories, $$\locmod_{\cZ(\cC)}(A^+) \; \overset{br.\otimes}{\simeq}\; \locmod_{\cC}(A) \boxtimes \overline{\cC}.$$
\end{lemma}

\begin{proof}
By Lemmas \ref{lem:RepAB} and \ref{lem:A+}, we have the following equivalences of monoidal categories:
$$\Rep_{\cC}(A)\boxtimes\Rep_{\ov{\cC}}(\one)
\; \tensimeq \;  \Rep_{\cC\boxtimes \ov{\cC}}(A\boxtimes \one) 
\; \tensimeq \; \Rep_{\cZ(\cC)}(A^+).$$
The result now follows from Lemmas \ref{lem:locmodAB} and \ref{lem:Floc}.
\end{proof}


\subsection{Main result: Modularity of categories of local modules}
\label{subsec:mainresult}

The following result generalizes  \cite{KO}*{Theorem~4.5} to   the non-semisimple setting. Recall the notion of a rigid Frobenius algebra from Definition~\ref{def:rigid-Frob}.

\begin{theorem}\label{thm:locmodular}
If $\cC$ is a modular tensor category and $A$ is a rigid Frobenius algebra in $\cC$, then the category $\locmod_\cC(A)$ of local modules over $A$ in $\cC$ is also modular.  
\end{theorem}

\begin{proof}
This holds by results established later in the section; namely, $\locmod_\cC(A)$ is a finite tensor category by Corollary~\ref{cor:finite-tensor}, is ribbon by Proposition~\ref{prop:ribbon}, and is non-degenerate by Proposition~\ref{prop:locmod-nondeg}.
\end{proof}

We now proceed to study the category $\locmod_\cC(A)$ to verify the results cited in the proof above.

\begin{notation}
We will assume, without loss of generality, the convention that a rigid Frobenius algebra $A=(A,m,u,\Delta,\varepsilon)$ is normalized in a way such that
$$m\Delta=d\; \ide_A, \qquad \varepsilon u=\ide_\one,$$
where  $d:=\dim_j(A)$. Fix the notation $p:=\varepsilon m\colon A \otimes A \to \one$ and $q:= \Delta u\colon \one \to A \otimes A.$ Note that
\begin{align} 
(m\otimes \ide_A)(\ide_A\otimes q) &=(\ide_A\otimes m)(q\otimes \ide_A) \quad (= \Delta), \label{eq:rigidFrob-I}\\
mq &=d u, \label{eq:rigidFrob-II}
\end{align}
by composing \eqref{eq:Frob-comp} with $\ide_A \otimes u_A$ and with  $u_A \otimes \ide_A$, and pre-composing the normalization $m\Delta=d\; \ide_A$ with $u$.
\end{notation}

The first step in the proof Theorem~\ref{thm:locmodular} is to show existence of duals and a ribbon structure for local modules. These arguments follow similarly to the semisimple case treated in \cite{KO}*{Section~1} and are included here for the reader's convenience. Compare the next result to \cite{KO}*{Theorems~1.15 and~1.17(3)} via Proposition~\ref{prop:conn-etale}, and note that the result on the pivotality of $\locmod_\cC(A)$ is new.


\begin{lemma}\label{lemma:rigid} 
Let $\cC$ be a  braided finite tensor  category, and take $A$ a rigid Frobenius algebra in~$\cC$. Then, the following statements hold.
\begin{enumerate}[(1),font=\upshape]
    \item $\Rep_{\cC}(A)$ is a rigid  monoidal category with rigid  subcategory $\locmod_\cC(A)$. \smallskip
    \item If, in addition, $\cC$ is pivotal, then so are $\Rep_{\cC}(A)$ and $\locmod_\cC(A)$.
\end{enumerate}
\end{lemma}

\begin{proof}
We first check that $\Rep_{\cC}(A)$ inherits left duals from $\cC$. Given an object $(V,a_V^r, a_V^l:= a_V^r c_{A,V})$ in  $\Rep_{\cC}(A)$, define
\begin{align*}
    a_{V^*}^r&:=(\ev_V\otimes \ide_{V^*})(\ide_{V^*} \otimes a_V^r\otimes \ide_{V^*})(\ide_{V^* \otimes V} \otimes c^{-1}_{A,V^*})(\ide_{V^*}\otimes \coev_V\otimes \ide_A) : V^* \otimes A \to V^*,\\
    \widehat{\ev}_V^A&:=(\ev_V\otimes \ide_A)(\ide_{V^*}\otimes a_V^r\otimes \ide_A)(\ide_{V^*\otimes V}\otimes q): V^* \otimes V \to A,\\
    \widehat{\coev}_V^A&:=d^{-1}(a_V^r\otimes \ide_V^*)(\ide_V\otimes c^{-1}_{A,V^*})(\coev_V\otimes\ide_A): A \to V \otimes V^*.
\end{align*}
With these maps, we need to show that: 
\begin{itemize}
    \item [(i)] The right $A$-module associativity and unitality axioms hold for $(V^*, a_{V^*}^r)$; \smallskip
    \item [(ii)] $\widehat{\ev}_V^A$ and $\widehat{\coev}_V^A$ are right $A$-module maps; \smallskip
    \item [(iii)] $\widehat{\ev}_V^A(a_{V^*}^r \otimes \ide_V) = \widehat{\ev}_V^A(\ide_{V^*} \otimes a_V^l)$, as then $\widehat{\ev}_V^A$ factors through a map $\ev_V^A: V^* \otimes_A V \to A$; \smallskip
    \item [(iv)] $a_V^r(\ide_V \otimes \widehat{\ev}_V^A)(\widehat{\coev}_V^A \otimes \ide_V) = a_V^l$ and $a_{V^*}^l(\widehat{\ev}_V^A \otimes \ide_{V^*})(\ide_{V^*} \otimes \widehat{\coev}_V^A) = a_{V^*}^r$, as this implies the rigidity axiom for $\ev_V^A$ in (iii) and the composition, $\coev_V^A: A \to V \otimes_A V^*$, of $\widehat{\coev}_V^A$ and the projection map $V \otimes V^* \to  V \otimes_A V^*$.
    \end{itemize}
Condition (i) follows from the associativity and unitality of $(V, a_V^r)$; details left to the reader. Condition (ii) follows from the calculations below:
\begin{align*}
\smallskip
\widehat{\ev}_V^A(\ide_{V^*} \otimes a_V^r) &=
(\ev_V\otimes \ide_A)(\ide_{V^*}\otimes a_V^r\otimes \ide_A)(\ide_{V^*\otimes V}\otimes q)(\ide_{V^*} \otimes a_V^r)\\ 
\smallskip
&\overset{\text{$V$ \hspace{-.08in} assoc}}{=}
(\ev_V\otimes \ide_A)(\ide_{V^*}\otimes a_V^r\otimes \ide_A)(\ide_{V^* \otimes V} \otimes m \otimes \ide_A)(\ide_{V^*\otimes V \otimes A}\otimes q)\\
\smallskip
& \overset{\eqref{eq:rigidFrob-I}}{=}
(\ev_V\otimes \ide_A)(\ide_{V^*}\otimes a_V^r\otimes \ide_A)(\ide_{V^* \otimes V \otimes A} \otimes m)(\ide_{V^*\otimes V} \otimes q \otimes \ide_A)\\
\smallskip
&= m(\widehat{\ev}_V^A \otimes \ide_A)\\
\smallskip
&= a_A^r(\widehat{\ev}_V^A \otimes \ide_A);
\end{align*}
\begin{align*}
\smallskip
\widehat{\coev}_V^A \; a_A^r &= d^{-1}(a_V^r\otimes \ide_V^*)(\ide_V\otimes c^{-1}_{A,V^*})(\coev_V\otimes\ide_A)m\\ \smallskip
&\overset{\text{$V$ \hspace{-.08in} assoc}}{=}d^{-1}(a_V^r \otimes \ide_{V^*})(\ide_V \otimes c^{-1}_{A,V^*})(a_V^r\otimes \ide_{V^* \otimes A})(\ide_V\otimes c^{-1}_{A,V^*} \otimes \ide_A)(\coev_V \otimes \ide_{A \otimes A})\\
&\overset{\text{$V$ \hspace{-.08in}  rigid}}{=} (\ide_V \otimes a_{V^*}^r)(\widehat{\coev}_V^A \otimes \ide_A).
\end{align*}
\noindent Condition (iii) holds by:
\begin{align*}
\smallskip
\widehat{\ev}_V^A(a_{V^*}^r \otimes \ide_V) &= [(\ev_V\otimes \ide_A)(\ide_{V^*}\otimes a_V^r\otimes \ide_A)(\ide_{V^*\otimes V}\otimes q)] \\ \smallskip
&\qquad  \circ [(\ev_V\otimes \ide_{V^*})(\ide_{V^*} \otimes a_V^r\otimes \ide_{V^*})(\ide_{V^* \otimes V} \otimes c^{-1}_{A,V^*})(\ide_{V^*}\otimes \coev_V\otimes \ide_A) \otimes \ide_V]\\ \smallskip
&\overset{\text{$V$ \hspace{-.08in}  rigid}}{=} (\ev_V \otimes \ide_A)(\ide_{V^*} \otimes a_V^r \otimes \ide_A)(\ide_{V^*} \otimes a_V^r \otimes \ide_{A \otimes A})\\ \smallskip
& \qquad \circ (\ide_{V^* \otimes V} \otimes c^{-1}_{A,A} \otimes \ide_A)(\ide_{V^*} \otimes c_{A,V} \otimes q)\\ \smallskip
&=(\ev_V \otimes \ide_A)(\ide_{V^*} \otimes a_V^r \otimes \ide_A)(\ide_{V^* \otimes V} \otimes m \;c_{A,A} \otimes \ide_A)(\ide_{V^*} \otimes c_{A,V} \otimes q)\\ \smallskip
&\overset{\text{$A$ \hspace{-.05in}  com}}{=} (\ev_V \otimes \ide_A)(\ide_{V^*} \otimes a_V^r \otimes \ide_A)(\ide_{V^* \otimes V} \otimes m  \otimes \ide_A)(\ide_{V^*} \otimes c_{A,V} \otimes q)\\ \smallskip
&\overset{\text{$V$ \hspace{-.08in}  assoc}}{=} (\ev_V \otimes \ide_A)(\ide_{V^*} \otimes a_V^r \otimes \ide_A)(\ide_{V^* \otimes V} \otimes q)(\ide_{V^*} \otimes a_V^r)(\ide_{V^*} \otimes c_{A,V})\\ 
&= \widehat{\ev}_V^A(\ide_{V^*} \otimes a_{V}^l).
\end{align*}

\noindent The first equation of condition (iv) holds by:
\begin{align*}
\smallskip
a_V^r(\ide_V \otimes \widehat{\ev}_V^A)(\widehat{\coev}_V^A \otimes \ide_V) &= 
d^{-1}\; a_V^r \; [\ide_V \otimes (\ev_V\otimes \ide_A)(\ide_{V^*}\otimes a_V^r\otimes \ide_A)(\ide_{V^*\otimes V}\otimes q)]\\ \smallskip
& \qquad \circ [(a_V^r\otimes \ide_V^*)(\ide_V\otimes c^{-1}_{A,V^*})(\coev_V\otimes\ide_A) \otimes \ide_V]\\ \smallskip
&\overset{\text{$V$ \hspace{-.08in}  rigid}}{=} d^{-1} a_V^r(a_V^r \otimes \ide_A)(\ide_A \otimes c_{A,V} \otimes \ide_A)(\ide_A \otimes a_V^r \otimes \ide_A)(\ide_{A \otimes V} \otimes q)\\\smallskip
&\overset{\text{$V$ \hspace{-.08in}  assoc}}{=} d^{-1} a_V^r(a_V^r \otimes \ide_A)(\ide_A \otimes m \; c_{A,A} \otimes \ide_A)( c_{A,V}  \otimes q)\\ \smallskip
&\overset{\text{$A$ \hspace{-.05in}  com}}{=} d^{-1} a_V^r(a_V^r \otimes \ide_A)(\ide_A \otimes m \otimes \ide_A)( c_{A,V}  \otimes q)\\ \smallskip
&\overset{\text{$V$ \hspace{-.08in}  assoc}}{=} d^{-1} a_V^r(a_V^r \otimes \ide_A)(a_V^r \otimes \ide_{A\otimes A})( c_{A,V}  \otimes q)\\ \smallskip
&\overset{\text{$V$ \hspace{-.08in}  assoc}}{=} d^{-1} a_V^r(a_V^r \otimes m)( c_{A,V}  \otimes q)\\ \smallskip
&\overset{\eqref{eq:rigidFrob-II}}{=} a_V^r(a_V^r \otimes \ide_A)( c_{A,V}  \otimes u)\\
&\overset{\text{$V$ \hspace{-.08in}  unital}}{=} a_V^l.
\end{align*}
Moreover, the second equation of condition (iv) by similar computations involving the associativity and unitality of $V$, commutativity of $A$, and \eqref{eq:rigidFrob-I}, \eqref{eq:rigidFrob-II}.

Likewise, we can check that $\Rep_{\cC}(A)$ inherits right duals of $\cC$, we define right duals via the maps
\begin{align*}
    a_{{}^*V}^r&:=(\ide_{{}^*V}\otimes \evr_V)(\ide_{{}^*V}\otimes a_V^r\otimes \ide_{{}^*V})(\ide\otimes c^{-1}_{{}^*V,A})(\coevr_V\otimes \ide_{{}^*V\otimes A}): {}^* V \otimes A \to {}^*V,\\
    \widehat{\widetilde{\ev}}_V^A&:=(\evr_V\otimes \ide_A)(a_V^r\otimes c_{A,{}^*V})(\ide_V\otimes q\otimes \ide_{{}^*V}): V \otimes {}^* V \to A,\\
    \widehat{\widetilde{\coev}}_{V}^A&:=d^{-1}(\ide_{{}^*V}\otimes a_V^r)(\coevr_V\otimes \ide_A): A \to {}^* V \otimes V.
\end{align*}
These yields morphisms
$$ \widetilde{\ev}_V^A\colon V \otimes_A {}^*V\to A, \qquad  \widetilde{\coev}_V^A\colon A\to {}^*V\otimes_A V$$
that give $({}^* V, a_{{}^* V}^r)$ the structure of a right dual of $(V,a_V^r)$. We leave the details to the reader.
Thus,  $\Rep_{\cC}(A)$ is a rigid monoidal category.

Now we show that $\locmod_\cC(A)$ is a rigid subcategory of $\Rep_{\cC}(A)$. Take $(V, a_V^r, a_V^l = a_V^r c_{A,V})$ in $\Rep_{\cC}(A)$. Then,
\begin{align*}
\smallskip
a_{V^*}^r  c_{A, V^*}  c_{V^*, A} &= (\ev_V\otimes \ide_{V^*})(\ide_{V^*} \otimes a_V^r\otimes \ide_{V^*})(\ide_{V^* \otimes V} \otimes c^{-1}_{A,V^*})(\ide_{V^*}\otimes \coev_V\otimes \ide_A)  c_{A, V^*} c_{V^*, A}\\[.2pc]
&\overset{\text{$V$ \hspace{-.06in} local}}{=} (\ev_V\otimes \ide_{V^*})(\ide_{V^*} \otimes a_V^r c_{A,V} c_{V,A} \otimes \ide_{V^*})(\ide_{V^* \otimes V} \otimes c^{-1}_{A,V^*}) \\
&\qquad \circ (\ide_{V^*}\otimes \coev_V\otimes \ide_A)\; c_{A, V^*} \;c_{V^*, A}\\[.5pc]
&= (\ev_A \otimes \ide_{V^*}) (\ide_V \otimes \ev_V \otimes \ide_{A \otimes V^*})(\ide_{A \otimes V^*} \otimes a^r_V c_{A,V} c_{V,A} \otimes \ide_{A \otimes V^*}) \\ 
&\qquad \circ (\ide_{A \otimes V^* \otimes V}\otimes c_{A,A} \coev_A \otimes \ide_{V^*}) (\ide_{A \otimes V^*}\otimes \coev_V) c_{V^*, A}\\[.5pc]
&= (\ev_V \otimes \ide_{V^*})(\ide_{V^*} \otimes a^r_V c_{A,V} \otimes \ide_{V^*}) (\ide_{V^* \otimes A} \otimes  \ev_A c_{A,A} \otimes \ide_{V \otimes V^*}) \\ 
&\qquad \circ (\ide_{V^*} \otimes c_{A,A} \coev_A \otimes \ide_A \otimes \coev_V)\\[.2pc]
&\overset{A \hspace{.02in}\textnormal{Frob}}{=} (\ev_V \otimes \ide_{V^*})(\ide_{V^*} \otimes a^r_V c_{A,V} \otimes \ide_{V^*}) (\ide_{V^* \otimes A} \otimes  \varepsilon_A m_A c_{A,A} \otimes \ide_{V \otimes V^*}) \\ 
&\qquad \circ (\ide_{V^*} \otimes c_{A,A} \coev_A \otimes \ide_A \otimes \coev_V)\\[.2pc]
&\overset{A \hspace{.02in}\textnormal{com}}{=} (\ev_V \otimes \ide_{V^*})(\ide_{V^*} \otimes a^r_V c_{A,V} \otimes \ide_{V^*}) (\ide_{V^* \otimes A} \otimes  \varepsilon_A m_A c^{-1}_{A,A} \otimes \ide_{V \otimes V^*}) \\ 
&\qquad \circ (\ide_{V^*} \otimes c_{A,A} \coev_A \otimes \ide_A \otimes \coev_V)\\[.5pc]
&\overset{A \hspace{.02in}\textnormal{Frob}}{=} (\ev_V \otimes \ide_{V^*})(\ide_{V^*} \otimes a^r_V c_{A,V} \otimes \ide_{V^*}) (\ide_{V^* \otimes A} \otimes  \ev_A c^{-1}_{A,A} \otimes \ide_{V \otimes V^*}) \\ 
&\qquad \circ (\ide_{V^*} \otimes c_{A,A} \coev_A \otimes \ide_A \otimes \coev_V)\\[.5pc]
&= (\ev_V\otimes \ide_{V^*})(\ide_{V^*} \otimes a_V^r \otimes \ide_{V^*})(\ide_{V^*}\otimes c_{A,V} \otimes \ide_{V^*})(\ide_{V^*\otimes A}\otimes \coev_V) \\[.5pc]
&= a_{V^*}^r.
\end{align*}
Likewise, $a_{{}^*V}^r \; c_{A, {}^*V}  \; c_{{}^*V, A} = a_{{}^*V}^r$. Thus, $\locmod_\cC(A)$ inherits rigidity from $\Rep_{\cC}(A)$.



\medskip

(2) It suffices to establish the conditions (i)-(iii) of Proposition~\ref{prop:piv-equiv} for the categories $\Rep_\cC(A)$ and $\locmod_\cC(A)$ when these conditions hold for $\cC$. Assume that ${}^* X = X^*$ for all $X \in \cC$. 

Then towards condition (i), take a morphism $(V, a_V^r) \to (W, a_W^r)$, with $f$ denoting the morphism $V \to W$, in $\Rep_\cC(A)$ (or in $\locmod_\cC(A)$). Now we need to show that 
\begin{align*} \smallskip
&(\ev_W^A \otimes_A \ide_{V^*})(\ide_{W^*} \otimes_A f \otimes_A \ide_{V^*})(\ide_{W^*} \otimes_A \coev_V^A)\\ 
&= (\ide_{V^*} \otimes_A \evr_W^A)(\ide_{V^*} \otimes_A f \otimes_A \ide_{W^*})(\coevr_V^A \otimes_A \ide_{W^*}).
\end{align*}
This holds due to the following computations:
\begin{align*}\medskip
&a_{V^*}^l(\widehat{\ev}_W^A \otimes \ide_{V^*})(\ide_{W^*} \otimes f \otimes \ide_{V^*})(\ide_{W^*} \otimes \widehat{\coev}_V^A)(\ide_{W^*} \otimes u_A)\\ \medskip
&=d^{-1}\;(\ev_V\otimes \ide_{V^*})(\ide_{V^*} \otimes a_V^r\otimes \ide_{V^*})(\ide_{V^* \otimes V} \otimes c^{-1}_{A,V^*})(\ide_{V^*}\otimes \coev_V\otimes \ide_A)c_{A,V^*}\\ \medskip
& \qquad \circ ([(\ev_W\otimes \ide_A)(\ide_{W^*}\otimes a_W^r\otimes \ide_A)(\ide_{W^*\otimes W}\otimes q)] \otimes \ide_{V^*})(\ide_{W^*} \otimes f \otimes \ide_{V^*})\\ \medskip
&\qquad \circ (\ide_{W^*} \otimes [(a_V^r\otimes \ide_V^*)(\ide_V\otimes c^{-1}_{A,V^*})(\coev_V\otimes\ide_A)])(\ide_{W^*} \otimes u_A)\\ \medskip
&\overset{\text{$V$ \hspace{-.05in} unital, rigid; $f \in \cC_A$}}{=} d^{-1} (\ev_W \otimes \ide_{V^*})(\ide_{W^*} \otimes f \otimes \ide_{V^*})(\ide_{W^*} \otimes a_V^r \otimes \ide_{V^*})(\ide_{W^*} \otimes a_V^r \otimes \ide_{A \otimes V^*})\\
&\qquad \qquad \qquad \qquad \circ (\ide_{W^* \otimes V} \otimes q \otimes \ide_{V^*})(\ide_{W^*} \otimes \coev_V)\\ \medskip
&\overset{\text{$V$ \hspace{-.05in} assoc}}{=} d^{-1} (\ev_W \otimes \ide_{V^*})(\ide_{W^*} \otimes f \; a_V^r \otimes \ide_{V^*})(\ide_{W^* \otimes V} \otimes mq \otimes \ide_{V^*})(\ide_{W^*} \otimes \coev_V)\\ \medskip
&\overset{\eqref{eq:rigidFrob-II},\text{$V$ \hspace{-.05in} unital}}{=} (\ev_W \otimes \ide_{V^*})(\ide_{W^*} \otimes f \otimes \ide_{V^*})(\ide_{W^*} \otimes \coev_V)\\ \medskip
&\overset{\cC \; \text{pivotal}}{=} (\ide_{V^*} \otimes \evr_W)(\ide_{V^*} \otimes f \otimes \ide_{W^*})(\coevr_V \otimes \ide_{W^*}),
\end{align*}
and with similar arguments,
\begin{align*} \medskip
&a_{V^*}^r(\ide_{V^*} \otimes \widehat{\evr}_W^A)(\ide_{V^*} \otimes f \otimes \ide_{W^*})(\widehat{\coevr}_V^A \otimes \ide_{W^*})(u_A \otimes \ide_{W^*} ) \\
&= (\ide_{V^*} \otimes \evr_W)(\ide_{V^*} \otimes f \otimes \ide_{W^*})(\coevr_V \otimes \ide_{W^*}).
\end{align*}
Thus, condition (i) of Proposition~\ref{prop:piv-equiv} holds for $\Rep_\cC(A)$ and $\locmod_\cC(A)$.

Moreover, we leave it to the reader to verify conditions (ii) and (iii) of Proposition~\ref{prop:piv-equiv} for $\Rep_\cC(A)$ and $\locmod_\cC(A)$. Therefore, the rigid categories $\Rep_\cC(A)$ and $\locmod_\cC(A)$ are pivotal.
\end{proof}

Moreover, we also have that $\locmod_{\cC}(A)$ is a finite tensor category when $\cC$ is a finite tensor category.

\begin{corollary}\label{cor:finite-tensor}
If $\cC$ is a braided finite tensor category, and $A$ is a rigid Frobenius algebra in $\cC$,  then  $\Rep_{\cC}(A)$ and $\locmod_\cC(A)$ are finite tensor categories.
\end{corollary}

\begin{proof}
We have that the monoidal category $\Rep_{\cC}(A)$ is a finite tensor category with rigid structure given by Lemma~\ref{lemma:rigid}(1), and the rest of the finite tensor structure given by \cite{KO}*{Lemma~1.4}. In particular, $\Hom_{\Rep_{\cC}(A)}(A,A) \cong \Bbbk$ follows from the assumption that $A$ is connected.

Moreover, $\locmod_\cC(A)$ is a full tensor subcategory of  $\Rep_{\cC}(A)$ in the sense that it is closed under taking  subobjects, subquotients, tensor products and duals due to Proposition~\ref{prop:pareigis} and Lemma~\ref{lemma:rigid}(1). So, $\locmod_\cC(A)$ is also a finite tensor category.
\end{proof}

\begin{corollary} \label{cor:indec-con}
Take $A$ to be a commutative, special Frobenius algebra in $\cC$. If $A$ is indecomposable as an algebra in $\cC$, then $A$ is connected.
\end{corollary}

\begin{proof}
Take a commutative, special Frobenius algebra  $A$ in $\cC$. Also note that ~\eqref{eq:modAA} can be updated by replacing $\rmod{A}(\cC)$ with $\Rep_{\cC}(A)$. Then by the proof of Corollary~\ref{cor:finite-tensor} it follows  that  $\Rep_{\cC}(A)$ is a multitensor category (that is, a tensor category where we do have necessarily have that $\End_{\Rep_{\cC}(A)}(A) \cong \Bbbk$.)  Therefore by \cite{EGNO}*{Theorem~4.3.1},  $A_\cC:=\End_{\Rep_{\cC}(A)}(A)$ is isomorphic to a direct sum of finitely many copies of $\Bbbk$.
Now if $A$ is not connected, then $\dim_{\Bbbk}\Hom_\cC(\one,A)=\dim_\Bbbk A_\cC> 1$ via~\eqref{eq:modAA}. In this case, $A_\cC$ has two orthogonal idempotent endomorphisms,   which are algebra endomorphisms. Indeed, using commutativity of $A$ we see that 
\begin{align*}
    \phi_e \; m=\;&\phi_e \; \phi_e \;  m=\phi_e \;  m \; (\ide_A\otimes \phi_e)=\phi_e \;  m \; c_{A,A} \; (\ide_A\otimes \phi_e)\\
    &=m \; (\ide\otimes \phi_e) \; c_{A,A} \; (\ide_A\otimes \phi_e)=m \; c_{A,A} \; (\phi_e\otimes \phi_e)=m \; (\phi_e\otimes \phi_e),
\end{align*}
for an idempotent $e\in \Hom_\cC(\one,A)$. Thus, $A$ is decomposable as an algebra in $\cC$.
\end{proof}

Next, we discuss the ribbon structure of $\locmod_\cC(A)$.

\begin{proposition}\label{prop:ribbon}
Let $\cC$ be a ribbon category and $A$ a rigid Frobenius algebra in $\cC$, then $\locmod_\cC(A)$ is a ribbon category.
\end{proposition}
    
\begin{proof}
By Proposition~\ref{prop:pareigis} and Lemma~\ref{lemma:rigid}, we have that $\locmod_\cC(A)$ is a braided pivotal category. So, by Proposition~\ref{prop:ribbon-left-right}, it suffices to show that the left and right twists, 
\[\theta_V^l := (\ev_V^A \otimes_A\ide_V)(\ide_{V^*}\otimes_A c_{V,V})(\coevr_{V}^A \otimes_A\ide_V), \quad \theta_V^r := (\ide_V\otimes_A \evr_V^A)(c_{V,V}\otimes \ide_{V^*})(\ide_V\otimes_A \coev_{V}^A),\]  coincide for all $(V, a_V^r) \in \locmod_\cC(A)$. Here, we adopt the notation of Lemma~\ref{lemma:rigid}, and consider ${}^* V = V^*$ via the pivotality of $\cC$. To check this, we consider the `lifts'  $\widehat{\theta}_V^l$ and $\widehat{\theta}_V^r$ as done in the proof of Lemma~\ref{lemma:rigid}, and compute:
\begin{align*}
\medskip
\widehat{\theta}_V^l &=a_V^l(\widehat{\ev}_V^A \otimes\ide_V)(\ide_{V^*}\otimes c_{V,V})(\widehat{\coevr}_V^A \otimes\ide_V)(u_A \otimes \ide_V), \\ \medskip  
&= d^{-1}\;a_V^l([(\ev_V\otimes \ide_A)(\ide_{V^*}\otimes a_V^r\otimes \ide_A)(\ide_{V^*\otimes V}\otimes q)] \otimes\ide_V)(\ide_{V^*}\otimes c_{V,V})\\ \medskip
& \quad \quad \circ ([(\ide_{{}^*V}\otimes a_V^r)(\coevr_V\otimes \ide_A)] \otimes\ide_V)(u_A \otimes \ide_V), \\ \medskip  
&\overset{\text{$V$ \hspace{-.05in} unital}}{=} d^{-1}\; a_V^r \; c_{A,V} \; c_{V,A}\; (\ev_V \otimes\ide_{V \otimes A})(\ide_{V^*}\otimes c_{V,V} \otimes \ide_A)(\ide_{V^* \otimes V} \otimes a_V^r \otimes \ide_A)\\ \medskip
& \quad  \quad \circ (\coevr_V \otimes \ide_V \otimes q), \\ \medskip  
&\overset{\text{$V$ \hspace{-.05in} local}}{=} d^{-1}\; a_V^r \;  (\ev_V \otimes\ide_{V \otimes A})(\ide_{V^*}\otimes c_{V,V} \otimes \ide_A)(\ide_{V^* \otimes V} \otimes a_V^r \otimes \ide_A)(\coevr_V \otimes \ide_V \otimes q), \\ \medskip  
&\overset{\text{$\cC$ \hspace{-.05in} ribbon}}{=} d^{-1}\; a_V^r \; (\ide_V \otimes \evr_V \otimes \ide_A)(c_{V,V} \otimes \ide_{V^* \otimes A}) (\ide_V \otimes \coev_V \otimes \ide_A)(\ide_V \otimes q), \\
\medskip   
&\overset{\text{$V$ \hspace{-.05in} unital}}{=} d^{-1}\; a_V^r(\ide_V\otimes [(\evr_V\otimes \ide_A)(a_V^r\otimes c_{A,{}^*V})(\ide_V\otimes q\otimes \ide_{{}^*V})])(c_{V,V}\otimes \ide_{V^*})\\
\medskip
& \quad \quad  \circ (\ide_V\otimes [(a_V^r\otimes \ide_V^*)(\ide_V\otimes c^{-1}_{A,V^*})(\coev_V\otimes\ide_A)])(\ide_V \otimes u_A)\\ \medskip 
&=a_V^r(\ide_V\otimes \widehat{\evr}_V^A)(c_{V,V}\otimes \ide_{V^*})(\ide_V\otimes \widehat{\coev}_V^A)(\ide_V \otimes u_A)\\
&=\widehat{\theta}_V^r.
\end{align*}
This implies that $\theta_V^l = \theta_V^r$, as required.
\end{proof}

For the proof of Theorem~\ref{thm:locmodular}, it remains to show that $\locmod_\cC(A)$ is non-degenerate. In fact, the following result holds.

\begin{proposition}\label{prop:locmod-nondeg}
Let $A$ be a rigid Frobenius algebra in a non-degenerate braided finite tensor category $\cC$. Then,  $\locmod_\cC(A)$ is non-degenerate.
\end{proposition}

\begin{proof}
By Corollary~\ref{cor:finite-tensor}, $\Rep_{\cC}(A)$ is a finite tensor category. Now $\cZ(\Rep_{\cC}(A))$ is  factorizable, and therefore non-degenerate, by \cite{Shi1}*{Theorem 1.1}. Next, the equivalent category $\locmod_{\cZ(\cC)}(A^+)$ is also non-degenerate due to Theorem~\ref{thm:local-center}. Then, by the braided tensor equivalence in Lemma~\ref{lem:A+local}, we obtain that $\locmod_\cC(A)\boxtimes \ov{\cC}$ is non-degenerate. Finally, we conclude that $\locmod_\cC(A)$ is non-degenerate by Lemma~\ref{lem:DeligneMueger}.
\end{proof}

\begin{remark} \label{rem:thm-vs-ss}
The main difference between the proof of  Theorem~\ref{thm:locmodular} and the semisimple version, \cite{KO}*{Theorem~4.5}, is the verification of non-degeneracy: we showed that the M\"uger center is trivial, whereas \cite{KO} showed that the S-matrix is invertible (which cannot be generalized to the non-semisimple case). The other aspects of $\locmod_\cC(A)$ being modular follow somewhat similarly to the semisimple case, but we provided more details in our arguments above for the reader's convenience. Moreover, ribbonality is achieved in a different way: via pivotality and Proposition~\ref{prop:ribbon-left-right}.
\end{remark}


\subsection{Invariants of categories of local modules} \label{sec:FP-results}

We can now compute the Frobenius--Perron dimension of categories of local modules. This generalizes \cite{DMNO}*{Lemma~3.11 and Corollary~3.32} to the non-semisimple case.

\begin{corollary}\label{cor:locmod-FPdim}
Let $A$ be a rigid Frobenius algebra in a braided finite tensor category $\cC$. Then 
$$\FPdim(\Rep_\cC(A))=\frac{\FPdim(\cC) }{\FPdim_{\cC}(A)}.$$
If $\cC$ is  non-degenerate, then 
$$\FPdim \big(\locmod_\cC(A)\big)=\frac{\FPdim(\cC)}{\FPdim_{\cC}(A)^2}.$$
\end{corollary}
\begin{proof}
The forgetful functor $F\colon \Rep_\cC(A)\to \cC$ is  right adjoint to the faithful monoidal functor $U$ in Lemma \ref{indrep}. Thus, $U$ is dominant, and by \cite{EGNO}*{Lemma 6.2.4}, we find that 
$$\FPdim(\Rep_\cC(A))=\frac{\FPdim_{\Rep_\cC(A)}(X)}{\FPdim_{\cC}(F(X))}\FPdim(\cC),$$
for any object $X$ in $\Rep_\cC(A)$. Specifying $X=A$, which is the tensor unit of $\Rep_\cC(A)$, and thus has FP-dimension $1$
by \cite{EGNO}*{Proposition 3.3.6(1)}. We obtain the first claimed formula.

Next, by Lemma~\ref{lem:A+local} we get that
$\locmod_{\cZ(\cC)}(A^+)\simeq \locmod_\cC(A)\boxtimes \ov{\cC}.$
Thus, 
$$\FPdim \big(\locmod_\cC(A)\big)=\frac{\FPdim \big(\locmod_{\cZ(\cC)}(A^+)\big)}{\FPdim (\cC)}=\frac{\FPdim(\Rep_\cC(A))^2}{\FPdim (\cC)}=\frac{\FPdim(\cC)}{\FPdim_{\cC}(A)^2}\, ,$$
where the second equality uses Theorem~\ref{thm:local-center} (\cite{Sch}*{Corollary 4.5}) and \cite{EGNO}*{Theorem~7.16.6}, and the third equality follows from the first claimed equation.
\end{proof}

\begin{remark} \label{rem:FPbound}
We obtain a restriction on the possible FP-dimensions of rigid Frobenius algebras in $\cC$. First, $\FPdim(\locmod_\cC(A))\geq 1$ by \cite{EGNO}*{Proposition 3.3.4}; hence, Corollary \ref{cor:locmod-FPdim} implies that $\FPdim_\cC(A)^2\leq \FPdim(\cC)$.  Secondly, if there exists a (quasi)-fiber functor $F\colon \cC\to \Vect$, then $\FPdim_\cC(A)=\dim_\Bbbk(F(A))$ \cite{EGNO}*{Proposition 4.5.7}.
\end{remark}


\section{Connections to relative monoidal centers}\label{sec:rel-center}

In this section, we illustrate our main theorem, Theorem~\ref{thm:locmodular}, for local modules over commutative algebras in relative monoidal centers. We first discuss a generalization of the material in Section~\ref{subsec:RepCA} in the case when $\cC$ is a finite tensor category, but  not necessary braided. Next, we recall the background for relative monoidal centers in Section~\ref{sec:rel-back}. We then present the main result of this part, Theorem~\ref{thm:ZRepA}, in Section~\ref{sec:rel-main}; this is a relative analogue of Schauenburg's result, Theorem~\ref{thm:local-center}. This gives us various ways to produce new modular categories from old ones. The main result is then used to study categories of Yetter-Drinfeld modules over braided tensor categories in Section~\ref{sec:YDB}.

\subsection{Local modules over central algebras}
In Definition \ref{def:RepCA}, we defined the category $\Rep_{\cC}(A)$ for a commutative algebra $A$ in a \emph{braided} finite tensor  category. Following \cites{Sch,DMNO} we can relax the requirement of $\cC$ being braided.

\begin{definition}\label{def:AinZC}
Let $\cC$ be a finite tensor category and $(A, c_{A,-})$ a commutative algebra in $\cZ(\cC)$. The underlying object $A$ is an algebra in $\cC$. We refer to $(A, c_{A,-})$ as a \emph{central algebra in $\cC$}. If, further, $(A,c_{A,-})$ is a rigid Frobenius algebra in $\cZ(\cC)$, then we  refer to the central algebra $(A,c_{A,-})$ as a {\it rigid Frobenius central algebra in $\cC$}.
\end{definition}

\begin{definition}\label{def:RepCA-general}
Let $(A, c_{A,-})$ be a central algebra in $\cC$.
We denote by $\Rep_\cC(A)$ the category $\rmod{A}(\cC)$ of right $A$-modules in $\cC$ with monoidal structure given as follows. Given a right $A$-module $(V,a^r_V)$, define the left action $a^l_V=a^r_Vc_{A,V}$ as in \eqref{def:a^l}. This way, $(V,a^l_V,a^r_V)$ defines an object in $\bimod{A}(\cC)$, and $\Rep_{\cC}(A)$ is monoidal with tensor product $\otimes_A$ and unit object $A$ as in Definition~\ref{def:RepCA}.
\end{definition}

\begin{remark} We note that if $(\cC,c)$ is braided and $A$ is a commutative algebra in $\cC$, then $(A^+,(c_{A,-})^+)$ is a commutative algebra in $\cZ(\cC)$ by Lemma~\ref{lem:A+}, and the definition of $\Rep_\cC(A^+)$ in Definition \ref{def:RepCA-general} recovers the monoidal category $\Rep_\cC(A)$ from Definition~\ref{def:RepCA}.
\end{remark}

\begin{remark}
We observe that the following results from Section~\ref{sec:local-mod} generalizes to using $\Rep_\cC(A)$ as defined in Definition \ref{def:RepCA-general}, i.e., without assuming that $\cC$ is braided, with the same proof.
\begin{enumerate}
\item Lemma~\ref{indrep}, verbatim. 
\smallskip
\item Theorem~\ref{thm:local-center} (or, \cite{Sch}*{Corollary 4.5}): there exists a braided equivalence between
$\cZ(\Rep_\cC(A))$ and $\locmod_{\cZ(\cC)}(A,c_{A,-}),$
for a central algebra $(A,c_{A,-})$ in $\cC$.
\smallskip
\item Lemma~\ref{lemma:rigid}: Let $\cC$ be a rigid (resp. pivotal) monoidal category, and take $A$ a rigid Frobenius central algebra in $\cC$. Then, $\Rep_\cC(A)$ is a rigid (resp. pivotal) monoidal category.
\smallskip
\item Corollary~\ref{cor:finite-tensor}: If $\cC$ is a finite tensor category and $A$ is a rigid Frobenius central algebra in~$\cC$, then $\Rep_\cC(A)$ is a finite tensor category.
\end{enumerate} 
\end{remark}

\subsection{Relative monoidal centers} \label{sec:rel-back}
In this section, we summarize the setup required for relative monoidal centers from \cites{LW,LW2}. 
Let $\cB$ be a braided finite tensor category, with mirror $\overline{\cB}$,  throughout this section.

\begin{definition} \label{def:Bcentral}
A finite tensor category $\cC$ is {\it $\cB$-central} if there exists a faithful braided tensor functor $G\colon \overline{\cB} \to \cZ(\cC)$.  
In this case, we refer to the functor $G$ as {\it $\cB$-central} as well.
\end{definition}

\begin{definition} \label{def:relcenter}
Given a $\cB$-central finite tensor category $\cC$, we define the \emph{relative monoidal center} $\cZ_\cB(\cC)$ to be the braided monoidal full subcategory consisting of objects $(V,c)$ of $\cZ(\cC)$, where $V$ is an object of $\cC$, and the half-braiding $c:=\{c_{V,X}\colon  V \otimes X \isomorph X \otimes V\}_{X \in \cC}$ is a natural isomorphism satisfying the two conditions below:
\begin{enumerate}
\item[(i)] [tensor product compatibility] \; $c_{V,X \otimes Y} = (\ide_X \otimes c_{V,Y})(c_{V,X} \otimes \ide_Y)$, for $X,Y \in \cC$.
\smallskip
\item[(ii)] [compatibility with $\cB$-central structure] \; $c_{G(B),V} \circ c_{V,G(B)} = \ide_{V \otimes G(B)}$, for any $B \in \overline{\cB}$.
\end{enumerate}
That is, $\cZ_\cB(\cC)$ is the full subcategory of $\cZ(\cC)$ of all objects that centralize $G(B)$ for any object $B$ of $\ov{\cB}$, the {\it M\"{u}ger centralizer} of $G(\overline{\cB})$ in $\cZ(\cC)$. 
\end{definition}

\begin{proposition} \label{prop:ZBA-bftc}
\cite{LW2}*{Proposition~4.9}  Given a $\cB$-central finite tensor category, we have that $\cZ_\cB(\cC)$ is a braided finite tensor subcategory  of $\cZ(\cC)$.  \qed
\end{proposition}

\begin{example}\label{expl:YDB}
Given $H$ a Hopf algebra in $\cB$, we have that $\cC=\lmod{H}(\cB)$ is a $\cB$-central finite tensor category \cite{L18}*{Example 3.17}. Here, the braided monoidal functor $G\colon \ov{\cB}\to \cZ(\cC)$ by sending $V\in \ov{\cB}$ to $((V,a_V^{\triv}),\psi^{-1}_{-,V})$,  where 
$a_V^\triv:=\varepsilon\otimes \ide_V\colon H \otimes V \to V$
is the trivial $H$-action on $V$ and $\psi$ is the braiding of $\cB$. Moreover, we have an equivalence of braided monoidal categories 
$$\cZ_\cB(\cC) \; \overset{br.\otimes}{\simeq} \; \lYD{H}(\cB),$$
see \cite{L18}*{Proposition 3.36}, where $\lYD{H}(\cB)$ is the braided tensor category of Yetter-Drinfeld modules over $H$ in $\cB$.
\end{example}

\begin{example}\label{expl:br-Drin}
As a special case of Example~\ref{expl:YDB}, let $\cB=\lmod{K}$ for a finite-dimensional quasi-triangular Hopf algebra $K$. Moreover, take $H$ a finite-dimensional Hopf algebra in $\cB$ with dual $H^*$ (as in dually paired Hopf algebras \cite{L17}*{Definition 3.1}). Then, $$\cC :=\lmod{H}(\lmod{K})\tensimeq \lmod{H\rtimes K},$$ and there is an equivalence of braided monoidal categories,
$$\cZ_\cB(\cC) \; \overset{br.\otimes}{\simeq}\; \lYD{H}(\cB) \; \overset{br.\otimes}{\simeq}\; \lmod{\Drin_K(H,H^*)}.$$
Here, $\Drin_K(H,H^*)$ is a quasi-triangular Hopf algebra called the \emph{braided Drinfeld double} of $H$. It is due to \cite{Maj99} where it is referred to as the \emph{double bosonization}. For details, including a presentation of $\Drin_K(H,H^*)$, see \cite{L17}*{Section 3.2}. 
\end{example}

Next, recall the main result of \cite{LW2}, which provides sufficient conditions for $\cZ_\cB(\cC)$ to be a modular tensor category, building on earlier work of \cites{KR93,Shi2}.

\begin{theorem}\cite{LW2}*{Theorem~4.14, Corollary~4.16}  \label{thm:ZBCmodular}
 Let $\cB$ be a non-degenerate braided finite tensor category. Let $\cC$ be a $\cB$-central finite tensor category such that the full image $G(\ov{\cB})$ is a topologizing subcategory of $\cZ(\cC)$.
Then,  the relative monoidal center $\cZ_\cB(\cC)$ is a modular tensor category if the set $\mathsf{Sqrt}_\cC(D,\xi_{D})$ from \cite{LW2}*{Definition~3.3} is non-empty. \qed
 \end{theorem}
 
Here, $D$ is the distinguished invertible object of $\cC$, and $\xi_D(-)$ is a certain natural isomorphism of the form $D \otimes (-) \to (-)^{4*} \otimes D$. Moreover, $\mathsf{Sqrt}_\cC(D,\xi_{D})$ are the equivalence classes of pairs $(V, \sigma_V)$, for $V \in \cC$ and $\sigma_V$ a certain natural isomorphism of the form $V \otimes (-) \to (-)^{**} \otimes V$, that `squares' to $(D,\xi_D)$. A special case of when $\cC$ satisfies the condition that $\mathsf{Sqrt}_\cC(D,\xi_{D}) \neq \varnothing$ is when $\cC$ is spherical in the sense of \cite{DSS}; see \cite{LW2}*{Definition 3.9 and Remark~3.11}.
 Now by applying Theorem~\ref{thm:locmodular}, we achieve the following result.
 
\begin{corollary} \label{cor:rel-modular1}
Take $\cC$ to be a $\cB$-central finite tensor category such that the full image $G(\ov{\cB})$ is a topologizing subcategory of $\cZ(\cC)$, with $\cB$ a non-degenerate finite tensor category. Assume that either
\begin{itemize}
    \item the set $\mathsf{Sqrt}_\cC(D,\xi_{D})$ from \cite{LW2}*{Definition~3.3} is non-empty, or 
    \item $\cC$ is spherical in the sense of \cite{DSS}.
\end{itemize}
Then, for $A$ a rigid Frobenius algebra in $\cZ_\cB(\cC)$, we obtain that  $\locmod_{\cZ_\cB(\cC)}(A)$ is a modular  tensor category. \qed
\end{corollary}

In the next subsection, we will show that the categories of local modules appearing in Corollary~\ref{cor:rel-modular1}
are themselves relative centers.


\subsection{On categories of local modules in relative monoidal centers} \label{sec:rel-main}

For this section, assume the following conditions.

\begin{notation} \label{not:locmod-relcenter} 
Let $\cC$ be a $\cB$-central finite tensor category, and consider the braided finite tensor category $(\cZ_\cB(\cC), c)$ [Proposition~\ref{prop:ZBA-bftc}]. Let $(A,c_{A,-})$ be a  commutative algebra in $\cZ_\cB(\cC)$, i.e. a central algebra in $\cC$, cf. Definition~\ref{def:AinZC}. We denote by $A$ the image of $(A,c_{A,-})$ under the forgetful functor $F\colon \cZ_\cB(\cC) \to \cC$.
\end{notation}

Towards generalizing Theorem~\ref{thm:local-center} to the relative setting, we first establish the following result.

\begin{proposition}\label{prop:RepA-Baug} We obtain that the category $\Rep_\cC(A)$ is a $\cB$-central finite tensor category.
\end{proposition}

\begin{proof}
Denote the $\cB$-central functor of $\cC$ by  $G\colon \ov{\cB}\to \cZ(\cC), \; B\mapsto (G(B),c_{G(B),-}).$
Define the functor 
\begin{equation} \label{eq:GA}
    G_A\colon \ov{\cB}\to \cZ(\Rep_\cC(A))
\end{equation} by sending $B$ to the pair $(G_A(B),c_{G_A(B),-})$, where $G_A(B)=(G(B)\otimes A,a^r_{G_A(B)})$, with $$a^r_{G_A(B)}=\ide_{G(B)}\otimes m_A.$$ 
From this, we get that 
$$a^l_{G_A(B)} = a^r_{G_A(B)}c_{A,G_A(B)} = (\ide_{G(B)}\otimes m_A)c_{A,G_A(B)}.$$
Moreover,  for $V \in \Rep_{\cC}(A)$, the morphism 
$$c_{G_A(B),V}\colon G_A(B)\otimes_A V\to V\otimes_AG_A(B)$$ 
is induced by
\begin{align*}
c'_{G_A(B),V}:=(c_{G(B),V}\otimes\ide_A)(\ide_{G(B)}\otimes c_{A,V})\colon G_A(B) \otimes V\to V  \otimes G_A(B).
\end{align*}

To check that $G_A$ is well-defined, we first need to show that $c_{G_A(B),V}$ defines a morphism in $\Rep_\cC(A)$. Consider the following computation: 
\begin{align*}
c'_{G_A(B),V} (a^r_{G_A(B)}\otimes \ide_V)&=(c_{G(B),V}\otimes\ide_A)(\ide_{G(B)}\otimes c_{A,V})(\ide_{G(B)}\otimes m_A\otimes \ide_V)\\
&=(c_{G(B),V}\otimes m_A)(\ide_{G(B)}\otimes c_{A,V}\otimes \ide_A)(\ide_{G(B)\otimes A}\otimes c_{A,V})\\
&=(\ide_{V\otimes G(B)}\otimes m_A)(\ide_V\otimes c_{A,G(B)\otimes A}\;c_{G(B)\otimes A,A})c_{G(B)\otimes A \otimes A,V}\\
&=(\ide_V\otimes a^l_{G_A(B) })
(\ide_V\otimes c_{G(B)\otimes A,A})c_{G(B)\otimes A \otimes A,V}\\
&=(a^r_V\otimes \ide_{G(B)\otimes A})
(\ide_V\otimes c_{G(B)\otimes A,A}) c_{G(B)\otimes A \otimes A,V}\\
&=c'_{G_A(B),V}(\ide_{G_A(B)}\otimes a^r_Vc_{A,V})\\
&=c'_{G_A(B),V}(\ide_{G_A(B)}\otimes a^l_V).
\end{align*}
The first, second, fourth, sixth, and seventh equalities follow from definition or by naturality. The third equality uses Definition~\ref{def:relcenter}(ii) 
applied to $G(B)\otimes A$ 
and naturality.
The fifth equality holds using that the target space is $V\otimes_A G_A(B)$.
A very similar computation shows that $c_{G_A(B),V}$ commutes with the right $A$-module structures and hence defines a morphism in $\Rep_{\cC}(A)$. 

Now, we check that $G_A$ is well-defined. As $c_{G_A(B),V}$ is defined using natural transformations in $V$, it is itself natural in $V$. Further, tensor compatibility of $c_{G_A(B),-}$ is inherited from tensor compatibility of the half-braidings $c_{A,-}$ and $c_{G(B),-}$. Thus, $(G_A(B), c_{G_A(B),-})$ defines an object in $\cZ(\cC)$. The assignment extends to morphisms by setting  $G_A(f)=G(f)\otimes \ide_A$, and thus defines a functor.

Next, we verify that $G_A$ defines a monoidal functor $\ov{\cB}\to \cZ(\Rep_\cC(A))$. We define the structural natural transformation $\mu^{G_A}$ for two objects $B,C$ in $\ov{\cB}$ by
$$\mu_{B,C}^{G_A}:=(\mu^G_{B,C} \otimes m_A)(\ide_{G(B)}\otimes c_{A,G(C)}\otimes \ide_A)\colon G_A(B) \otimes  G_A(C) \to G_A(B \otimes C),$$
where $\mu^G$ is the natural isomorphism of the monoidal functor $G\colon \ov{\cB}\to \cZ(\cC)$. It follows using associativity of $A$ that $\mu_{B,C}^{G_A}$ gives a morphism from $G_A(B)\otimes_AG_A(C)$ to $G_A(B\otimes C)$ in $\Rep_\cC(A)$. The natural transformation obtained has an inverse induced by
$$
\left(\mu^{G_A}_{B,C}\right)^{-1}:=(\ide_{G(B)}\otimes u_A\otimes \ide_{G (C)\otimes A})\left(\left(\mu_{B,C}^{G}\right)^{-1}\otimes \ide_A\right).
$$

Next, we check compatibility of $G_A$ with the inverse braiding of $\cB$.
{\small
\begin{align*}
\big(\mu^{G_A}_{C,B}\big)^{-1}G_A(c_{C,B}^{-1})\mu^{G_A}_{B,C}
&=(\ide_{G(C)}\otimes u_A\otimes \ide_{G(B) \otimes A})(\mu^{G}_{C,B}G(c^{-1}_{C,B})\mu^{G}_{B,C}\otimes m_A)(\ide_{G(B)}\otimes c_{A,G(C)}\otimes \ide_{A})\\
&=(\ide_{G(C)}\otimes u_A\otimes \ide_{G(B) \otimes A})(c_{G(B),G(C)}\otimes m_A)(\ide_{G(B)}\otimes c_{A,G(C)}\otimes \ide_{A})\\
&=(\ide_{G(C)}\otimes u_A\otimes \ide_{G(B) \otimes A})(c_{G(B),G(C)}\otimes m_A c_{A,A})(\ide_{G(B)}\otimes c_{A,G(C)}\otimes \ide_{A})\\
&=(\ide_{G(C)}\otimes u_A\otimes \ide_{G(B)}\otimes m_A)(\ide_{G(C)} \otimes c_{A,G(B)}\otimes \ide_{A})c_{G_A(B),G_A(C)}\\
&=(\ide_{G(C)}\otimes u_A\otimes \ide_{G(B)}\otimes m_Ac_{A,A})(\ide_{G(C)} \otimes c_{A,G(B)}\otimes \ide_{A}) c'_{G_A(B),G_A(C)}\\
&=(\ide_{G(C)}\otimes u_A\otimes \ide_{G(B)}\otimes \ide_A)a^l_{G_A(B)} c'_{G_A(B),G_A(C)}\\
&=(\ide_{G(C)}\otimes m_A(u_A\otimes \ide_A)\otimes \ide_{G_A(B)}) c'_{G_A(B),G_A(C)}\\
&= c'_{G_A(B),G_A(C)}.
\end{align*}
}

\noindent Here, the first equality follows from the definition of $\mu^{G_A}$ and its inverse, while the second equality uses that $G$ is compatible with the inverse braiding of $\cB$, and the third equality uses commutativity of $A$. The fourth equality uses  Definition~\ref{def:relcenter}(ii), since $A$ is in the relative center.
Next, we again use that $A$ is commutative, and then apply the definition of $a^l_{G_A(B)}$ in the sixth equality.
Next, we apply that the codomain is the relative tensor product $G_A(C)\otimes_A G_A(B)$
, followed by the unit axiom for $A$ in the last equality. Thus, $G_A\colon\ov{\cB}\to\cZ(\Rep_\cC(A))$ is a braided monoidal functor. 

As the tensor product in $\cB$ is faithful and $G$ is faithful, we conclude that $G_A(f)=f\otimes \ide_A$ is faithful on morphisms.
Finally, if $\cC$ is finite category, then $\Rep_{\cC}(A)$ is finite since as an abelian category it is $\rmod{A}(\cC)$ which is finite by \cite{EGNO}*{Exercise 7.8.16}.
\end{proof}



Now we prove a relative version of Theorem~\ref{thm:local-center}.

\begin{theorem}\label{thm:ZRepA}
Retain Notation~\ref{not:locmod-relcenter}. There is an equivalence of braided monoidal categories:
$$\locmod_{\cZ_\cB(\cC)}(A,c_{A,-})\isomorph \cZ_\cB(\Rep_\cC(A)).$$
\end{theorem}

\begin{proof}
Consider the composition $$\Phi\colon \locmod_{\cZ_\cB(\cC)}(A,c_{A,-})\hookrightarrow \Rep_{\cZ_\cB(\cC)}(A) \xrightarrow{F}\Rep_\cC(A)$$
of the inclusion and the functor $F$ induced by the forgetful functor $F\colon \cZ_\cB(\cC)\to \cC$.
Then $\Phi$ extends to a $\cB$-central functor, where for $(V,c_{V,-}) \in \cZ_\cB(\cC)$ we have the half-braiding for $\Phi(V)$ given by the  morphism $$c_{\Phi(V),X}\colon V\otimes_AX\to X\otimes_A V$$
induced by the half braiding $c_{V,X}$. Tensor compatibility of $c_{\Phi(V),-}$ is inherited from tensor compatibility of $c_{V,-}$. Since $(V,c_{V,-})$ is an object in the relative center, it follows that $$c_{G(B),\Phi(V)}\;c_{\Phi(V),G(B)}=\ide_{\Phi(V) \otimes G(B)},$$
for any object $B\in \ov{\cB}$, and hence $\Phi$ is, indeed, a functor 
$$\Phi\colon \locmod_{\cZ_\cB(\cC)}(A,c_{A,-})\to \cZ_\cB(\Rep_\cC(A)).$$
Note that here $\Rep_\cC(A)$ is $\cB$-central monoidal using Proposition \ref{prop:RepA-Baug}. 
It is also clear that $\Phi$ is a monoidal functor, as it is the a composition of monoidal functors with structural isomorphism $\Phi_{V,W}$ given by identities.

To construct an inverse functor $\Psi$ to $\Phi$, we note the following isomorphisms of right $A$-modules, for $(V,a_V^r)$ in $\Rep_\cC(A)$ and $X\in \cC$.
\begin{align*}
\medskip
  \phi_{X,V} := \ide_X \otimes a_V^rc_{A,V} \colon  &(X\otimes A)\otimes_A V\longrightarrow X\otimes V\\
  \psi_{V,X} := (a_V^r\otimes \ide_X)(\ide_V\otimes c^{-1}_{A,X}) \colon   &V\otimes_A(X\otimes A)\longrightarrow V\otimes X.
\end{align*}
Assume given an object $(M,c_{M,-})$ in $\cZ_\cB(\Rep_\cC(A))$, define for any $X\in \cC$,  the isomorphism
$$c'_{M,X}:=\phi_{X,M}\;c_{M,X\otimes A}\;\psi_{M,X}^{-1}.$$
As a composition of natural transformations in $X$, we see that $c'_{M,-}$ defines a natural transformation. Further, $c'_{M,-}$ is tensor compatible since $c_{M,-}$ is tensor compatible. Similarly, for any object $B$ in $\ov{\cB}$, we have that $c_{G_A(B),M}\; c_{M,G_A(B)}=\ide_{M \otimes_A G_A(B)}$ since $M$ is in the relative center of $\Rep_\cC(A)$. This directly implies that $c'_{G(B),M}c'_{M,G(B)}=\ide_{M \otimes G(B)}$ showing that $(M,c'_{M,-})$ defines an object in $\cZ_\cB(\cC)$. In particular,  $(M,c'_{M,-})$ is an object in $\Rep_{\cZ_\cB(\cC)}(A,c_{A,-})$.
Next, we check that $(M,c'_{M,-})$ defined above is a local module via the following computations:
\[
    a_{M}^rc_{A,M} 
    c'_{M,A} =a_M^r c_{A,M} c_{M,A}
    =a_{M}^r.
\]
In the first equality, we use the definition of $c'_{M,A}$ and $\psi_{M,A}^{-1}=\ide_{M\otimes A}\otimes u_A$, along with unitality and naturality.
The second equality uses Definition~\ref{def:relcenter}(ii), Proposition~\ref{prop:RepA-Baug} and \eqref{eq:GA} with the fact that $G_A(\one) \cong  A$ (namely, $G(\one) \cong \one$ as $G$ is a tensor functor). This way, we get that $(M,c'_{M,-})$ is an object in $\locmod_{\cZ_\cB(\cC)}(A)$.

Now it is easily checked that $$\Psi\colon \cZ_\cB(\Rep_\cC(A))\longrightarrow \locmod_{\cZ_\cB(\cC)}(A,c_{A,-}), \quad \Psi(M,c_{M,-})=(M,c'_{M,-}),$$
defined as the identity on morphism spaces, gives a functor as desired. We can use identities as structural isomorphisms, and $c_{M,N}$ to make $\Psi$ a monoidal functor which is clearly braided. 

Finally, we check that $\Phi$ and $\Psi$ are mutally inverse functors. Let $(M,c)$ be an object in $\cZ_\cB(\Rep_\cC(A))$, and consider $\Phi \Psi(M,c)=(M,d)$. We compare the half-braidings $d$ and $c$. The half-braiding $d$ is the restriction to relative tensor products $\otimes_A$ of the morphism $c'$ considered above. Hence, for $X$ in $\Rep_\cC(A)$, 
\begin{align*}
    d_{M,X}&=(\ide_X\otimes a_M^rc_{A,M})c_{M,X\otimes A}(\ide_{M\otimes X}\otimes u_A)\\
    &=(\ide_X\otimes a_M^l)c_{M,X\otimes A}(\ide_{M\otimes X}\otimes u_A)\\
    &=(a_X^r\otimes \ide_M)c_{M,X\otimes A}(\ide_{M\otimes X}\otimes u_A)\\
    &=c_{M,X}(\ide_M\otimes a^r_X)(\ide_{M\otimes X}\otimes u_A)\\
    &=c_{M,X}.
\end{align*}
This shows that $\Phi \Psi (M,c)=(M,c)$.

Now, let $(V,c)$ be an object in $\locmod_{\cZ_\cB(\cC)}(A,c_{A,-})$ and consider $\Psi\Phi(V,c)=(V,d)$. Again, we compare the half-braidings $d$ and $c$, this time for $X$ an object in $\cC$:
\begin{align*}
    d_{V,X}&=(\ide_X\otimes a_V^rc_{A,V})\overline{c}_{V,X\otimes A}(\ide_{V\otimes X}\otimes u_A)\\
    &=(\ide_X\otimes a_V^rc_{A,V})c_{V,X\otimes A}(\ide_{V\otimes X}\otimes u_A)\\
    &=(\ide_X\otimes a_V^rc_{A,V}(u_A\otimes \ide_V))c_{V,X\otimes A}\\
    &=(\ide_X\otimes a_V^r( \ide_V\otimes u_A))c_{V,X}\\
    &=c_{V,X}
\end{align*}
Here, $\ov{c}_{M,X\otimes A}$ denotes the induced morphism from $c_{M,X\otimes A}$ to relative tensor products $\otimes_A$. However, the map $\ide_{V\otimes X}\otimes u_A$ lifts along the quotient map $V\otimes_A(X\otimes A)\to V\otimes(X\otimes A)$, giving the first equality. The second equality uses naturality of $\ov{c}_{M,-}$ with respect to morphisms in $\cC$. The third equality uses that $u_A$ is a morphism in $\cZ_\cB(\cC)$, followed by the unit axiom. Hence $\Psi\Phi(V,c)=(V,c)$, which completes the proof.
\end{proof}

\begin{corollary} \label{cor:local-mod}
Take $\cC$ to be a $\cB$-central finite tensor category such that the full image $G(\ov{\cB})$ is a topologizing subcategory of $\cZ(\cC)$, with $\cB$ a non-degenerate finite tensor category. Assume that either
\begin{itemize}
    \item the set $\mathsf{Sqrt}_\cC(D,\xi_{D})$ from \cite{LW2}*{Definition~3.3} is non-empty, or 
    \item $\cC$ is spherical in the sense of \cite{DSS}.
\end{itemize} 
Then, if $(A,c_{A,-})$ a central algebra in $\cC$, so that $(A,c_{A,-})$ is rigid Frobenius in $\cZ_\cB(\cC)$, we obtain that $\cZ_\cB(\Rep_\cC(A))$ is a modular tensor category.
\end{corollary}

\begin{proof}
This follows from Theorem~\ref{thm:ZRepA} and Corollary~\ref{cor:rel-modular1}. 
The category $\cZ_\cB(\Rep_\cC(A))$ obtains a ribbon structure through the equivalence from Theorem~\ref{thm:ZRepA}.
\end{proof}


\subsection{Examples for Yetter-Drinfeld categories over braided tensor categories}
\label{sec:YDB}

Recalling Examples~\ref{expl:YDB} and~\ref{expl:br-Drin}, we compute in this part the category of local modules over a Hopf algebra $H$ in a category of Yetter-Drinfeld modules $\lYD{H}(\cB)$, for $\cB$ a braided finite tensor category. The main result here is Proposition~\ref{prop:local-triv}.

\smallskip

To start, note that the result below is an immediate consequence of work of Brugui\`{e}res-Natale.

\begin{proposition}\label{prop-BNappl}
Recall Notation~\ref{not:locmod-relcenter}, and assume that  the forgetful functor $F\colon \cZ_\cB(\cC)\to \cC$ has a right adjoint $R_\cB$ which is faithful and exact. Then, $(R_\cB(\one), c_{R_\cB(\one),-})$ is a commutative algebra in $\cZ_\cB(\cC)$, and there is an equivalence of monoidal categories 
$$\Rep_{\cZ_\cB(\cC)}(R_\cB(\one), c_{R_\cB(\one),-})\;\overset{\otimes}{\simeq}\; \cC.$$
\end{proposition}

\begin{proof}
The  result follows from \cite{BN11}*{Proposition~6.1}. 
\end{proof}

\begin{example}\label{expl:R1}
Take the setting of Example~\ref{expl:YDB} for $\cZ_\cB(\cC) \simeq \lYD{H}(\cB)$ with $H$ a Hopf algebra in $\cB$. Let $R_\cB$ be the right adjoint to the forgetful functor $F: \cZ_\cB(\cC)\to \cC$, which is faithful and exact. 
Then,  $R_\cB(\one)$ is the algebra $H$ with the structure of a commutative algebra in  $\lYD{H}(\cB)$ \cite{LW}*{Example~3.17} (see \cite{CFM}*{page~1332}, in the case of $\cB = \Vect$). Now by Proposition~\ref{prop-BNappl}, 
\begin{equation} \label{eq:loc-YDH}
    \Rep_{\lYD{H}(\cB)}(H, c_{H,-})\;\overset{\otimes}{\simeq}\; \lmod{H}(\cB) = \cC.
\end{equation}
\end{example}

\begin{proposition} \label{prop:local-triv}
The monoidal equivalence \eqref{eq:loc-YDH} above restricts to an equivalence of braided tensor categories
$$\locmod_{\lYD{H}(\cB)}(H, c_{H,-})\;\overset{\text{br.}\otimes}{\simeq}\; \cB.$$
\end{proposition}
\begin{proof}
Set the notation $A:= (H,c_{H,-})$, an algebra in $\cZ_\cB(\cC)$.
Consider the monoidal equivalence $R_\cB\colon \lmod{H}(\cB) \to \Rep_{\lYD{H}(\cB)}(A)$ from \eqref{eq:loc-YDH}, which restricts to a monoidal functor
$$\Gamma \colon \cB\xrightarrow{G}\lYD{H}(\cB)\xrightarrow{F}\lmod{H}(\cB) \xrightarrow{R_\cB}  \Rep_{\lYD{H}(\cB)}(A)$$
obtained by composition. Here, $G$ is the $\cB$-central functor for $\lYD{H}(\cB)$, and $F$ is the forgetful functor. Note here that as monoidal categories, $\cB$ and $\ov{\cB}$ are identical and only distinguished by their braiding. Recall that, for an object $X$ in $\cB$, $FG(X)=(X,a^l_X)$ is the left $H$-module with trivial $H$-action given by  $a^l_X=\varepsilon_H\otimes \ide_X$.

Since $\Gamma$ is a restriction of the equivalence \eqref{eq:loc-YDH} to full subcategories, it is fully faithful. Next, we show that $\Gamma$ is essentially surjective.

We observe that by Proposition \ref{prop-BNappl}, every object in $\Rep_{\lYD{H}(\cB)}(A)$ is of the form $R_\cB(V)$ for an object $V$ in $\cC=\lmod{H}(\cB)$. The right action of $A$ is given by the lax monoidal structure $\tau=(\tau_{V,W})_{V,W\in \lmod{H}(\cB)}$ (cf.  \cite{LW}*{Theorem~3.10}) via
\begin{align*}
a^r_{R_\cB(V)}=\tau_{V,\one}\colon R_\cB(V)\otimes A\to R_\cB(V).
\end{align*}
Recall the Yetter--Drinfeld module structure on $R_\cB(V)$ described in \cite{LW}*{Equation~(3.4)}. Note that, by definition of $\Gamma(X)$ we use this Yetter--Drinfeld module structure rather on $\Gamma(X)=R_\cB(FG(X))$, rather than the one obtained by the central functor $G$.
Using this Yetter--Drinfeld structure, we see that $\Gamma(X)$ is a local module for any object $X$ in $\cB$. Thus, the image of the functor $\Gamma$ is contained in the full subcategory of local modules.

Conversely, assume that $R_\cB(V)$ is a local module for some object $V\in \lmod{H}(\cB)$. Then we compute that the condition $$a^r_{R_\cB(V)}c_{H,R_\cB(V)}c_{R_\cB(V),H}=a^r_{R_\cB(V)}$$
implies that $a^l_V=\varepsilon_H \otimes \ide_V$ for the left action of $H$ on $V$. 
Thus, $V=FG(X)$, and the functor $\Gamma$ is essentially surjective onto the category $\locmod_{\lYD{H}(\cB)}(A)$.

   Therefore, $\Gamma$ is an equivalence of categories. Moreover, $\Gamma$ is a monoidal equivalence  by Proposition~\ref{prop-BNappl}, with the natural isomorphism $\mu^{\Gamma}_{X,Y}$ induced by 
$$\mu^{\Gamma}_{X,Y}:=\tau_{GF(X),GF(Y)}\colon \Gamma(X)\otimes \Gamma(Y)\to \Gamma(X\otimes Y).$$

It remains to show that $\Gamma$ is compatible with braidings. Indeed, we compute that
\begin{align*}
    \mu_{X,Y}^{\Gamma} c_{\Gamma(X),\Gamma(Y)}&=(m_H\otimes c_{X,Y})(\ide_H\otimes c_{X,H}\otimes \ide_Y)=\Gamma(c_{X,Y})\mu_{X,Y}^{\Gamma}.
\end{align*}
Thus, $\Gamma\colon \cB\to \locmod_{\lYD{H}(\cB)}(A)$ is a braided tensor equivalence.
\end{proof}


\section{Examples} \label{sec:examples}

In this section, we illustrate our main result, Theorem \ref{thm:locmodular}, for various rigid Frobenius algebras in non-semisimple modular tensor categories. In Section~\ref{sec:pos}, we consider examples in positive characteristic involving non-semisimple commutatative Hopf algebras. In  Section~\ref{sec:Z(kG-mod)}, we extend Davydov's classification of  connected \'etale algebras in $\cZ(\lmod{\Bbbk G})$ from \cite{Dav3} to arbitrary characteristic. The end of both Sections~\ref{sec:pos} and~\ref{sec:Z(kG-mod)}  includes further directions for consideration, and we discuss a conjecture pertaining to completely anisotropic categories and  Witt equivalence in Section~\ref{sec:ques}. To proceed, consider the notation below, which we will use throughout.

\begin{notation}[$A_g$]
For a group $G$ and an algebra $A$ in $\lcomod{\Bbbk G}$, set
$$A_g = \{a \in A ~|~ \delta_A(a) = g \otimes a\}.$$
Note that $A_{1_G}$ is the invariant subalgebra $A^{(\Bbbk G)^*}$ of $A$. 
\end{notation}

\subsection{Examples involving commutative Hopf algebras in positive characteristic}
\label{sec:pos}
In this part, let $\Bbbk$ be a field of characteristic $p>0$. 

\begin{proposition} \label{prop:ex-charp}
Let $K$ be a finite-dimensional, commutative, non-semisimple Hopf algebra over $\Bbbk$ such that its Drinfeld double, $\Drin(K)$, is ribbon. Let $L= \Bbbk N$, for $N$ a finite abelian group with $p \nmid  |N|$. Let $H$ denote the tensor product Hopf algebra $K \otimes L$ over $\Bbbk$. Then:
\begin{enumerate}[(1),font=\upshape]
     \item $\cZ(\lmod{H})$ is a non-semisimple modular tensor category;
    \smallskip
    \item $L$ is a rigid Frobenius algebra in $\cZ(\lmod{H})$;
    \smallskip
    \item $\locmod_{\cZ(\lmod{H})}(L) \simeq  \cZ(\lmod{K})$ as modular (i.e., ribbon)  categories;
    \smallskip
    \item $\FPdim(\locmod_{\cZ(\lmod{H})}(L)) = (\dim_\Bbbk K)^2$ \; and \; $\FPdim_{\cZ(\lmod{H})}(L) = \dim_\Bbbk L = |N|$.
\end{enumerate}
\end{proposition}

\begin{proof}
(1) By \eqref{eq:boxtimes}, \eqref{eq:Drin}, and Lemma \ref{lem:Z-Deligne}, we have that $$\cZ(\lmod{H}) \simeq \lmod{\Drin(K)} \; \boxtimes\;  \lmod{\Drin(L)}$$ as braided monoidal categories. By the assumptions on $K$ and $L$, both $\lmod{\Drin(K)}$ and $\lmod{\Drin(L)}$ are non-semisimple modular categories; see \cite{Rad}*{Theorems~13.2.1 and~13.7.3}. Thus, $\cZ(\lmod{H})$ is as well, and the equivalence above upgrades to an equivalence of such structures.

\smallskip

(2) To start, let us identify elements $\ell \in L$ with elements $1_K \otimes \ell \in H$. Next, it is straight-forward to check that $L = \Bbbk N$ is an algebra in $\lYD{H}$ via the adjoint action and regular coaction, 
$$h \cdot \ell = h_1 \; \ell\; S(h_2) = \varepsilon(h) \;\ell, \qquad \quad \delta_L(\ell) =  \ell_{-1} \otimes \ell_0 = \ell \otimes \ell \qquad \text{(Sweedler notation)},$$
for all $h \in H$, and $\ell \in N$ (and extending linearly to $L$). The commutativity of $H$ is used in computations here. Moreover, the braiding of $\lYD{H}$ is given by $c_{L,L}(\ell \otimes \ell') = (\ell_{-1} \cdot \ell') \otimes \ell_0$, for all $\ell, \ell' \in L$. Now, $m\; c_{L,L}(a \otimes b) = \varepsilon(a) b a = ba = m(a \otimes b)$, for all $a, b \in N$. Therefore, $L$ is a commutative algebra in $\lYD{H}$.

Now we show that $L$ is a connected algebra in $\lYD{H}$. We have that 
$$ 1 \; \leq \; \dim_\Bbbk \Hom_{\lYD{H}}(\Bbbk, L) \; =\; \dim_\Bbbk L^{\Drin(H)} \; \leq \; \dim_\Bbbk L^{H^*}.$$
The first inequality holds as $u_L \in \lYD{H}$. The second equality holds as the isomorphism, $\Hom_\Bbbk(\Bbbk, L) \cong L$, given by $f \mapsto f(1_\Bbbk)$, descends to $\Hom_{\lmod{\Drin(H)}}(\Bbbk, L) \cong L^{\Drin(H)}$. Since $H^*$ is a Hopf subalgebra of $\Drin(H)$, we get that  $L^{\Drin(H)}$ is a invariant subalgebra of $L^{H^*}$. Finally,
$$L^{H^*} = L_1 = \{\ell \in L ~|~ \ell_{-1} \otimes \ell_0 = 1_H \otimes \ell\} = \Bbbk 1_L.$$
So, $\dim_\Bbbk L^{H^*} = 1$, and thus, $\dim_\Bbbk \Hom_{\lYD{H}}(\Bbbk, L) = 1$. Hence, $L$ is a connected algebra in $\lYD{H}$.

Now we show that $L = \Bbbk N$ is a special Frobenius in $\lYD{H}$. Take
$$\textstyle \Delta_L(a) = \sum_{b \in N} b \otimes b^{-1}a \quad \text{and} \quad \varepsilon_L = \delta_{e_N,a} 1_\Bbbk, \quad \text{for all } a \in N.$$
It is straight-forward to check that $(L, \Delta_L, \varepsilon_L)$ is a Frobenius algebra over $\Bbbk$. Next, $\Delta_L$ and  $\varepsilon_L$ are maps in $\lYD{H}$ due to the following computations: For all $h \in H$ and $a \in N$,
\[
\begin{array}{rl}
 \medskip
 \Delta_L(h \cdot a) &=  \;\textstyle \sum_{b \in N} b \otimes b^{-1}(h \cdot a) \; = \; \varepsilon(h) \sum_{b \in N} b \otimes b^{-1}a\\
 \vspace{.15in}
 &=  \; \textstyle \sum_{b \in N} \varepsilon(h_1) b \otimes  \varepsilon(h_2) b^{-1}a  \;=  \;\sum_{b \in N} (h_1 \cdot b) \otimes  (h_2 \cdot b^{-1}a) \; =  \; h \cdot \Delta_L(a); \\ 
\textstyle  \varepsilon_L(h\cdot a) &=  \; \varepsilon(h)\delta_{e_N, a} 1_{\Bbbk} \;=\; h \cdot \varepsilon_L(a);  \vspace{.15in}\\
 \medskip
 (\ide_H \otimes \Delta_L)\lambda^H_L(a) &= (\ide_H \otimes \Delta_L)(a \otimes a) = \textstyle \sum_{b \in N} a \otimes b \otimes b^{-1}a =  \sum_{b \in N} b b^{-1} a \otimes b \otimes b^{-1}a \\
& = \lambda^H_{L \otimes L} \Delta_L(a); \vspace{.15in}\\
 (\ide_H \otimes \varepsilon_L)\lambda^H_L(a) &= (\ide_H \otimes \varepsilon_L)(a \otimes a) \;=\; \delta_{e_N, a}(a \otimes 1_{\Bbbk}) \;=\;  \delta_{e_N, a} \lambda^H_{\Bbbk}(1_{\Bbbk}) \;=\; \lambda^H_{\Bbbk} \varepsilon_L(a).
\end{array}
\]
Here, $\lambda^H_X$ denotes the left $H$-coaction on $X$. Thus, $(L,\Delta_L,\varepsilon_L)$ is a Frobenius algebra in $\lYD{H}$. Finally, $\varepsilon_L \; u = \ide_{\Bbbk}$ and
$m \; \Delta_L = |N| \ide_{\Bbbk}$. In particular, $|N| \neq 0$ as $p$ does not divide $|N|$. So, $L$ is a special Frobenius algebra, and thus, a rigid Frobenius algebra in $\lYD{H}$.

\smallskip 

(3) We continue to identify elements $\ell \in L$ with elements $1_K \otimes \ell \in H$. We then have the following natural equivalences of ribbon categories:
\[
\begin{array}{rl}
\medskip
\locmod_{\cZ(\lmod{H})}(L) 
&\simeq \locmod_{\cZ(\lmod{K}) \boxtimes \cZ(\lmod{L})}(\one_{\lmod{K}} \boxtimes L) \\
\medskip
&\simeq \locmod_{\cZ(\lmod{K})}(\one_{\lmod{K}}) \boxtimes \locmod_{\cZ(\lmod{L})}(L)\\
\medskip
&\simeq \cZ(\lmod{K}) \boxtimes \Vect\\
&\simeq \cZ(\lmod{K}).
\end{array}
\]
The first equivalence follows from \eqref{eq:boxtimes} and Lemma \ref{lem:Z-Deligne}. Indeed, the Frobenius algebra $L$ in $\cZ(\lmod{H})$ is the image of $\one_{\lmod{K}}\boxtimes L$ under the functor given there. The third equivalence holds by Proposition~\ref{prop:local-triv}. 

\smallskip 

(4) The first equation follows directly from part (3), and we have by Corollary~\ref{cor:locmod-FPdim} that 
$$\FPdim(\Rep_{\cZ(\lmod{H})}(L))
= \frac{[\FPdim(\cZ(\lmod{H}))]^{1/2}}{[\FPdim(\locmod_{\cZ(\lmod{H})}(L))]^{1/2} }  = \frac{\dim_\Bbbk H}{\dim_\Bbbk K} =  \dim_\Bbbk L.$$

\vspace{-.2in}
\end{proof}

\begin{remark}
Even though Proposition~\ref{prop:ex-charp} does not yield new non-semisimple modular tensor categories, part (3) of this result does illustrate interesting correspondences between modular categories (of local modules) in the non-semisimple setting. Compare to \cite{FFRS}*{(1.13)-(1.15) and Theorem~7.6}.
\end{remark}

\begin{example} \label{ex:charp}
Examples of finite-dimensional, commutative, non-semisimple Hopf algebras $K$ over~$\Bbbk$, with $\Drin(K)$ ribbon, include the following.
\begin{enumerate}
    \item One could take $K = \Bbbk G$, for $G$ a finite abelian group with $ |G|$ divisible by $p$.
    \smallskip
    \item One could also take $K = u^{[p]}(\mathfrak{g})$, the restricted enveloping algebra of an abelian restricted Lie algebra $\mathfrak{g}$. Namely, $\text{Tr(ad}(x))=0$ for all $x \in \mathfrak{g}$, so by \cite{Humphreys}*{Theorems~1 and~3},  $u^{[p]}(\mathfrak{g})$ is unimodular. On the other hand, $u^{[p]}(\mathfrak{g})$ is cocommutative, so $(u^{[p]}(\mathfrak{g}))^*$ is commutative and thus is unimodular. Now $\Drin(u^{[p]}(\mathfrak{g}))$ is ribbon by \cite{KR93}*{Theorem~3(b)}.
\end{enumerate}
We thank Cris Negron for discussions toward obtaining (2) above.
\end{example}

\begin{question} \label{ques:charp}
 Can one generalize Proposition~\ref{prop:ex-charp} by replacing $L = \Bbbk N$ with a more general commutative Hopf algebra $L$ over $\Bbbk$, or further, by removing the commutativity assumptions on the $\Bbbk$-Hopf algebras $K$ and $L$? 
\end{question}


\subsection{Rigid Frobenius algebras in \texorpdfstring{$\cZ(\lmod{\Bbbk G})$}{Z(Rep-loc(kG))} and their local modules}\label{sec:Z(kG-mod)}
In this subsection, let $G$ be a finite group, and let $\Bbbk$ be an algebraically closed field of arbitrary characteristic. Our goal here is to  upgrade Davydov's classification of rigid Frobenius algebras in $\cZ(\lmod{\Bbbk G})$ [Theorem~\ref{thm:davclassification}], and to study their categories of local modules [Corollary~\ref{cor:localmodules}].  We work over a field of arbitrary characteristic, whereas Davydov only considers the characteristic 0 setting; our arguments are packaged in the language of rigid Frobenius algebras instead of the equivalent notion of indecomposable étale algebras as done in Davydov (cf. Proposition~\ref{prop:conn-etale}). 

\smallskip
To start, note that we will employ the braided monoidal equivalence between $\cZ(\lmod{\Bbbk G})$ and $\lYD{\Bbbk G}$ often without mention; see \eqref{eq:Drin}. Next, recall the following notation from \cite{Dav3}.

\begin{notation}[$H,N,\gamma,\epsilon$] \label{not:cocycledata}
Let $H$ be a finite group, and $N\triangleleft H$ a normal subgroup with $|N| \in \Bbbk^\times$. 

Let $\gamma\in Z^2(N,\Bbbk^\times)$ be a normalized $2$-cocycle, that is:
\begin{align}\label{eq:cocycle}
    \gamma(n,m)\gamma(nm,r)&=\gamma(n,mr)\gamma(m,r), & \gamma(n,1_N)&=\gamma(1_N,n)=1_{\Bbbk}, &\forall n,m,r\in N.
\end{align}

Moreover, let 
$$\epsilon\colon H\times N\to \Bbbk^\times, \quad (h,n)\mapsto \epsilon_h(n)$$
be a map such that for all $g,h\in H$ and $n,m\in N$,
\begin{align}
    \epsilon_{gh}(n)&=\epsilon_g(hnh^{-1})\epsilon_h(n),  \label{eq:Davcond1}\\
    \epsilon_h(nm)\gamma(n,m)&=\epsilon_h(n)\epsilon_h(m)\gamma(hnh^{-1},hmh^{-1}), \label{eq:Davcond2}\\
    \gamma(n,m)&=\epsilon_n(m)\gamma(nmn^{-1},n).
    \label{eq:Davcond3}
\end{align}
In particular, the normalized condition on $\gamma$, along with \eqref{eq:Davcond2} and \eqref{eq:Davcond3} respectively, imply that
\begin{align}
    \epsilon_{h}(1_N) = 1_{\Bbbk}, \qquad \epsilon_{1_H}(n) = 1_{\Bbbk}. \label{eq:Davcond4}
\end{align}
\end{notation}

The rigid Frobenius algebras in $\cZ(\lmod{\Bbbk G})$, up to isomorphism, will depend on the data above along with a subgroup $H$ of $G$, and will be denoted by $A(H,N,\gamma,\epsilon)$ later [Theorem~\ref{thm:davclassification}]. Moreover, $A(H,N,\gamma,\epsilon)$ will be an inflation of the rigid Frobenius algebra $B(N,\gamma,\epsilon)$ in $\cZ(\lmod{\Bbbk H})$ constructed below, via the subsequent Remark~\ref{rem:laxmon}. 
Next, compare the next result to \cite{Dav3}*{Proposition~3.4.2}.

\begin{proposition}[$B(N,\gamma,\epsilon)$] \label{prop:B(N,g,e)} Recall Notation~\ref{not:cocycledata}.
\begin{enumerate}[font=\upshape]
    \item Consider the $\Bbbk$-vector space $B(N,\gamma, \epsilon)$ with $\Bbbk$-basis $\{ e_n \mid n\in N\}$, and with
    \begin{enumerate}[font=\upshape]
    \smallskip
        \item[(i)] left $\Bbbk H$-action given by $h \cdot e_n=\epsilon_h(n) e_{hnh^{-1}}$, for $h\in H$;
        \smallskip
        \item[(ii)] left $\Bbbk H$-coaction given by $\delta(e_n)=n\otimes e_n$, that is, $e_n$ is homogeneous of degree $n\in N$; 
        
        \smallskip
        
         \item[(iii)] multiplication $m_B$ given by $e_n e_m = \gamma(n,m) e_{nm}$ for all  $ n,m\in N$; \smallskip
         \item[(iv)] unit $u_B$ given by $u_B(1_\Bbbk):=1_B = 1_{\Bbbk} e_{1_N}$.\smallskip
    \end{enumerate}
    Then, $B(N,\gamma, \epsilon)$ is a connected, commutative algebra in $\cZ(\lmod{\Bbbk H}) \simeq \lYD{\Bbbk H}$. \medskip
    \item Further, $B(N,\gamma, \epsilon)$ is a rigid Frobenius algebra in $\cZ(\lmod{\Bbbk H})$ with 
    $$\Delta_B(e_n) = \sum_{m\in N}\frac{\gamma(m^{-1},n)}{\gamma(m^{-1},m)} \;e_{m}\otimes e_{m^{-1}n}, \qquad \varepsilon_B(e_n) = \delta_{n,1_N}1_\Bbbk$$
 for all $n\in N$. \smallskip
    \item Every rigid Frobenius algebra $B$ in $\cZ(\lmod{\Bbbk H})$ such that $\dim_\Bbbk B_1=1$ is isomorphic to one of the form $B(N,\gamma, \epsilon)$, for some choice of data $N, \gamma, \epsilon$.
\end{enumerate}
\end{proposition}

\begin{proof}
(1) The cocycle conditions \eqref{eq:cocycle} ensure that $B$ is a $\Bbbk$-algebra, and \eqref{eq:Davcond1} implies that $B$ is a $H$-module, while it is clear that $B$ is a Yetter--Drinfeld module over $\Bbbk H$ as $N$ is closed under conjugation. One checks that \eqref{eq:Davcond2}  ensures that the multiplication is a morphism in $\lYD{\Bbbk H}$. Moreover, $u_B\colon \one\to B$ is a morphism in $\lYD{\Bbbk H}$ by  \eqref{eq:Davcond4}. Thus, $B$ is an algebra in $\cZ(\lmod{\Bbbk H})$. By~\eqref{eq:Davcond3}, $B$ is commutative. Further, $B$ is connected since 
$$1 = \dim_\Bbbk \Bbbk 1_B \leq \dim_\Bbbk \Hom_{\lYD{\Bbbk H}}(\one, B)\leq \dim_\Bbbk \Hom_{\lcomod{\Bbbk H}}(\one, B) =  \dim_{\Bbbk}B^{H^*} =  \dim_{\Bbbk}B_1=1.$$ 

\smallskip

(2) It is straight-forward to check that $\Delta_B$ and $\varepsilon_B$ are morphisms in $\lYD{\Bbbk H}$. Now,  by \eqref{eq:cocycle}, $(B,\Delta_B, \varepsilon_B)$ is a Frobenius algebra in $\lYD{\Bbbk H}$. Finally, $m\Delta_B(e_n) = |N|e_n$, with $|N| \neq 0$, and $\varepsilon u(1_\Bbbk) = 1_\Bbbk$. So, $B$ is special Frobenius, as required.

\smallskip

(3) If $B$ is a rigid Frobenius algebra in $\cZ(\lmod{\Bbbk H})$ with $\dim_\Bbbk B_1=1$, then $B$ is indecomposable and étale by Proposition~\ref{prop:conn-etale}. Now proceeding as in \cite{Dav3}*{Proposition~3.4.2} yields the result.
\end{proof}

\begin{remark}\label{rem:laxmon}
Take a subgroup $H$  of $G$. Then the restriction functor $F\colon \lmod{\Bbbk G}\to \lmod{\Bbbk H}$ has a biadjoint functor given by induction $R\colon \lmod{\Bbbk H}\to \lmod{\Bbbk G}$,
which is faithful and exact.
In particular, the induction functor $R$ is lax monoidal with

\begin{align*}
R_{V,W}\colon R(V)\otimes R(W)&\to R(V\otimes W),\\ (g\otimes v)\otimes(k\otimes w) &\mapsto 
\begin{cases}
k\otimes (k^{-1}g v\otimes w), &\text{if $gH=kH$}\\
0, & \text{else,}
\end{cases}
\end{align*}
for $\Bbbk H$-modules $V,W$ 
 (cf. \cite{FHL}*{Appendix~B} for the formulas). The unit morphism of this lax monoidal structure is $R_0\colon \Bbbk \to R(\one), 1_\Bbbk \mapsto \sum_{i\in I}g_i$, where $G=\bigsqcup_{i\in I}g_i H$ is a left coset decomposition. In fact, say by \cite{FHL}*{Proposition~B.1}, $R$ induces a lax monoidal functor $$\widetilde{R}:\cZ(\lmod{\Bbbk H})\to \cZ(\lmod{\Bbbk G}).$$ 
\end{remark}

This brings us to the main result of the section, which is \cite{Dav3}*{Theorem~3.5.1} revised in the language of rigid Frobenius algebras. Its proof uses a different approach than in \cite{Dav3} by way of Remark~\ref{rem:laxmon}, along with Proposition~\ref{prop:B(N,g,e)}.

\begin{theorem}[$A(H,N,\gamma,\epsilon)$] \label{thm:davclassification} Recall Notation~\ref{not:cocycledata}, take a subgroup $H$ of $G$ with $|G:H| \in \Bbbk^{\times}$. Fix a coset decomposition $G=\bigsqcup_{i\in I} g_i H$.
\begin{enumerate}[font=\upshape]
    \item Let $A(H,N,\gamma, \epsilon)$ be the quotient $\Bbbk$-vector space spanned by $\{ a_{g,n} \mid g \in G, n\in N\}$, subject to the relations:
    \begin{align}
    a_{gh,n}=\epsilon_h(n)a_{g,hnh^{-1}}, & \; \; \forall h\in H,\label{eq:Davquotrel}
\end{align}
    and with
    \begin{enumerate}[font=\upshape]
    \smallskip
        \item[(i)] left $\Bbbk G$-action given by  $k\cdot a_{g,n}=a_{kg,n}$, for $k\in G$;
        \smallskip
        \item[(ii)] left $\Bbbk G$-coaction given by $\delta(a_{g,n})=g ng^{-1}\otimes a_{g,n}$; 
        \smallskip
         \item[(iii)] multiplication $m_A$ given by $$a_{g,n}a_{k,m}=\delta_{gH,kH} \;\gamma(k^{-1}gng^{-1}k,m)\;\epsilon_{k^{-1}g}(n)\;a_{k,k^{-1}gng^{-1}km},$$ for $g,k\in G$ and $n,m\in N$ (note that $k^{-1} g \in H$ if and only if $gH = kH$); \smallskip
         \item[(iv)] unit $u_A$ given by $u_A(1_\Bbbk) := 1_A =\sum_{i\in I}a_{g_i,1_N}$.\smallskip
    \end{enumerate}
    Then, $A(H,N,\gamma, \epsilon)$ is a connected, commutative algebra in $\cZ(\lmod{\Bbbk G}) \simeq \lYD{\Bbbk G}$. \medskip
    \item Further, $A(H,N,\gamma, \epsilon)$ is a rigid Frobenius algebra in $\cZ(\lmod{\Bbbk G})$ with 
    $$\Delta_A(a_{g,n}) = \sum_{m\in N}\frac{\gamma(m^{-1},n)}{\gamma(m^{-1},m)} \;a_{g,m}\otimes a_{g,m^{-1}n}, \qquad \varepsilon_A(a_{g,n}) = \delta_{n,1_N}1_\Bbbk,$$
 for all $g\in G$ and $n\in N$. \smallskip

    \item Every rigid Frobenius algebra  in $\cZ(\lmod{\Bbbk G})$  is isomorphic to one of the form $A(H,N,\gamma, \epsilon)$, for some choice of data $H,N, \gamma, \epsilon$.
\end{enumerate}
\end{theorem}

\begin{remark} \label{rem:correction}
Note that the formulas in \eqref{eq:Davquotrel} and part (1.iii) are a correction of those given in \cite{Dav3}*{Theorem~3.5.1}, which do not allow their version of the relations \eqref{eq:Davquotrel} to form an ideal. The revised multiplication and unit formulas are derived from Remark~\ref{rem:laxmon}.

Moreover, note that the isomorphism classes of $A(H,N,\gamma, \varepsilon)$ only depend on the cohomology classes of $\gamma, \varepsilon$; cf.  
\cite{Dav3}*{Lemma 3.4.3}.
\end{remark}

\begin{proof}[Proof of Theorem~\ref{thm:davclassification}] (1) Given Remark~\ref{rem:correction}, we will verify that the relations \eqref{eq:Davquotrel} form an ideal. This holds by the following calculations:
{\small
\begin{align*}
a_{k,m} (a_{gh,n}) &= \delta_{kH,ghH} \;\gamma(h^{-1}g^{-1}kmk^{-1}gh,n)\;
\epsilon_{h^{-1}g^{-1}k}(m)\;
a_{gh,h^{-1}g^{-1}kmgk^{-1}ghn}\\
&\overset{\eqref{eq:Davquotrel}}{=} \delta_{kH,gH} \;
\gamma(h^{-1}g^{-1}kmk^{-1}gh,n)\;
\epsilon_{h^{-1}g^{-1}k}(m)\;
\epsilon_h(h^{-1}g^{-1}kmk^{-1}ghn)\;
a_{g,g^{-1}kmgk^{-1}ghnh^{-1}}\\
&\overset{\eqref{eq:Davcond2}}{=} \delta_{kH,gH} \;
\epsilon_h(h^{-1}g^{-1}kmk^{-1}gh)\;
\epsilon_h(n)\;
\gamma(g^{-1}kmk^{-1}g,hnh^{-1})\;
\epsilon_{h^{-1}g^{-1}k}(m)\;
a_{g,g^{-1}kmgk^{-1}ghnh^{-1}}\\
&\overset{\eqref{eq:Davcond1}}{=} \delta_{kH,gH} \;
\epsilon_{g^{-1}k}(m)\;
\epsilon_h(n)\;
\gamma(g^{-1}kmk^{-1}g,hnh^{-1})\;
a_{g,g^{-1}kmgk^{-1}ghnh^{-1}}\\
&= a_{k,m}\;(\epsilon_h(n)\;a_{g,hnh^{-1}});
\end{align*}
}
\[
{\small
\begin{aligned}
(a_{gh,n})a_{k,m} &= \delta_{gH,kH}\; \gamma(k^{-1}ghnh^{-1}g^{-1}k,m) \; \epsilon_{k^{-1}gh}(n) \; a_{k,k^{-1}ghnh^{-1}g^{-1}km}\\
&\overset{\eqref{eq:Davcond1}}{=} \epsilon_h(n) \; \delta_{gH,kH} \; \gamma(k^{-1}ghnh^{-1}g^{-1}k,m)\; \epsilon_{k^{-1}g}(hnh^{-1}) \; a_{k,k^{-1}ghnh^{-1}g^{-1}km}\\
&= (\epsilon_h(n) \; a_{g,hnh^{-1}})\;a_{k,m}.
\end{aligned}
}
\]
The rest of the proof follows similarly to the proof of Proposition~\ref{prop:B(N,g,e)}(1). In particular, note that
$$1 = \dim_\Bbbk \Bbbk 1_A \leq \dim_\Bbbk \Hom_{\lYD{\Bbbk G}}(\one, A)\leq \dim_\Bbbk (\Hom_{\lmod{\Bbbk G}}(\one, A) \cap \Hom_{\lcomod{\Bbbk G}}(\one, A)).$$ 
Here, for scalars $\lambda_{g,n} \in \Bbbk$, we have
\[
\begin{aligned}
\Hom_{\lmod{\Bbbk G}}(\one, A) \cong A^{G} &= \{\textstyle \sum_{g,n}\lambda_{g,n}\;a_{g,n} ~|~ \sum_{g,n}\lambda_{g,n}\; a_{kg,n} = \sum_{g,n}\lambda_{g,n}\;a_{g,n} \; \forall k \in G \}\\
&= \{\textstyle\sum_{g}\lambda_{n}\; a_{g,n}  ~|~ n \in N\},\\
\Hom_{\lcomod{\Bbbk G}}(\one, A) \cong A^{G^*} = A_1 &= \{a_{g,n} ~|~ \delta(a_{g,n}) = 1_G \otimes a_{g,n} \} = \{a_{g,1_N} ~|~ g \in G\}.
\end{aligned}
\]
Therefore, $A^G \cap A^{G^*}$ is the 1-dimensional vector space with basis $\sum_{g \in G} a_{g, 1_N}$. So, $A$ is connected.
\smallskip

(2) This follows similarly to Proposition~\ref{prop:B(N,g,e)}(2). For instance, the Frobenius  compatibility condition \eqref{eq:Frob-comp} between $m_A$ and $\Delta_A$ holds as follows:
{\small
\begin{align*}
\Delta_A m_A(a_{g,n} \otimes a_{k,m}) &= \delta_{gH,kH} \;\gamma(k^{-1}gng^{-1}k,m)\;\epsilon_{k^{-1}g}(n)\; \Delta_A(a_{k,k^{-1}gng^{-1}km})\\
&= \delta_{gH,kH} \;\gamma(k^{-1}gng^{-1}k,m)\;\epsilon_{k^{-1}g}(n)\; \\
& \hspace{.2in} \cdot \sum_{p \in N} \frac{\gamma(p^{-1},k^{-1}gng^{-1}km)}{\gamma(p^{-1},p)} \;a_{k,p} \otimes a_{k,p^{-1}k^{-1}gng^{-1}km}\\
&\overset{p=k^{-1}gqg^{-1}k}{=} \delta_{gH,kH} \;\gamma(k^{-1}gng^{-1}k,m)\;\epsilon_{k^{-1}g}(n)\; \\
& \hspace{.2in} \cdot \sum_{q \in N} \frac{\gamma(k^{-1}gq^{-1}g^{-1}k,\;k^{-1}gng^{-1}km)}{\gamma(k^{-1}gq^{-1}g^{-1}k,\;k^{-1}gqg^{-1}k)} \;a_{k,k^{-1}gqg^{-1}k} \otimes a_{k,k^{-1}gq^{-1}ng^{-1}km}\\
&\overset{\eqref{eq:Davquotrel}}{=} \delta_{gH,kH} \;\gamma(k^{-1}gng^{-1}k,m)\; \\
& \hspace{.2in} \cdot \sum_{q \in N} \frac{\gamma(k^{-1}gq^{-1}g^{-1}k,\;k^{-1}gng^{-1}km)}{\gamma(k^{-1}gq^{-1}g^{-1}k,\;k^{-1}gqg^{-1}k)} \;a_{g,q} \otimes a_{k,k^{-1}gq^{-1}ng^{-1}km}\\
&\overset{\eqref{eq:cocycle}}{=} \delta_{gH,kH} \; \sum_{q \in N} \frac{\gamma(q^{-1},n)}{\gamma(q^{-1},q)}\;\gamma(k^{-1}gq^{-1}ng^{-1}k, m)\; \epsilon_{k^{-1}g}(q^{-1}n)\; a_{g,q} \otimes   a_{k,k^{-1}gq^{-1}ng^{-1}km}\\
&=\sum_{q \in N} \frac{\gamma(q^{-1},n)}{\gamma(q^{-1},q)}\;a_{g,q} \otimes (\delta_{gH,kH} \; \gamma(k^{-1}gq^{-1}ng^{-1}k, m)\; \epsilon_{k^{-1}g}(q^{-1}n)\; a_{k,k^{-1}gq^{-1}ng^{-1}km})\\
&=\sum_{q \in N} \frac{\gamma(q^{-1},n)}{\gamma(q^{-1},q)}\;a_{g,q} \otimes m_A(a_{g,q^{-1}n} \otimes a_{k,m})\\
&=(\ide \otimes m_A)(\Delta_A \otimes \ide)(a_{g,n} \otimes a_{k,m}).
\end{align*}
}
Likewise, $\Delta_A m_A = (m_A \otimes \ide)(\ide \otimes \Delta_A)$.
Moreover, the hypothesis that $|N|, |G:H| \in \Bbbk^\times$ are needed for the special condition because of the computations below:
\begin{align*}
m_A \Delta_A(a_{g,n})&= \sum_{m \in N} \frac{\gamma(m^{-1},n)}{\gamma(m^{-1},m)}\;m_A(a_{g,m} \otimes a_{g,m^{-1}n})\\
&= \sum_{m \in N} \frac{\gamma(m^{-1},n)}{\gamma(m^{-1},m)} \; \delta_{gH,gH} \; \gamma(g^{-1}gmg^{-1}g, m^{-1}n) \; \epsilon_{g^{-1}g}(m) \; a_{g,g^{-1}gmg^{-1}gm^{-1}n}\\
&\overset{\eqref{eq:cocycle},\eqref{eq:Davcond4}}{=} \textstyle \sum_{m \in N} \gamma(1_N,n)  a_{g,n}  \\  
&\overset{\eqref{eq:cocycle}}{=} |N| a_{g,n}, \qquad \text{and} \\ 
\end{align*}

\vspace{-.4in}

$$ \hspace{-1in} \varepsilon_A u_A(1_\Bbbk) =  \textstyle \sum_{i \in I} \varepsilon_A(a_{g_i,1_N}) 1_\Bbbk\; =\; \textstyle \sum_{i \in I} \delta_{1_N,1_N} 1_\Bbbk \;=\; |G:H|\; 1_\Bbbk.$$

(3) It suffices to show that any rigid Frobenius algebra $A$ in $\cZ(\lmod{\Bbbk G})$ is of the form $\widetilde{R}(B)$, for the monoidal functor $\widetilde{R}$ in Remark~\ref{rem:laxmon}, with $B$ a rigid Frobenius algebra in $\cZ(\lmod{\Bbbk H})$. This is due to the definition of $A(H,N,\gamma,\epsilon)$ and Proposition~\ref{prop:B(N,g,e)}(3). By Remark~\ref{rem:laxmon} and \cite{BN11}*{Proposition~6.1}, the functor $R$ induces a monoidal equivalence $R:\lmod{\Bbbk H} \overset{\sim}{\longrightarrow} \Rep_{\lmod{\Bbbk G}}(R(\one))$, where $R(\one)$ is a central algebra in $\cZ(\lmod{\Bbbk G})$. Here, $R(\one) \cong \Bbbk(G/H)$. Therefore, with \cite{Sch}*{Corollary~4.5}, this equivalence induces an equivalence of braided monoidal categories 
\begin{equation} \label{eq:Rep-loc-GH}
\cZ(\lmod{\Bbbk H})\isomorph \locmod_{\cZ(\lmod{\Bbbk G})}(\Bbbk(G/H));
\end{equation}
cf. \cite{Dav3}*{Theorem~3.3.2}. On the other hand, $A$ can be identified as a rigid Frobenius algebra in $\locmod_{\cZ(\lmod{\Bbbk G})}(A_1)$ by \cite{Dav3}*{Corollary~3.3.5}. But $A_1 \cong \Bbbk(G/H)$ by \cite{KO}*{2.2. Theorem}. Now $A = \widetilde{R}(B)$ for some commutative algebra $B$ in $\cZ(\lmod{\Bbbk H})$. Since $A$ is connected, so is $B$. Lastly, the structure maps on $A$ restrict to a special Frobenius structure on $B$.
\end{proof}

\begin{corollary} \label{cor:localmodules}
For $A:=A(H,N,\gamma,\epsilon)$ defined in Theorem~\ref{thm:davclassification}, the following statements hold. \smallskip
\begin{enumerate}[font=\upshape]
    \item The categories $\locmod_{\cZ(\lmod{\Bbbk G})}(A(H,N,\gamma,\epsilon))$ are (non-semisimple) modular categories for any choice of data $(H,N,\gamma,\epsilon)$ (respectively, when $\cha \Bbbk$ divides $|G|$).\smallskip
    \item We have the dimensions below: 
    \begin{align*}
    \dim_\Bbbk A &= \frac{|G||N|}{|H|},&
    \FPdim\left(\Rep_{\cZ(\lmod{\Bbbk G})}(A)\right)&=\frac{|G||H|}{|N|},&
    \FPdim\left(\locmod_{\cZ(\lmod{\Bbbk G})}(A)\right)&=\frac{|H|^2}{|N|^2}. \smallskip
\end{align*}
\item $A$ is trivializing, i.e. $\locmod_{\cZ(\lmod{\Bbbk G})}(A)\simeq \Vect$ as ribbon categories, if and only if $N=H$. \smallskip

\item  If $N=\{1\}$, then  $\locmod_{\cZ(\lmod{\Bbbk G})}(A)\simeq \cZ(\lmod{\Bbbk H})$  as ribbon categories.
\end{enumerate}
\end{corollary}

\begin{proof}
(1) This follows from Theorem~\ref{thm:locmodular}. \smallskip

(2) This follows from  Corollary \ref{cor:locmod-FPdim}. Namely, $\{a_{g_i,n}\mid i\in I, n\in N\}$ form a $\Bbbk$-basis for $A(H,N, \gamma,\epsilon)$, and $\FPdim \left(\cZ(\lmod{\Bbbk G})\right)=|G|^2$ (see \cite{EGNO}*{Theorem~7.16.6.}).

\smallskip

(3) This now follows part (2) since if $N=H$, 
$\FPdim \left( \locmod_{\cZ(\lmod{\Bbbk G})}(A) \right)=1$. The inclusion functor $I\colon \Vect\to \locmod_{\cZ(\lmod{\Bbbk G})}(A)$ is an injective tensor functor of finite tensor categories of equal FP-dimension, and hence an equivalence \cite{EGNO}*{Proposition 6.3.3}. The functor $I$ is a ribbon functor since $\theta_\one=\ide_{\one}$.
 \smallskip
 
(4) 
It follows from the presentation of $A$ that $a_{g_i, 1_N}$, for $g_i$ a set of representatives of left $H$-cosets form a basis for $A$, the $G$-grading is trivial and the left $\Bbbk G$-action is corresponds to the action on the left cosets $\{g_iH\}$. This provides an isomorphism between $A$ and the function algebra $\Bbbk (G/H)$. The statement follows from \eqref{eq:Rep-loc-GH}; see also  \cite{Dav3}*{Theorem~3.3.2}.
\end{proof}

In the case when $\cha \Bbbk$ does not divide $|G|$, part (1) is a special case of \cite{KO}*{Theorem~4.5} and  part (3) was originally proved in \cite{Dav3}*{Theorem~3.5.3} in this case. Note that by parts~(3) and~(4) above, interesting examples of modular categories can only arise if $N\neq H$ and $N\neq \{1\}$. This prompts the  question below.

\begin{question} \label{ques:Davrib}
Are the modular categories
 $\locmod_{\cZ(\lmod{\Bbbk G})}(A(H,N,\gamma,\epsilon))$ and $\cZ(\lmod{\Bbbk H/N})$ equivalent as ribbon categories?
\end{question}

This question was addressed in the affirmative in \cite{DS}*{Theorem~3.13} in the case when $\Bbbk$ has characteristic zero.
The following provides a small example of the setting above in the non-semisimple case.

\begin{example}
Let $\Bbbk$ be a field of characteristic 3, and take the groups $G=S_4$, $H=A_4$, and $N=C_2\times C_2=\langle(12)(34),(13)(24)\rangle$. Here, 3 divides $|G|$, but 3 does not divide $|N|$ nor $|G|/|H|$ as required. We obtain $2$-cocycles $\gamma$ on $N$ by a choice of an element in $\{\pm 1\}\times \{\pm 1\}$ and may choose $\epsilon$ to be trivial. The resulting  rigid Frobenius algebra $A=A(H,N,\gamma,\epsilon)$ is $8$-dimensional and decomposes as $\Bbbk N\times \Bbbk N$ as a $\Bbbk$-algebra, corresponding to an idempotent decomposition $1_A=e_1+e_2$, with $e_1=a_{1,1}$ and $e_2=a_{(12),1}$. The action of $(12)\in S_4$ permutes these idempotents. Further, 
\begin{align*}
    A_1&=\Bbbk \langle e_1,\;e_2\rangle, & A_{a}=\Bbbk \langle a_{1,a},\;a_{(12),a} \rangle, \quad 
    & A_{b}=\Bbbk \langle a_{1,b},\;a_{(12),ab} \rangle, \quad & A_{ab}=\Bbbk \langle a_{1,ab},\;a_{(12),b} \rangle, 
\end{align*}
with $a=(12)(34)$ and $b=(13)(24)$.
\end{example}




\subsection{On completely anisotropic categories and Witt equivalence} 
\label{sec:ques}
In work of Davydov-M\"{u}ger-Nikshych-Ostrik in the semisimple setting, a non-degenerate braided fusion category $\cC$ is said to be {\it completely  anisotropic} if the only connected \'{e}tale algebra $A$ in  $\cC$ is $A = \one$ \cite{DMNO}*{Definition~5.10}. By a {\it fusion} category, we mean a finite tensor category that is semisimple. In this case, techniques of using categories of local modules to produce {\it new} modular fusion categories cannot be applied. In any case, one of the main results in their work is on the prevalence of completely anisotropic categories. Two non-degenerate braided fusion categories $\cC_1$ and $\cC_2$ are said to be {\it Witt equivalent} if there exists a braided equivalence, $\cC_1 \boxtimes \cZ(\cA_1) \simeq \cC_2 \boxtimes \cZ(\cA_2)$, for some fusion categories $\cA_1, \cA_2$. Then it was shown that each Witt equivalence class contains a completely anisotropic category, that is unique up to braided equivalence \cite{DMNO}*{Theorem~5.13}.

\smallskip

Likewise, given Proposition~\ref{prop:conn-etale}, we set the following terminology. 

\begin{definition}
A non-degenerate braided finite tensor category $\cC$ is said to be {\it completely  anisotropic} if the only rigid Frobenius algebra $A$ in  $\cC$ is $A = \one$.
\end{definition}

After extensive experimentation, we pose the following conjecture.

\begin{conjecture} \label{conj:uqsl2}
Let $\Bbbk$ be an algebraically closed field of characteristic 0. For the small quantum group $u_q(\mathfrak{sl}_2)$, for $q$ an odd root of unity, the non-semisimple modular tensor category $\lmod{u_q(\mathfrak{sl}_2)}$  is completely anisotropic.
\end{conjecture}

See, e.g., \cite{LW2}*{Example~5.3} for a discussion of the modularity of $\lmod{u_q(\mathfrak{sl}_2)}$. In particular, we inquire in \cite{LW2}*{Question~4.21} if this category is {\it prime} in the sense that every topologizing non-degenerate
braided tensor subcategory is equivalent to either itself or $\Vect$.

\smallskip

Moreover, consider the following terminology generalizing the semisimple notion above.

\begin{definition}
Two non-degenerate braided finite tensor categories $\cC_1$ and $\cC_2$ are said to be {\it Witt equivalent} if there exists a braided equivalence, $\cC_1 \boxtimes \cZ(\cA_1) \simeq \cC_2 \boxtimes \cZ(\cA_2)$, for some finite tensor categories $\cA_1, \cA_2$. 
\end{definition}

Naturally, we inquire: 

\begin{question} \label{ques:Witt}
Working over an algebraically closed field $\Bbbk$ of characteristic 0, which (modular) non-degenerate  braided finite tensor categories are Witt equivalent to $\lmod{u_q(\mathfrak{sl}_2)}$?
\end{question}

Moreover, after discussions with Victor Ostrik and Christoph Schweigert, we pose the following questions about general non-semisimple modular tensor categories in characteristic 0.

\begin{question} \label{ques:nonsem-rigidFrob}
Let $\Bbbk$ be an algebraically closed field of characteristic 0. 
\begin{enumerate}
    \item Must a non-semisimple, prime, modular tensor category over $\Bbbk$ be completely anisotropic? (See, e.g., \cite{LW2}*{Corollary~4.20}.) \smallskip
    \item Does there exist an example of a non-semisimple modular tensor category over $\Bbbk$ that is not Witt equivalent to any semisimple modular tensor category?
\end{enumerate}
\end{question}

Lastly, prompted by feedback from the anonymous referee, we have the following generalization of \cite{DMNO}*{Proposition~5.4} to the non-semisimple setting.

\begin{proposition}
Let $\cC$ be a non-degenerate braided finite tensor category and $A$ a rigid Frobenius algebra in $\cC$. Then we obtain that $\locmod_{\cC}(A)$ is Witt equivalent to $\cC$.
\end{proposition}

\begin{proof}
Since $\cC$ is a non-degenerate braided finite tensor category, then it is factorizable by \cite{Shi1}*{Theorem~1.1}. Namely, $\cZ(\cC)$ is equivalent to $\cC \boxtimes \overline{\cC}$ as braided tensor categories. Now by Lemma~\ref{lem:A+local} and Theorem~\ref{thm:local-center},  $\cZ(\Rep_\cC(A)) \; \overset{br.\otimes}{\simeq}\; \locmod_{\cC}(A) \boxtimes \overline{\cC}$. Thus, we have
\[
\locmod_{\cC}(A) \boxtimes \cZ(\cC) \; \overset{br.\otimes}{\simeq}\;
 \cC  \boxtimes \cZ(\Rep_\cC(A)),
\]
which yields the result.
\end{proof}

\medskip


\section*{Declarations}

Data Availability Statement: Data sharing not applicable to this article as no datasets were generated or analysed during the current study.

\medskip

Competing Interests Statement: The authors have no competing interests to declare that are relevant to the content of this article.


\section*{Acknowledgements} 
The authors are very grateful to the anonymous referees for their careful consideration  and insightful feedback, which enabled us to improve the quality of our manuscript. The authors thank Alexei Davydov for encouraging this project and helpful insights. We also thank Harshit Yadav for pointing out a correction to Lemma~\ref{lemma:rigid}, and Cris Negron for discussions for the material in Example~\ref{ex:charp}. We have also benefited from discussions with  Victor Ostrik and Christoph Schweigert, which led to Question~\ref{ques:nonsem-rigidFrob}.
R.~Laugwitz was partially supported by an AMS-Simons travel grant and supported by a Nottingham Research Fellowship. C.~Walton was partially supported by the US National Science Foundation grant \#DMS-2100756 and by the Alexander von Humboldt foundation.


\bibliography{lm-qea-bib}
\bibliographystyle{amsrefs}

\end{document}